\documentclass{amsart}
\usepackage[utf8]{inputenc}

\usepackage[margin=1in]{geometry}
\usepackage[expansion=false]{microtype}

\usepackage{amsmath, amsfonts, amssymb, amsthm, amsopn, amscd}
\usepackage{eucal}

\usepackage{booktabs,enumitem}
\setlist[enumerate]{label=(\roman*),font=\normalfont}

\usepackage{tikz}
\usetikzlibrary{matrix}
\usepackage{tikz-cd}

\usepackage[pdfusetitle,unicode,bookmarksopen]{hyperref}

\numberwithin{equation}{section}

\theoremstyle{plain}
\newtheorem*{theorem*}{Theorem}
\newtheorem{theorem}[equation]{Theorem}
\newtheorem{proposition}[equation]{Proposition}
\newtheorem{lemma}[equation]{Lemma}
\newtheorem{corollary}[equation]{Corollary}

\theoremstyle{definition}
\newtheorem{definition}[equation]{Definition}
\newtheorem{construction}[equation]{Construction}
\newtheorem{example}[equation]{Example}

\newtheorem{variant}[equation]{Variant}
\newtheorem{summary}[equation]{Summary}

\newtheorem{remark}[equation]{Remark}

\let\scr=\mathcal
\let\bb=\mathbb
\let\rm=\mathrm
\def\N{\bb N}
\def\Z{\bb Z}

\def\F{\bb F}
\def\A{\bb A}
\def\P{\bb P}
\def\V{\bb V}
\def\1{\mathbf 1}
\def\h{\mathrm h}
\def\t{\mathrm t}
\def\q{\mathrm q}
\def\s{\mathrm s}
\def\g{\mathrm g}
\def\c{\mathrm c}
\def\C{\mathrm C}
\def\D{\mathrm D}

\def\G{\mathbb G}
\def\ph{\mathord-}
\makeatletter
\def\pt{{\mathpalette\pt@{.75}}}
\def\pt@#1#2{\mathord{\scalebox{#2}{$\m@th#1\bullet$}}}
\makeatother

\def\loc{\mathrm{loc}}

\let\from=\leftarrow
\let\into=\hookrightarrow
\let\onto=\twoheadrightarrow
\def\simto{\xrightarrow{\sim}}
\def\simfrom{\xleftarrow{\sim}}

\let\emptyset=\varnothing

\DeclareMathOperator{\Sym}{Sym}

\def\id{\mathrm{id}}
\def\th{\mathrm{th}}
\def\br{\mathrm{br}}

\def\Hom{\mathrm{Hom}}

\def\Map{\mathrm{Map}}

\def\map{\mathrm{map}}

\DeclareMathOperator{\Spec}{Spec}

\def\Th{\mathrm{Th}}
\def\fTh{\mathrm{Th}^f}
\def\Pic{\mathrm{Pic}{}}
\def\Br{\mathrm{Br}{}}
\def\K{\mathrm{K}{}}
\def\GW{\mathrm{GW}{}}
\def\L{\mathrm{L}{}}

\def\B{\mathrm{B}}

\let\sect=\S
\def\S{\mathrm{S}}
\def\E{\mathbb{E}}

\DeclareMathOperator{\fib}{fib}
\DeclareMathOperator{\cofib}{cofib}

\DeclareMathOperator{\rk}{rk}
\def\Nis{\mathrm{Nis}}
\def\Zar{\mathrm{Zar}}

\def\hyp{\mathrm{hyp}}

\def\QCoh{\mathrm{QCoh}{}}

\def\Bl{\mathrm{Bl}}

\let\cat=\mathrm

\def\MGL{\mathrm{MGL}}

\def\PMGL{\mathrm{PMGL}}
\def\KGL{\mathrm{KGL}}
\def\KO{\mathrm{KO}}
\def\KW{\mathrm{KW}}
\def\KSp{\mathrm{KSp}}

\def\MSL{\mathrm{MSL}}
\def\MML{\mathrm{MML}}
\def\MSp{\mathrm{MSp}}
\def\SL{\mathrm{SL}}
\def\ML{\mathrm{ML}}

\def\Mod{\cat{M}\mathrm{od}{}}
\def\Sch{\mathrm{Sch}{}}

\def\dSpc{\mathrm{dSpc}{}}
\def\SpSpc{\mathrm{SpSpc}{}}

\def\Sm{{\cat{S}\mathrm{m}}}

\def\Fin{\cat F\mathrm{in}}
\def\op{\mathrm{op}}

\def\Vect{\mathrm{Vect}{}}

\def\MS{\mathrm{MS}{}}
\def\An{\mathrm{An}{}}
\def\Sp{\mathrm{Sp}{}}

\def\Pr{\mathrm{Pr}{}}

\def\CAlg{\mathrm{CAlg}{}}
\def\Poly{\mathrm{Poly}{}}

\def\Ind{\mathrm{Ind}{}}

\def\sm{\mathrm{sm}}
\def\an{\mathrm{an}}

\def\Cat{\mathrm{C}\mathrm{at}{}}

\def\GL{\mathrm{GL}}

\def\fp{\mathrm{fp}}
\def\qcqs{\mathrm{qcqs}}
\def\Fun{\mathrm{Fun}}

\def\st{\mathrm{st}}
\def\p{\mathrm{p}}
\def\b{\mathrm{b}}

\def\Span{\mathrm{Span}{}}

\def\Perf{\mathrm{Perf}{}}
\def\QStk{\mathrm{QStk}{}}

\def\gp{\mathrm{gp}}
\def\sym{\mathrm{sym}}
\def\alt{\mathrm{alt}}
\def\qu{\mathrm{quad}}
\def\m{\mathrm{m}}

\def\epi{\mathrm{epi}}

\def\sbu{\mathrm{sbu}}
\def\ebu{\mathrm{ebu}}
\def\Tate{\scr T}

\newcommand{\lax}{\mathrm{lax}}
\newcommand{\idem}{\mathrm{idem}}

\let\heart=\heartsuit
\def\Hyp{\mathrm{Hyp}{}}
\def\Hlgy{\mathrm{Hlgy}{}}

\let\lim=\relax
\DeclareMathOperator*{\lim}{lim}
\DeclareMathOperator*{\colim}{colim}

\let\phi=\varphi
\let\epsilon=\varepsilon

\usepackage[bbgreekl]{mathbbol}
\DeclareSymbolFontAlphabet{\mathbb}{AMSb} 
\DeclareSymbolFontAlphabet{\mathbbl}{bbold}

\DeclareFontFamily{T1}{cbgreek}{}
\DeclareFontShape{T1}{cbgreek}{m}{n}{<-6>  grmn0500 <6-7> grmn0600 <7-8> grmn0700 <8-9> grmn0800 <9-10> grmn0900 <10-12> grmn1000 <12-17> grmn1200 <17-> grmn1728}{}
\DeclareSymbolFont{quadratics}{T1}{cbgreek}{m}{n}
\DeclareMathSymbol{\rawqoppa}{\mathord}{quadratics}{19}
\DeclareMathSymbol{\Qoppa}{\mathord}{quadratics}{21}

\def\qoppa{%
{\mathchoice{\raisebox{-1.5pt}{$\displaystyle\rawqoppa$}}
{\raisebox{-1.5pt}{$\rawqoppa$}}
{\raisebox{-0.9pt}{$\scriptstyle\rawqoppa$}}
{\raisebox{-0.8pt}{$\scriptscriptstyle\rawqoppa$}}}}

\title{\texorpdfstring{Grothendieck--Witt}{Grothendieck–Witt} theory of derived schemes}
\date{\today}

\author{Marc Hoyois}
\address{Fakultät für Mathematik\\
Universität Regensburg\\
Universitätsstraße 31\\
93040 Regensburg\\
Germany}
\email{\href{mailto:marc.hoyois@ur.de}{marc.hoyois@ur.de}}
\urladdr{\url{https://hoyois.app.uni-regensburg.de}}
\thanks{M.H.\ was partially supported by the Collaborative Research Center SFB 1085 \emph{Higher Invariants} funded by the DFG}

\author{Markus Land}
\address{Mathematisches Institut\\
Ludwig-Maximilians-Universität München\\
Theresienstraße 39\\
80333 München\\
Germany}
\email{\href{mailto:markus.land@math.lmu.de}{markus.land@math.lmu.de}}
\urladdr{\url{https://www.markus-land.de}}

\begin{document}
	
\maketitle

\begin{abstract}
	We construct a non-$\A^1$-invariant motivic ring spectrum $\KO$ over $\Spec(\Z)$, whose associated cohomology theory on qcqs derived schemes is the Grothendieck–Witt theory of classical symmetric forms (as opposed to homotopy symmetric forms). In particular, we show that this theory satisfies Nisnevich descent, smooth blowup excision, a projective bundle formula, and is locally left Kan extended from smooth $\Z$-schemes up to Bass delooping. More generally, our construction produces $\KO$-modules representing localizing invariants of two different families of Poincaré structures on derived schemes, which we call ``classical'' and ``genuine''; the latter Poincaré structures are defined for spectral schemes with involution, but the former only for derived schemes.
	
	We then establish basic properties of these motivic spectra.
	As in $\A^1$-homotopy theory, the fracture square of $\KO$ with respect to the Hopf element recovers the fundamental cartesian square relating GW-theory, L-theory, and K-theory. A new phenomenon when $2$ is not a unit is that $\KO$ is not Bott-periodic, and the left and right Bott periodizations of $\KO$ represent the Grothendieck–Witt theories of homotopy symmetric and homotopy quadratic forms, respectively.
	We also construct the expected metalinear $\E_\infty$-orientation of $\KO$.
	Finally, we show that the $\A^1$-localization of $\KO$ recovers the motivic spectrum recently constructed by Calmès, Harpaz, and Nardin.
\end{abstract}

\tableofcontents

\section{Introduction}

The goal of this paper is to lay the foundations for the Grothendieck–Witt theory of derived schemes.
In particular, we establish the following ``motivic'' properties of Grothendieck–Witt theory: it satisfies Nisnevich descent, smooth blowup excision, and a projective bundle formula. These properties imply that Grothendieck–Witt theory is represented by a motivic spectrum in the sense of Annala, Hoyois, and Iwasa \cite{AnnalaIwasa2,AHI,atiyah}. 
Note however that it is not represented by a motivic spectrum in the sense of Morel and Voevodsky, as it is not $\A^1$-invariant even on regular schemes.

We use the framework of Poincaré categories and localizing invariants thereof developed by Calmès, Dotto, Harpaz, Hebestreit, Land, Moi, Nardin, Nikolaus, and Steimle in the series of papers \cite{hermitianI,hermitianII,hermitianIII,hermitianIV,hermitianV}.
Accordingly, the Grothendieck–Witt theory of a qcqs derived scheme $X$ depends on a choice of \emph{Poincaré structure} on its category $\Perf(X)$ of perfect complexes, which is a nondegenerate quadratic functor $\Qoppa\colon \Perf(X)^\op\to\Sp$. 
We will in fact study six families of such Poincaré structures: three $\Pic^\dagger(X)$-graded families $\Qoppa^\c_{\scr L}$, $\Qoppa^\varsigma_{\scr L}$, and $\Qoppa^\qoppa_{\scr L}$, and three $\Pic^\dagger(X)^{\B\C_2}$-graded families $\Qoppa^\g_{\scr L}$, $\Qoppa^\s_{\scr L}$, and $\Qoppa^\q_{\scr L}$. Here, $\Pic^\dagger(X)$ is the anima of shifted line bundles on $X$.

The so-called \emph{classical} Poincaré structure $\Qoppa^\c$ is the most fundamental one: it encodes the classical notion of symmetric bilinear form, 
and all the other Poincaré structures are modules over $\Qoppa^\c$. 
When $X$ is an affine derived scheme, the classical Grothendieck–Witt theory of $X$ is a spectrum whose connective cover is the group completion of the commutative monoid of nondegenerate symmetric forms over $X$, where the latter monoid is classified by the classical algebraic stack $\Vect^\sym=\coprod_{n\geq 0}\GL_n^\sym/\GL_n$. We also recover the group completions of nondegenerate alternating and quadratic forms by considering symmetric forms valued in shifts of the structure sheaf:
\begin{align*}
\Omega^\infty\GW^\c_{\scr O}(X)&=\Vect^\sym(X)^\gp,\\
 \Omega^\infty\GW^\c_{\scr O[2]}(X)&=\Vect^\alt(X)^\gp,\\
 \Omega^\infty\GW^\c_{\scr O[4]}(X)&=\Vect^\qu(X)^\gp.
\end{align*}

The Poincaré structures $\Qoppa^\s$ and $\Qoppa^\q$ encode the less classical but standard notions of homotopy symmetric and homotopy quadratic forms. The so-called \emph{genuine} Poincaré structure $\Qoppa^\g$ is obtained from $\Qoppa^\s$ by taking the connective cover of its linear part; this general construction was already considered in \cite{hermitianI} for classical rings and in \cite{CHN} for classical schemes. The Poincaré structure $\Qoppa^\c$ can be understood as an animated version of $\Qoppa^\g$, and the Poincaré structures $\Qoppa^\varsigma$ and $\Qoppa^\qoppa$ are similarly animated versions of $\Qoppa^\s$ and $\Qoppa^\q$.

If $\scr L$ is a shifted line bundle with trivial involution, we have $\Qoppa^\c_{\scr L}=\Qoppa^\g_{\scr L}$ for classical schemes, $\Qoppa^\varsigma_{\scr L}=\Qoppa^\s_{\scr L}$ for bounded derived schemes (i.e., derived schemes whose structure sheaf is truncated), and $\Qoppa^\qoppa_{\scr L}=\Qoppa^\q_{\scr L}$ for all derived schemes.
Furthermore, all six Poincaré structures coincide for derived schemes over $\Z[\tfrac 12]$.
However, the classical and genuine Poincaré structures diverge for nonclassical derived schemes where $2$ is not a unit.

For each of these six families of Poincaré structures, we will construct motivic spectra representing their Karoubi-localizing invariants. 
These are defined over arbitrary derived schemes, except for the Poin\-caré structures $\Qoppa^\g$ and $\Qoppa^\s$, for which we can only prove Nisnevich descent on bounded derived schemes.
The motivic spectra representing Grothendieck–Witt theory itself will be denoted by $\KO_{\scr L}$, $\KO^\varsigma_{\scr L}$, $\KO^\qoppa_{\scr L}$, $\KO^\g_{\scr L}$, $\KO^\s_{\scr L}$, and $\KO^\q_{\scr L}$.

In \cite[\sect 5.2]{AnnalaIwasa2}, Annala and Iwasa construct motivic spectra representing Karoubi-localizing invariants of stable categories, such as the motivic spectrum $\KGL$ representing (nonconnective) algebraic K-theory. Their construction hinges on the \emph{Bott periodicity} of these invariants. An analogous construction would be possible for Karoubi-localizing invariants of the Poincaré structures $\Qoppa^\varsigma$, $\Qoppa^\qoppa$, $\Qoppa^\s$, and $\Qoppa^\q$, which are also Bott-periodic. However, Bott periodicity does not hold for the Poincaré structures $\Qoppa^\c$ and $\Qoppa^\g$, and so our approach is necessarily different. We use instead a Poincaré analogue of Robalo's ``bridge'' from \cite{Robalo}, that is, an adjunction between motivic spectra and Karoubi-localizing invariants of Poincaré categories.

\begin{remark}[$\A^1$-invariance]
	For regular noetherian schemes over $\Z[\tfrac 12]$, Grothendieck–Witt theory was shown to be $\A^1$-invariant by Hornbostel \cite[Corollary 1.12]{Hornbostel}, by combining the $\A^1$-invariance of K-theory proved by Quillen and that of L-theory proved by Balmer (with the case of even L-groups proved previously by Karoubi). This was recently generalized by Calmès, Harpaz, and Nardin, who showed that symmetric Grothendieck–\allowbreak Witt theory $\GW^\s$ is $\A^1$-invariant for all finite-dimensional regular noetherian schemes \cite[Theorem~6.3.1]{CHN}. 
	Hornbostel and Calmès–Harpaz–Nardin also construct $\A^1$-invariant motivic spectra representing Grothendieck–\allowbreak Witt theory in their respective contexts.
	However, if $2$ is not a unit, then quadratic and genuine Grothendieck–\allowbreak Witt theory fail to be $\A^1$-invariant even on regular schemes \cite[Example~6.3.2]{CHN}. Additionally, $\A^1$-invariance fails in any case for singular and derived schemes.
\end{remark}

\begin{remark}[$\C_2$-equivariant refinements]
	One important feature of Grothendieck–Witt theory that we do not investigate in this paper is its $\C_2$-equivariant refinement, by which we mean its extension from schemes to schemes with involution. Here, a crucial difference between the classical and genuine theories appears: the genuine theory naturally admits such an extension (see Remark~\ref{rmk:genuine}), which is why we call it ``genuine'', but the classical theory does not.
\end{remark}

\begin{remark}[Spectral schemes]
	The Poincaré structures $\Qoppa^\g_{\scr L}$, $\Qoppa^\s_{\scr L}$, and $\Qoppa^\q_{\scr L}$ are defined for qcqs \emph{spectral} schemes, and we prove Nisnevich descent for Grothendieck–Witt theory in this setting. The category $\MS_S$ of motivic spectra over $S$ also makes sense if $S$ is a spectral scheme, since elementary blowup excision only involves toric varieties that are defined over the sphere. 
	We expect that the computation of the genuine Grothendieck–Witt theory of $\P^n$ and of $\Bl_{0}(\P^n)$ in Section~\ref{sec:PBF} also holds for spectral schemes, which would allow us to define the motivic spectra $\KO^\g_{\scr L}$, $\KO^\s_{\scr L}$, and $\KO^\q_{\scr L}$ over bounded spectral schemes. However, we do not pursue this generalization here.
\end{remark}

The following theorem summarizes our results on the classical Grothendieck–Witt theory of derived schemes. We refer to the main text for more detailed statements, including statements for arbitrary localizing invariants and genuine Poincaré structures.

\begin{theorem}\label{thm:intro}
	There is a motivic $\E_\infty$-ring spectrum $\KO\in \CAlg(\MS_\Z)$ with the following features.
	\begin{enumerate}
		\item \textnormal{(Representability)} For any integer $n$, the weight $n$ cohomology theory on qcqs derived algebraic spaces defined by $\KO$ is classical $\scr O[n]$-valued Grothendieck–Witt theory. More generally, for any qcqs derived algebraic space $X$ and any $\K$-theory element $\xi\in\K(X)$, we have
		\[
		\map_{\MS_X}(\1_X,\Sigma^\xi\KO)= \GW^\c_{\det^\dagger(\xi)}(X),
		\]
		where $\det^\dagger(\xi)\in\Pic^\dagger(X)$ is the shifted determinant of $\xi$.
		In particular, $\KO$, $\Sigma_{\P^1}^2\KO$, and $\Sigma_{\P^1}^4\KO$ represent the Grothendieck–Witt theory of classical symmetric, alternating, and quadratic forms, respectively.
		\item \textnormal{(Hopf fracture square)} Let $\eta\colon \Sigma_{\P^1}\1\to \Sigma\1$ be the Hopf map in $\MS_\Z$. Then, over any derived algebraic space,
		\[
		\KO[\eta^{-1}]=\KW,\quad \KO^\wedge_\eta=\KGL^{\h\C_2},\quad\KO^\wedge_\eta[\eta^{-1}]=\KGL^{\t\C_2},
		\]
		where $\KW$ and $\KGL$ are motivic spectra representing classical $\L$-theory and $\K$-theory as in \textnormal{(i)}.
		\item \textnormal{(Bott fracture square)} The fourth power of the Bott element $\beta$ in $\KGL$ has a canonical lift to $\KO$, which is not a unit. Over any derived algebraic space, $\fib(\KO\to \KO^\wedge_\beta)$ represents $\GW^\q$. Over any bounded derived algebraic space, $\KO[\beta^{-1}]$ represents $\GW^\s$ and $\KO^\wedge_\beta[\beta^{-1}]$ represents normal $\L$-theory $\L^{\mathrm{n}}=\L^\s/\L^\q$.
		\item \textnormal{(Metalinear orientation)} There is a canonical $\E_\infty$-map $\MML\to\KO$ compatible with the $\E_\infty$-map $\MGL\to\KGL$. Here, $\MML$ is the metalinear algebraic cobordism spectrum, which comes with $\E_\infty$-maps $\MSp\to \MSL\to \MML\to\MGL$. 
		\item \textnormal{($\A^1$-localization)} Over any derived algebraic space, $\L_{\A^1}\KO$ is the $\A^1$-invariant motivic spectrum constructed by Calmès, Harpaz, and Nardin in \cite{CHN}, and it represents the $\A^1$-localization of symmetric Grothendieck–Witt theory $\L_{\A^1}\GW^\s$.
		Over regular algebraic spaces of finite Krull dimension, we have $\L_{\A^1}\KO=\KO[\beta^{-1}]$.
	\end{enumerate}
\end{theorem}

\subsection*{Outline}

In Section~\ref{sec:genuine}, we prove Nisnevich descent for the genuine Grothendieck–Witt theory of bounded spectral algebraic spaces.
In Section~\ref{sec:derived}, we introduce the classical Poincaré structures and prove Nisnevich descent for the classical Grothendieck–Witt theory of derived algebraic spaces.
In Section~\ref{sec:PBF}, we analyze the classical and genuine Poincaré categories of projective bundles, and we deduce the projective bundle formula and smooth blowup excision.
Section~\ref{sec:Brauer} introduces Poincaré–Azumaya algebras, which are twists for Grothendieck–Witt theory; this section is not essential for anything else.
In Section~\ref{sec:KO}, we construct the motivic spectrum $\KO$ and its variants.
In Section~\ref{sec:fracture}, we consider the Hopf and Bott fracture squares and prove Theorem~\ref{thm:intro}(ii,iii).
In Section~\ref{sec:BC}, we show that the motivic spectrum $\KO$ is stable under base change, which implies Theorem~\ref{thm:intro}(i) when $\xi$ is an integer.
Section~\ref{sec:Thom} establishes Thom isomorphisms for classical and genuine Grothendieck–Witt theory, which leads to the general case of Theorem~\ref{thm:intro}(i) and to the metalinear orientation of Theorem~\ref{thm:intro}(iv).
Finally, Section~\ref{sec:A1} computes the $\A^1$-localization of classical and genuine Grothendieck–Witt theory and proves Theorem~\ref{thm:intro}(v).

\subsection*{Conventions and notation}

\begin{itemize}
	\item Grothendieck–Witt theory and L-theory, denoted by $\GW$ and $\L$, always refer to the \emph{Karoubi-localizing} versions from \cite{hermitianIV}. This is in contrast to \emph{loc.\ cit.}, where they are called Karoubi–\allowbreak Grothendieck–\allowbreak Witt theory and Karoubi L-theory, and denoted by $\mathbb{GW}$ and $\mathbb{L}$. Recall that the additive versions of $\GW$ and $\L$ agree with the Karoubi-localizing versions if the underlying stable category has connective K-theory (e.g., for regular noetherian algebraic spaces),
	and that the additive and Karoubi-localizing versions of $\GW$ have the same connective cover on idempotent complete Poincaré categories.
	\item Starting in Section~\ref{sec:KO}, derived algebraic spaces are always assumed to be \emph{decent} in the sense of \cite[Tag 03I7]{stacks} (e.g., locally quasi-separated). This ensures that the Nisnevich topology is defined (as points have residue fields) and that Nisnevich sheaves on derived algebraic spaces (such as $S\mapsto \MS_S$) are right Kan extended from affines.
	\item We say ``category'' instead of ``$\infty$-category'' and ``$2$-category'' instead of ``$(\infty,2)$-category''.
	\item We denote by $\An$ the category of anima, by $\Cat$ the category of (possibly large) categories, and by $\Cat_2$ the category of $2$-categories.
	\item We denote by $\Cat^\st$ the category of (usually small) stable categories, by $\Cat^\h$ that of hermitian categories, and by $\Cat^\p$ that of Poincaré categories. The last two are defined in \cite[\sect 1.2]{hermitianI}.
	\item We denote by $\Map(X,Y)$ the anima of maps in a category and by $\map(X,Y)$ the spectrum of maps in a stable category.
	\item Animated rings are always animated \emph{commutative} rings. For an animated ring $R$, $\CAlg^\an_R$ is the category of animated $R$-algebras and $\CAlg^\sm_R$ is the category of smooth $R$-algebras.
	\item For a spectral or derived algebraic space $X$, a \emph{quasi-coherent sheaf} on $X$ is an object of the \emph{stable} category $\QCoh(X)$, and an \emph{invertible sheaf} is a $\otimes$-invertible object therein. We denote by $\Pic^\dagger(X)$ the anima of invertible sheaves on $X$. 
	The usual Picard anima $\Pic(X)$ is the subanima of invertible sheaves $\scr L$ such that both $\scr L$ and its dual $\scr L^\vee$ are connective. The \emph{degree} $\deg\scr L$ of some $\scr L\in \Pic^\dagger(X)$ is the unique locally constant integer $n$ on $X$ such that $\scr L[-n]\in \Pic(X)$.
	\item If $\scr E$ is a finite locally free sheaf on a derived algebraic space $X$, then $\V_X(\scr E)$ is the $X$-scheme classifying linear maps $\scr E\to\scr O$ and $\P_X(\scr E)$ is the $X$-scheme classifying $\pi_0$-epimorphisms $\scr E\onto\scr L$ with $\scr L\in \Pic(X)$.
\end{itemize}

\subsection*{Acknowledgments}

Wolfgang Steimle was involved in the early stages of this project and we are grateful to him for multiple useful insights.
The first author would like to thank Denis Nardin for teaching him much of the formalism of Poincaré categories and motives.

\section{Étale descent for genuine Poincaré categories}
\label{sec:genuine}

We start by recalling the definition of the genuine Poincaré structures $\Qoppa^{\geq m}$ from \cite[Example 3.2.8]{hermitianI} and \cite[\sect 3.4]{CHN}.
Let $X$ be a  qcqs spectral algebraic space, $\scr L\in \Pic^\dagger(X)^{\B\C_2}$ an invertible sheaf with involution on $X$, and $m\in\Z\cup\{\pm\infty\}$ an extended integer (which can be locally constant on $X$, but we leave this implicit).
Since $X$ is qcqs, we have $\QCoh(X)=\Ind(\Perf(X))$.
We denote by $\Tate_{\scr L}$ the quasi-coherent sheaf on $X$ such that
\[
\map(\scr P\otimes\scr P,\scr L)^{\t\C_2} = \map(\scr P,\Tate_{\scr L})
\]
for all $\scr P\in\Perf(X)$. 
Then $\Qoppa^{\geq m}_{\scr L}$ is the Poincaré structure on $\Perf(X)$ with symmetric bilinear part 
\[
\B_\scr L(\scr P,\scr Q)=\map(\scr P\otimes\scr Q,\scr L)
\]
and linear part represented by $\tau_{\geq m}\Tate_{\scr L}$.

\begin{definition}[Genuine $\scr F$-theory]
	\label{def:genuineGW}
	Let $X$ be a qcqs spectral algebraic space, let $\scr L\in \Pic^\dagger(X)^{\B\C_2}$, and let $m\in\Z\cup\{\pm\infty\}$.
	Given a functor $\scr F\colon \Cat^\p\to\scr S$, we define
	\[
	\scr F^{\geq m}_{\scr L}\colon (\SpSpc_X^\qcqs)^\op\to \scr S,\quad U\mapsto \scr F(\Perf(U),\Qoppa^{\geq m}_{\scr L}).
	\]
	Because of the suspension isomorphisms 
	\[
	\Sigma\colon (\Perf(X),\Qoppa^{\geq m}_{\scr L})\simto (\Perf(X),\Qoppa^{\geq m+1}_{-\scr L[2]}),
	\]
	each such Poincaré category with $m\in \Z$ is canonically isomorphic to one with $m=\deg\scr L$ (see Section~\ref{sec:Brauer} for a more detailed discussion of the relevant parameter space). We will write
	\begin{alignat*}{2}
		\Qoppa^\g_{\scr L} &= \Qoppa^{\geq\deg\scr L}_{\scr L}, &\qquad \scr F_{\scr L}^\g&=\scr F^{\geq \deg\scr L}_{\scr L},\\
		\Qoppa^\s_\scr L&=\Qoppa^{\geq-\infty}_\scr L,& \scr F^\s_\scr L&=\scr F^{\geq-\infty}_\scr L, \\
		\Qoppa^\q_\scr L &= \Qoppa^{\geq\infty}_\scr L, & \scr F^\q_\scr L &= \scr F^{\geq\infty}_\scr L.
	\end{alignat*}
	Then, for all $m\in \Z$, we have
	\[
	\scr F^{\geq m}_{\scr L}=\scr F^\g_{(-1)^{u}\scr L[2u]},\quad\text{where}\quad u=m-\deg\scr L.
	\] 
	We refer to $\scr F^\g_{\scr L}$ as the \emph{genuine $\scr L$-valued $\scr F$-theory} of spectral algebraic spaces.
	Note that the Poincaré structures $\Qoppa^\s_{\scr L}$ and $\Qoppa^\q_\scr L$ on $\Perf(X)$ are the symmetric and quadratic Poincaré structures associated with the symmetric bilinear functor $\B_\scr L$, respectively:
	\begin{align*}
	\Qoppa^\s_\scr L(\scr P)&=\map(\scr P^{\otimes 2},\scr L)^{\h\C_2},\\
	\Qoppa^\q_\scr L(\scr P)&=\map(\scr P^{\otimes 2},\scr L)_{\h\C_2}.
	\end{align*}
\end{definition}

\begin{remark}\label{rmk:genuine}
	More generally, if $X\in (\SpSpc^\qcqs)^{\B\C_2}$ is a qcqs spectral algebraic space with involution and $\scr L\in\Pic^\dagger(X/\C_2)$,
	we have a symmetric bilinear functor $\B_\scr L$ on $\Perf(X)$ given by
	\[
	\B_\scr L(\scr P,\scr Q) = \map(q_\otimes(\scr P,\scr Q),\nabla^*\scr L),
	\]
	where $q\colon X\sqcup X\to(X\sqcup X)/\C_2$ is the quotient by the diagonal action of $\C_2$ and $\nabla\colon (X\sqcup X)/\C_2\to X/\C_2$ is the fold map.
	We can then define the Poincaré structures $\Qoppa^{\geq m}_{\scr L}$ on $\Perf(X)$ as above.
	However, we will only consider spaces with trivial involution in what follows.
\end{remark}

A spectral algebraic space $X$ is called \emph{bounded} if its structure sheaf $\scr O_X$ is locally truncated. We denote by $\SpSpc^\b$ the category of bounded spectral algebraic spaces.
If $X$ is a bounded qcqs spectral algebraic space and $j\colon U\into X$ is a quasi-compact open immersion, then the Poincaré functor
\[
j^*\colon (\Perf(X),\Qoppa_\scr L^{\geq m})\to (\Perf(U),\Qoppa_\scr L^{\geq m})
\]
is a Poincaré–Karoubi projection for any $\scr L\in \Pic^\dagger(X)^{\B\C_2}$ and any $m\in \Z\cup\{\pm\infty\}$, by \cite[Proposition~4.2.8]{CHN}.
It follows that the functor
\[
(\SpSpc_X^{\b,\qcqs})^\op\to \Cat^\p,\quad U\mapsto (\Perf(U),\Qoppa_\scr L^{\geq m}),
\]
sends Zariski squares to Poincaré–Karoubi squares \cite[Proposition 4.3.1]{CHN}.  Note that the boundedness restriction is only used to control the linear parts, so it is not needed when $m=\infty$.
It is proved in \cite[Proposition 4.4.1]{CHN} that this functor also sends Nisnevich squares of \emph{classical} qcqs schemes to Poincaré–Karoubi squares.
We will give a new proof of this fact that works more generally for bounded qcqs spectral algebraic spaces (see Corollary~\ref{cor:nis-descent}).

\begin{lemma}[Nisnevich descent for the Tate construction]
	\label{lem:free-Nisnevich}
	Let $S$ be a qcqs spectral algebraic space and let $\scr L\in \Pic^\dagger(S)^{\B\C_2}$.
	Then the section $X\mapsto\Tate_{\scr L_X}$ of $\QCoh$ over $(\SpSpc_S^\qcqs)^\op$ satisfies Nisnevich descent.
\end{lemma}

\begin{proof}
	Recall that $\QCoh$ is an fpqc hypersheaf \cite[Remark 6.2.3.3]{SAG}, hence so is $U\mapsto \map(\scr P_U^{\otimes 2},\scr L_U)$ for any $\scr P\in \QCoh(X)$.
	Since the Tate construction preserves \emph{finite} limits of $\C_2$-spectra, the given section of $\QCoh$ sends any finite diagram that becomes a colimit diagram in the topos of fpqc hypersheaves to a relative limit diagram. This applies in particular to Nisnevich squares.
\end{proof}

The following lemma is the key to our descent results:

\begin{lemma}
	\label{lem:etale-Tate}
	Let $f\colon A\to B$ be an étale morphism of connective $\E_\infty$-rings and let $M\in \Mod_A^{\B\C_2}$ be bounded above. Then the canonical map
	\[
	B\otimes_A M^{\t\C_2} \to (B\otimes_A M)^{\t\C_2},
	\]
	where the Tate constructions are modules via the Frobenius, is an isomorphism.
\end{lemma}

\begin{proof}
	We will use repeatedly that the Tate construction $(\ph)^{\t\C_2}$ preserves uniformly bounded above filtered colimits of $\C_2$-spectra \cite[Corollary 4.2.10]{CHN}.
	We can decompose this map as
	\[
	B\otimes_A M^{\t\C_2} \to ((B\otimes_A B)\otimes_A M)^{\t\C_2} \to (B\otimes_A M)^{\t\C_2}.
	\]
	The first map is an instance of the Tate diagonal map
	\[
	N\otimes_A M^{\t\C_2} \to ((N\otimes_A N)\otimes_A M)^{\t\C_2}
	\]
	defined for all $A$-modules $N$. This map is clearly an isomorphism if $N=A$, and both sides preserve finite sums and filtered colimits in $N$, provided $N$ is flat (since $M$ is bounded above). By Lazard's theorem \cite[Theorem 7.2.2.15]{HA}, it is an isomorphism for all flat $A$-modules $N$, in particular for $N=B$. Thus, it remains to show that the second map is an isomorphism, i.e., that
	\[
	(I_{B/A}\otimes_AM)^{\t\C_2}=0,\quad\text{where}\quad I_{B/A}=\fib(B\otimes_A B\to B)\in \Mod_A^{\B\C_2}.
	\]
	The morphism $f$ is the base change of an étale morphism $f_0\colon A_0\to B_0$ where $A_0$ is a finitely presented $\E_\infty$-ring (by \cite[Corollary 4.4.1.3, Lemma B.1.3.3]{SAG} and \cite[Theorem 7.4.3.18(1), Lemma 7.3.5.13]{HA}).
	Replacing $f$ by $f_0$, we may thus assume that $A$ is finitely presented and in particular that $\pi_0(A)$ has finite Krull dimension.
	We then prove the Tate vanishing by induction on the Krull dimension of $\pi_0(A)$.
	If we regard $(I_{B/A}\otimes_AM)^{\t\C_2}$ as an $A$-module via the composite $A\to A^{\h\C_2} \to A^{\t\C_2}$, then the canonical map
	\[
	N\otimes_A (I_{B/A}\otimes_AM)^{\t\C_2} \to (N\otimes_A I_{B/A}\otimes_AM)^{\t\C_2}
	\]
	is an isomorphism for any flat $A$-module $N$, by Lazard's theorem. 
	Since modules satisfy étale hyper\-descent, it suffices to show that the left-hand side vanishes when $N$ is a strict henselization of $A$, so we can assume $A$ strictly henselian. 
	In this case, by \cite[Tag 04GG]{stacks} and the topological invariance of the étale site, we can write $B=B_1\times B_2$ where $\Spec(B_1)\to \Spec(A)$ is a fold map and $\Spec(B_2)\to\Spec(A)$ factors through $\Spec(A)-\{\mathfrak m\}$. 
	Consider the induced decomposition 
	\[
	I_{B/A}=I_{B_1/A}\oplus I_{B_2/A}\oplus((B_1\otimes_AB_2)\times (B_2\otimes_AB_1)).
	\]
	 The first and third summands are induced $\C_2$-modules, so that their Tate constructions after tensoring with $M$ vanish. For the second summand, choose a finite open cover of $\Spec(A)-\{\mathfrak m\}$ by affine schemes. The $\C_2$-spectrum $I_{B_2/A}\otimes_AM$ is correspondingly a finite limit of $\C_2$-spectra of the form $I_{B'/A'}\otimes_{A'}M'$ with $\dim(\pi_0(A'))<\dim(\pi_0(A))$, so its Tate construction vanishes by the induction hypothesis.
\end{proof}

\begin{proposition}\label{prop:etale-Tate}
	Let $f\colon Y\to X$ be an étale morphism of bounded qcqs spectral algebraic spaces and let $\scr L\in \Pic^\dagger(X)^{\B\C_2}$.
	Then the canonical map
	\[
	f^*(\Tate_{\scr L}) \to \Tate_{f^*(\scr L)}
	\]
	in $\QCoh(Y)$ is an isomorphism.
\end{proposition}

\begin{proof}
	By Lemma~\ref{lem:free-Nisnevich}, the assertion is Nisnevich-local on $X$ and $Y$.
	We are thus reduced to the affine case, which follows from Lemma~\ref{lem:etale-Tate}.
\end{proof}

\begin{remark}\label{rmk:bounded}
	The boundedness assumption in Proposition~\ref{prop:etale-Tate} cannot be removed, even for an open immersion of affine \emph{derived} schemes.
	For a counterexample, choose a sequence $(M_i)_{i\geq 0}$ of connective $\Z$-modules whose colimit is not preserved by $(\ph)^{\t\C_2}$, e.g., the linearization of the skeletal filtration of $\B\C_2$. Let $M=\bigoplus_{i\geq 0}M_i$ and let $A$ be the free animated $\Z[x]$-algebra on $M$, where the action of $x$ on $M$ is given by the transition maps $M_i\to M_{i+1}$. Then Lemma~\ref{lem:etale-Tate} fails for $B= A[x^{-1}]$ and $M=A$.
	Equivalently, the Poincaré functor $(\Perf(A),\Qoppa^\s)\to (\Perf(A[x^{-1}]),\Qoppa^\s)$ is not a Poincaré–Karoubi projection.
\end{remark}

\begin{proposition}[Étale hyperdescent for genuine Poincaré categories]
	\label{prop:etale-descent}
	Let $S$ be a qcqs spectral algebraic space, let $\scr L\in \Pic^\dagger(S)^{\B\C_2}$ and let $m\in \Z\cup\{\pm\infty\}$. Then the functor
	\[
	(\SpSpc_S^{\b,\qcqs})^\op \to \Cat^\p,\quad X \mapsto (\Perf(X),\Qoppa^{\geq m}_{\scr L}),
	\]
	is an étale hypersheaf that sends open immersions to Poincaré–Karoubi projections. 
	If $m=-\infty$, it is even an fpqc hypersheaf on $\SpSpc_S^{\qcqs}$, and if $m=\infty$, it is a Nisnevich sheaf on $\SpSpc_S^{\qcqs}$ that sends open immersions to Poincaré–Karoubi projections.
\end{proposition}

\begin{proof}
	If $j\colon U\into X$ is an open immersion in $\Sp\Sch_S^{\qcqs}$, 
	$j^*\colon \Perf(X)\to \Perf(U)$ is a Karoubi projection that preserves the duality $\D_{\scr L}$.
	Hence, it is a Poincaré–Karoubi projection if and only if 
	\[
	j^*(\tau_{\geq m}\Tate_{\scr L}) \simto \tau_{\geq m}\Tate_{j^*(\scr L)}
	\]
	(cf.\ the proof of \cite[Proposition 4.2.8]{CHN}), 
	which is trivial if $m=\infty$ and follows from Proposition~\ref{prop:etale-Tate} if $X$ is bounded.
	Since the functor in the statement clearly sends coproducts to products, it remains to check monogenic hyperdescent.
	Given an fpqc hypercover $f_\bullet\colon X_\bullet\to X$, we have
	\[
	\lim_n (\Perf(X_n),\Qoppa^{\geq m}_{f_n^*\scr L}) = (\Perf(X),\lim_n \Qoppa^{\geq m}_{f_n^*\scr L}\circ f_n^*).
	\]
	Assuming $X$ bounded and $f_\bullet$ an étale hypercover if $m\neq-\infty$, the claim is then that the map
	\[
	\Qoppa^{\geq m}_{\scr L}\to \lim_n \Qoppa^{\geq m}_{f_n^*\scr L}\circ f_n^*\colon \Perf(X)^\op\to\Sp
	\]
	is an isomorphism. If $m=-\infty$, that is, for $\Qoppa^\s_{\scr L}=\map((\ph)^{\otimes 2},\scr L)^{\h\C_2}$, this follows from the fact that the map $\scr L\to \lim_n f_{n*}f_n^*\scr L$ is an isomorphism, by fpqc hyperdescent for quasi-coherent sheaves.
	Otherwise, given the fiber sequence
	\[
	\Qoppa^{\geq m}_{\scr L}\to \Qoppa^\s_{\scr L} \to \map(\ph,\tau_{< m}\Tate_{\scr L}),
	\]
	we must check that the canonical map
	\[
	\tau_{<m}(\Tate_\scr L)\to \lim_n f_{n*}\tau_{<m}\Tate_{f_n^*\scr L}
	\]
	is an isomorphism. But this follows from Proposition~\ref{prop:etale-Tate}, the fact that $f_{n}^*$ is t-exact, and fpqc hyper\-descent for quasi-coherent sheaves.
	If $m=\infty$ and we consider a Nisnevich square instead of a hypercover, the result follows directly from Lemma~\ref{lem:free-Nisnevich}.
\end{proof}

\begin{corollary}[Nisnevich descent for genuine $\scr F$-theory]
	\label{cor:nis-descent}
	Let $S$ be a qcqs spectral algebraic space, let $\scr L\in \Pic^\dagger(S)^{\B\C_2}$ and let $m\in \Z\cup\{\pm\infty\}$. Then the functor
	\[
	(\SpSpc_S^{\b,\qcqs})^\op \to \Cat^\p,\quad X \mapsto (\Perf(X),\Qoppa^{\geq m}_{\scr L}),
	\]
	sends Nisnevich squares to Poincaré–Karoubi squares. Hence, for any category $\scr S$ and Karoubi-localizing invariant $\scr F\colon \Mod_{(\Perf(S),\Qoppa^\g)}(\Cat^\p)\to \scr S$, the presheaf $\scr F^{\geq m}_{\scr L}$ on $\SpSpc_S^{\b,\qcqs}$ is a Nisnevich sheaf. If $m=\infty$, these statements hold on $\SpSpc_S^{\qcqs}$.
\end{corollary}

\begin{proof}
	This follows directly from Proposition~\ref{prop:etale-descent}.
\end{proof}

\begin{remark}[Galois hyperdescent for genuine Poincaré categories]
	Lemma~\ref{lem:etale-Tate} holds with no boundedness assumption if $f\colon A\to B$ is finite étale. Indeed, the $A$-module $B$ is then a retract of a finite free $A$-module, and $\Spec(f)$ is a fold map locally in the finite étale topology. Consequently, Proposition~\ref{prop:etale-Tate} holds for any finite étale map $f\colon Y\to X$ of qcqs spectral algebraic spaces, and the functor of Proposition~\ref{prop:etale-descent} satisfies Galois hyperdescent on $\SpSpc_S^{\qcqs}$.
\end{remark}

\begin{remark}[Descent for the Tate Poincaré structure]
	\label{rmk:Tate-descent}
	Let $S$ be a qcqs spectral algebraic space, let $\sqrt\scr L\in \Pic^\dagger(S)$, and let $\scr L=(\sqrt\scr L)^{\otimes 2}\in\Pic^\dagger(S)^{\B\C_2}$.
	There is then a Poincaré structure $\Qoppa^\t_{\scr L}$ on $\Perf(S)$ with bilinear part $\B_\scr L$ and linear part represented by $\sqrt\scr L$. 
	The arguments in the proof of Proposition~\ref{prop:etale-descent} show that the functor
	\[
	(\SpSpc_S^\qcqs)^\op \to \Cat^\p,\quad X\mapsto (\Perf(X),\Qoppa^{\t}_{\scr L}),
	\]
	sends open immersions to Poincaré–Karoubi projections, is an étale hypersheaf on $\SpSpc_S^{\b,\qcqs}$, and is a Nisnevich sheaf and Galois hypersheaf on $\SpSpc_S^{\qcqs}$. Hence, as in Corollary~\ref{cor:nis-descent}, for any Karoubi-localizing invariant $\scr F\colon\Mod_{(\Perf(S),\Qoppa^\t)}(\Cat^\p)\to\scr S$, the functor $\scr F^{\t}_{\scr L}=\scr F(\Perf(\ph),\Qoppa^\t_{\scr L})$ is a Nisnevich sheaf on $\SpSpc_S^{\qcqs}$.
\end{remark}

We conclude with a categorification of Proposition~\ref{prop:etale-descent}:

\begin{proposition}[Étale descent for Poincaré stacks]
	\label{prop:mod-descent}
	The functors
	\begin{align*}
	(\SpSpc^{\qcqs})^\op \to \CAlg(\Cat_2),\quad & X \mapsto \Mod_{(\Perf(X),\Qoppa^{\t})}(\Cat^\p_\idem),\\
	(\SpSpc^{\b,\qcqs})^\op \to \CAlg(\Cat_2),\quad&  X \mapsto \Mod_{(\Perf(X),\Qoppa^\g)}(\Cat^\p_\idem),
	\end{align*}
	are étale sheaves. Here, $\Cat^\p_\idem$ is the category of idempotent complete Poincaré categories.
\end{proposition}

\begin{proof}
	The functor $X\mapsto \Mod_{\QCoh(X)}(\Pr^\L)$ is an étale sheaf of symmetric monoidal $2$-categories by \cite[Theorem 10.2.0.2, Remark 10.1.2.10]{SAG}, and the subfunctor $X \mapsto \Mod_{\Perf(X)}(\Cat^\st_\idem)$ is also an étale sheaf by \cite[Theorem D.5.3.1, Proposition D.5.2.2]{SAG}. 
	
	We first prove that $X\mapsto \Mod_{(\Perf(X),\Qoppa^{\t})}(\Cat^{\h}_\idem)$ is an étale sheaf.
	Fix a $\Perf(X)$-module $\scr C$ and write $\scr C_U$ for $\scr C\otimes_{\Perf(X)}\Perf(U)$. Denote by $\mathrm{Herm}_\scr C(U)$ the category of $(\Perf(U),\Qoppa^\t)$-linear hermitian structures on $\scr C_U$, which is given by the pullback
	\[
	\begin{tikzcd}[column sep=1.5em]
		\mathrm{Herm}_\scr C(U) \ar{r} \ar{d} & \Ind(\scr C_U\otimes_{\Perf(U)}\scr C_U)^{\h\C_2} \ar{d}{\Delta^*(\ph)^{\t\C_2}} \\
		\mathrm{Ar}(\Ind(\scr C_U)) \ar{r}{\mathrm{target}} & \Ind(\scr C_U)
	\end{tikzcd}
	\]
	\cite[Corollary 1.3.12 and Example 5.4.13]{hermitianI}.
	We have to show that $\mathrm{Herm}_\scr C$ is an étale sheaf. Note that the other three corners are étale sheaves in $U$, but this does not immediately imply the claim since the functor $\Delta^*(\ph)^{\t\C_2}$ is only laxly natural in $U$: for a morphism $f\colon V\to U$ over $X$, we have a lax square
	\[
	\begin{tikzcd}[column sep=2em]
		\Ind(\scr C_U\otimes\scr C_U)^{\h\C_2} \ar{d}[swap]{\Delta^*(\ph)^{\t\C_2}} \ar{r}{f^*} \ar[phantom]{dr}{\Longrightarrow} & \Ind(\scr C_V\otimes\scr C_V)^{\h\C_2} \ar{d}{\Delta^*(\ph)^{\t\C_2}} \\
		\Ind(\scr C_U) \ar{r}[swap]{f^*} & \Ind(\scr C_V)
	\end{tikzcd}
	\]
	\cite[Corollary 1.4.4]{hermitianI}.
	If however $f$ is finite étale, then the left Kan extension along $f^*\colon \scr C_U^\op\to\scr C_V^\op$ is given by precomposition with its two-sided adjoint $f_*$, so that the square commutes strictly. This shows that $\mathrm{Herm}_\scr C$ is a sheaf for the finite étale topology.
	
	It remains to show that $\mathrm{Herm}_\scr C$ is a Nisnevich sheaf.
	To that end, fix $\B\in \Ind(\scr C\otimes_{\Perf(X)}\scr C)^{\h\C_2}$ and denote by $\mathrm{Herm}_{\scr C,\B}(U)$ the category of hermitian refinements of $(\scr C_U,\B_U)$:
	\[
	\mathrm{Herm}_{\scr C,\B}(U)=\Ind(\scr C_U)_{/\Delta^*(\B_U)^{\t\C_2}}.
	\]
	It suffices to show that $\mathrm{Herm}_{\scr C,\B}$ is a Nisnevich sheaf. 
	Suppose given a finite diagram $(V_i)_i$ with colimit $U$ in the étale topos (e.g., a Nisnevich square), and let $f_i\colon V_i\to U$ be the canonical map.
	 Since the Tate construction preserves finite limits, we then have
	\[
	\Delta^*(\B_U)^{\t\C_2} \simto \lim_i (\Delta^*(\B_{V_i})^{\t\C_2}\circ f_i^*) \colon \scr C_U^\op \to \Sp,
	\]
	which implies that $\mathrm{Herm}_{\scr C,\B}$ takes the given diagram to a limit diagram.
	
	Hence, if $X\mapsto (\scr C_X,\Qoppa_X)$ is a section of $X\mapsto \CAlg(\Mod_{(\Perf(X),\Qoppa^{\t})}(\Cat^{\h}_\idem))$ that is cocartesian over étale morphisms, then also $X\mapsto \Mod_{(\scr C_X,\Qoppa_X)}(\Cat^{\h}_\idem)$ is an étale sheaf. 
	This holds for the section $X\mapsto(\Perf(X),\Qoppa^\g)$ by Proposition~\ref{prop:etale-Tate}: for any étale map $Y\to X$ in $\SpSpc^{\b,\qcqs}$, the induced map
	\[
	(\Perf(X),\Qoppa^\g)\otimes_{(\Perf(X),\Qoppa^{\t})} (\Perf(Y),\Qoppa^{\t}) \to (\Perf(Y),\Qoppa^\g)
	\]
	is an isomorphism.
	It remains to show that if a $(\Perf(X),\Qoppa^{\t})$-module is Poincaré étale-locally on $X$, then it is Poincaré, and similarly for duality-preserving functors. But this follows immediately from the fact that $X\mapsto\Mod_{\Perf(X)}(\Cat^\st_\idem)$ is an étale sheaf.
\end{proof}

\begin{remark}\label{rem:PK-local}
	It follows from Proposition~\ref{prop:mod-descent} that the property of being a Poincaré–Karoubi inclusion, projection, sequence, or square in $\Mod_{(\Perf(X),\Qoppa^{\t})}(\Cat^\p_\idem)$ is étale-local on $X$.
\end{remark}

\section{Classical Poincaré structures on derived algebraic spaces}
\label{sec:derived}

In derived algebraic geometry, the symmetric power functors induce natural Poincaré structures:

\begin{construction}[Classical Poincaré structures I]
	\label{cst:Sym^n}
	Recall from \cite[\sect 25.2]{SAG} that there are symmetric, exterior, and divided power functors
	\[
	\Sym^n,\; \Lambda^n,\; \Gamma^n\colon \Mod\to \Mod\colon \CAlg^\an \to \Cat,
	\]
	which are the unique filtered-colimit-preserving $n$-excisive functors extending the usual functors on finite free modules over polynomial $\Z$-algebras (see \cite[Theorem 3.36]{BrantnerMathew} for the extension to nonconnective modules).
	They preserve perfect modules and are related by
	\[
	\Sym^n(M[1])=\Lambda^n(M)[n]\quad\text{and}\quad \Lambda^n(M[1])=\Gamma^n(M)[n].
	\]
	Let $X\colon \CAlg^\an\to\An$ be a functor and let $\scr L\in \Pic^\dagger(X)$. We define nondegenerate quadratic functors $\Perf(X)^\op\to\Sp$ by
	\begin{alignat*}{2}
		&\Qoppa^\sym_\scr L &&= \map(\Sym^2(\ph),\scr L),\\
		&\Qoppa^{\alt}_{\scr L} &&=\map(\Lambda^2(\ph),\scr L),\\
		&\Qoppa^{\qu}_{\scr L} &&=\map(\Gamma^2(\ph),\scr L).
	\end{alignat*}
	Note that $\Qoppa^\sym_\scr L$ and $\Qoppa^\qu_\scr L$ have symmetric bilinear part $\B_{\scr L}$, while $\Qoppa^\alt_\scr L$ has symmetric bilinear part $\B_{-\scr L}$.
Because of the above relations between $\Sym^n$, $\Lambda^n$, and $\Gamma^n$, we have
\[
\Qoppa^\alt_{\scr L} = \Qoppa^\sym_{\scr L[2]}\circ\Sigma\quad\text{and}\quad \Qoppa^\qu_{\scr L} = \Qoppa^\alt_{\scr L[2]}\circ \Sigma.
\]
\end{construction}

\begin{remark}
	If $n\geq 2$, the functors $\Sym^n$, $\Lambda^n$, and $\Gamma^n$ are the only members of the family 
	\[\rm P^n_{i}=\Sigma^{-ni}\Sym^n\Sigma^i,\quad {i\in \Z},\]
	that take static values on finite free modules over static rings, which explains their special status.
	The quadratic functors $\rm P^2_{i}$ define as above Poincaré structures on $\Perf(X)$ (which are denoted by $\c\Qoppa^{\geq\deg\scr L+ i}_{(-1)^i\scr L}$ below for qcqs derived algebraic spaces). These Poincaré categories are however isomorphic via $\Sigma^i$ to $(\Perf(X),\Qoppa^\sym_{\scr L[2i]})$, so that the Poincaré structures $\Qoppa^\sym_{\scr L}$ for $\scr L\in \Pic^\dagger(X)$ subsume the others.
\end{remark}

To better understand these classical Poincaré structures and to compare them to the genuine Poincaré structures of Section~\ref{sec:genuine}, we use another perspective.
If $A$ is a static ring, then the double-speed Postnikov filtrations $\tau_{\geq 2*}A^{\t\C_2}$ and $\tau_{\geq 2*+1}(-A)^{\t\C_2}$ have associated graded $\Sigma^{2*}A/2$ and $\Sigma^{2*+1}A/2$, respectively. This does not hold anymore if $A$ is animated, as then $A/2$ need not be concentrated in degrees $0$ and $1$. We can however consider the animation of these Postnikov filtrations, i.e., their left Kan extension from polynomial rings to animated rings, which will have these associated graded for all animated rings. These filtrations are the linear parts of new Poincaré structures $\c\Qoppa^{\geq 2*}_A$ and $\c\Qoppa^{\geq 2*+1}_{-A}$ for perfect modules over animated rings, 
with comparison maps to $\Qoppa^{\geq 2*}_A$ and $\Qoppa^{\geq 2*+1}_{-A}$. We will see in Proposition~\ref{prop:Sym^n} that they recover the Poincaré structures of Construction~\ref{cst:Sym^n}.

Through this animation process, the ``skew'' Poincaré structures such as $\Qoppa^{\geq 0}_{-A}$ (genuine skew-sym\-met\-ric), $\Qoppa^{\geq 1}_A$ (genuine even), and $\Qoppa^{\geq 2}_{-A}$ (genuine skew-quadratic) disappear, since on polynomial $\Z$-algebras we have $\Qoppa^{\geq 2m}_{-A}=\Qoppa^{\geq 2m+1}_{-A}$ and $\Qoppa^{\geq 2m-1}_A=\Qoppa^{\geq 2m}_A$.
On the other hand, the quadratic Poincaré structures $\Qoppa^\q_{\pm A}$ are already animated, as are the symmetric Poincaré structures $\Qoppa^\s_{\pm A}$ if $A$ is bounded (Proposition~\ref{prop:genuine-vs-derived}).

\begin{construction}[Classical Poincaré structures II]
	\label{cst:dQoppa}
	We define a functor
	\[
	(\dSpc^\qcqs)^\op\to \CAlg(\Cat^\p),\quad X\mapsto (\Perf(X),\c\Qoppa^{\geq 0}).
	\]
	The bilinear part of $\c\Qoppa^{\geq 0}$ is $\B_{\scr O}$, and its linear part is represented by the sheaf 
	\[\c\Tate^{\geq 0}_{\scr O}\in \CAlg(\QCoh(X))_{/\Tate_{\scr O}}\] 
	constructed as follows.
	Let $p\colon\scr C\to(\dSpc^\qcqs)^\op$ be the cocartesian fibration classified by the functor $X\mapsto\CAlg(\QCoh(X))$. We define the section $\c\Tate^{\geq 0}$ of $p$ to be the $p$-right Kan extension from $\CAlg^\an$ of the $p$-left Kan extension from smooth $\Z$-algebras of $\tau_{\geq 0}(\ph)^{\t\C_2}$, where $A^{\t\C_2}$ is an $\E_\infty$-$A$-algebra via the Frobenius; this comes with an $\E_\infty$-map $\c\Tate^{\geq 0}_{\scr O}\to \Tate_{\scr O}$, since $X\mapsto\Tate_{\scr O_X}$ is $p$-right Kan extended from affine derived schemes (by Lemma~\ref{lem:free-Nisnevich}).
	
	More generally, for any $\scr L\in\Pic^\dagger(X)$, $m\in\Z\cup\{\pm\infty\}$, and $\epsilon\in\{\pm 1\}$, we define a Poincaré category
	\[
	(\Perf(X),\c\Qoppa^{\geq m}_{\epsilon\scr L})
	\]
	as follows.
	The bilinear part of $\c\Qoppa^{\geq m}_{\epsilon\scr L}$ is $\B_{\epsilon\scr L}$, and its linear part $\c\Tate^{\geq m}_{\epsilon\scr L}$ is obtained by Kan extending as above the section $\tau_{\geq m}(\epsilon(\ph))^{\t\C_2}$ of $\QCoh(\ph)$ from the category of pairs $(S,I)$ where $S$ is a smooth $\Z$-algebra and $I\in\Pic^\dagger(S)$. We also write $\c\Tate_{\epsilon\scr L}$ for $\c\Tate_{\epsilon\scr L}^{\geq-\infty}$.
\end{construction}

\begin{remark}\label{rmk:Pic-LKE}
	We explain how to explicitly compute the left Kan extension in the above construction.
	Let $R$ be an animated ring, $\scr E_R\to \CAlg_R^\an$ the cocartesian fibration classified by $\Pic^\dagger$ (or any smooth algebraic stack over $R$ with smooth and quasi-affine diagonal), and $\scr E_R^\sm$ its restriction to $\CAlg_R^\sm$.
	By \cite[Proposition A.0.4]{EHKSY3}, the functor $\Pic^\dagger$ on $\CAlg_R^\an$ is left Kan extended from $\CAlg_R^\sm$, as is the functor $\Map(I,I'\otimes_R(\ph))$ for any $I,I'\in\Pic^\dagger(R)$. The former implies that the category $(\scr E_R^\sm)_{/(A,L)}$ is nonempty for any $(A,L)\in\scr E_R$.
	The latter implies, using Quillen's Theorem~A, that for any $(S,I) \in (\scr E_R^\sm)_{/(A,L)}$, the functor
	\[
	(\CAlg_{S}^\sm)_{/A} \to (\scr E_R^\sm)_{/(A,L)},\quad S'\mapsto (S',I\otimes_{S}S'),
	\]
	is cofinal.
	Thus, a functor $F\colon\scr E_R\to\scr C$ is left Kan extended from $\scr E_R^\sm$ if and only if, for any $(S,I)\in\scr E_R^\sm$, the functor $F(\ph,I\otimes_S\ph)\colon \CAlg_S^\an\to \scr C$ is left Kan extended from $\CAlg_S^\sm$.
\end{remark}

\begin{remark}
	\label{rmk:dTate}
	Let $A$ be an animated ring, $L\in\Pic^\dagger(A)$, $m\in\Z\cup\{\pm\infty\}$, and $\epsilon\in\{\pm 1\}$.
	Using Remark~\ref{rmk:Pic-LKE}, we can understand $\c\Tate^{\geq m}_{\epsilon L}$ quite concretely:
	\begin{enumerate}
	\item If $\epsilon=(-1)^{\deg L}$ (resp.\ if $\epsilon=-(-1)^{\deg L}$), then
	\[
	\c\Tate^{\geq 2*-1}_{\epsilon L}=\c\Tate^{\geq 2*}_{\epsilon L}\quad \text{resp.}\quad \c\Tate^{\geq 2*}_{\epsilon L}=\c\Tate^{\geq 2*+1}_{\epsilon L},
	\]
	since the odd (resp.\ even) homotopy groups of $(\epsilon L)^{\t\C_2}$ vanish if $A$ is smooth over $\Z$. 
	Moreover, $\c\Tate^{\geq 2*}_{\epsilon L}$ (resp.\ $\c\Tate^{\geq 2*+1}_{\epsilon L}$) is an exhaustive filtration of $\c\Tate_{\epsilon L}$ with associated graded $\Sigma^{2*-\deg L}L/2$ (resp.\ $\Sigma^{2*+1-\deg L}L/2$), viewed as an $A$-module via the Frobenius. This filtration is also complete, since $\c\Tate^{\geq m}_{\epsilon L}$ is $m$-connective.
	\item The underlying $\Z$-module of the $A$-module $\c\Tate^{\geq m}_{\epsilon L}$ is given by
	\[
	\c\Tate^{\geq m}_{\epsilon L} = L\otimes_\Z\tau_{\geq m-\deg L}(\epsilon \Z)^{\t\C_2},
	\]
	where $(\epsilon \Z)^{\t\C_2}$ is a left $\Z$-module via the canonical map $\Z\to\Z^{\t\C_2}$ and a right $\Z$-module via the Frobenius map $\Z\to\Z^{\t\C_2}$.
	Indeed, the right-hand side preserves colimits in $L$ and agrees with 
	$\tau_{\geq m}(\epsilon L)^{\t\C_2}$ if $A$ is a flat $\Z$-algebra
	(since $L$ is then a filtered colimit of uniformly shifted finite free $\Z$-modules and $\tau_{\geq m}(\ph)^{\t\C_2}$ preserves such colimits).
	\item It follows from either (i) or (ii) that the functor
	\[
	\c\Tate^{\geq m}_{\ph\otimes_A\epsilon L}\colon \CAlg^\an_A\to \Sp
	\]
	preserves sifted colimits, i.e., it is left Kan extended from $\Poly_A$.
	\end{enumerate}
\end{remark}

\begin{definition}[Classical $\scr F$-theory]
	\label{def:derivedGW}
	Let $X$ be a qcqs derived algebraic space, let $\scr L\in \Pic^\dagger(X)$, let $m\in\Z\cup\{\pm\infty\}$, and let $\epsilon\in\{\pm 1\}$.
		Given a functor $\scr F\colon \Cat^\p\to\scr S$, we define
		\[
		\c\scr F^{\geq m}_{\epsilon\scr L}\colon (\dSpc_X^\qcqs)^\op\to \scr S,\quad U\mapsto \scr F(\Perf(U),\c\Qoppa^{\geq m}_{\epsilon\scr L}).
		\]
		Because of the suspension isomorphisms
		\[
		\Sigma\colon (\Perf(X),\c\Qoppa^{\geq m}_{\epsilon\scr L})\simto (\Perf(X),\c\Qoppa^{\geq m+1}_{-\epsilon\scr L[2]})
		\]
		and the isomorphisms of Remark~\ref{rmk:dTate}(i)
		\[
		\c\Qoppa^{\geq 0}_{-\scr L}=\begin{cases}
		\c\Qoppa^{\geq 1}_{-\scr L}, & \text{if $\deg\scr L$ is even,} \\ 
		\c\Qoppa^{\geq -1}_{-\scr L}, & \text{if $\deg\scr L$ is odd,} 
		\end{cases}
		\]
		each such Poincaré category with $m\in \Z$ is canonically isomorphic to one with $m=\deg\scr L$ and $\epsilon=1$ (see Section~\ref{sec:Brauer} for a more detailed discussion of the relevant parameter space). We will write
		\begin{alignat*}{2}
			\Qoppa^\c_{\scr L} &= \c\Qoppa^{\geq\deg\scr L}_{\scr L}, &\qquad \scr F^\c_{\scr L}&=\c\scr F^{\geq \deg\scr L}_{\scr L},\\
			\Qoppa^\varsigma_\scr L&=\c\Qoppa^{\geq-\infty}_\scr L,& \scr F^\varsigma_\scr L&=\c\scr F^{\geq-\infty}_\scr L, \\
			\Qoppa^\qoppa_\scr L &= \c\Qoppa^{\geq\infty}_\scr L, & \scr F^\qoppa_\scr L &= \c\scr F^{\geq\infty}_\scr L.
		\end{alignat*}
		Then, for all $m\in\Z$, if $u=m-\deg\scr L$, we have
		\[
		\c\scr F^{\geq m}_{\epsilon\scr L}
		=\begin{cases}
			\scr F^\c_{\scr L[2u]}, & \text{if $\epsilon = (-1)^{u}$,} \\
			\scr F^\c_{\scr L[2u-2]}, & \text{if $\epsilon=-(-1)^u$ and $\deg\scr L$ is even,} \\
			\scr F^\c_{\scr L[2u+2]}, & \text{if $\epsilon=-(-1)^u$ and $\deg\scr L$ is odd.}
		\end{cases}
		\]
		We refer to $\scr F^\c_{\scr L}$ as the \emph{classical $\scr L$-valued $\scr F$-theory} of derived algebraic spaces.
		By Proposition~\ref{prop:genuine-vs-derived} below, we have $\Qoppa^\qoppa_\scr L=\Qoppa^\q_\scr L$ (but the right-hand side is defined more generally with an involution on $\scr L$), and the canonical map $\Qoppa^\varsigma_\scr L\to\Qoppa^\s_\scr L$ is an isomorphism on bounded qcqs derived algebraic spaces.
\end{definition}

\begin{remark}
	Note that for genuine $\scr F$-theory $\scr F^\g_\scr L$, $\scr L$ can be an invertible sheaf with involution, whereas for classical $\scr F$-theory $\scr F^\c_{\scr L}$, $\scr L$ is simply an invertible sheaf.
\end{remark}

\begin{proposition}\label{prop:Sym^n}
	Let $X$ be a qcqs derived algebraic space and let $\scr L\in \Pic^\dagger(X)$. Then there are isomorphisms of Poincaré structures on $\Perf(X)$, natural in $(X,\scr L)$:
	\[
		\Qoppa^\sym_\scr L=\Qoppa^\c_{\scr L}=\c\Qoppa^{\geq \deg\scr L}_{\scr L},\quad\Qoppa^\alt_\scr L=\c\Qoppa^{\geq \deg\scr L+1}_{-\scr L},\quad\Qoppa^\qu_\scr L=\c\Qoppa^{\geq \deg\scr L+2}_{\scr L}.
	\]
	Hence, for any functor $\scr F\colon \Cat^\p\to\scr S$, there are isomorphisms of functors $(\dSpc_X^\qcqs)^\op\to \scr S$:
	\[
		\scr F^\c_\scr L=\scr F(\Perf(\ph),\Qoppa^\sym_\scr L),\quad\scr F^\c_{\scr L[2]}=\scr F(\Perf(\ph),\Qoppa^\alt_\scr L),\quad\scr F^\c_{\scr L[4]}=\scr F(\Perf(\ph),\Qoppa^\qu_\scr L).
	\]
\end{proposition}

\begin{proof}
	By the suspension isomorphisms, it is enough to prove the claim for $\Qoppa^\sym_{\scr L}$.
	For static $\Z$-algebras, this is exactly \cite[Proposition 4.2.22]{hermitianI}.
	It thus remains to show that the linear part of $\Qoppa^\sym_{\scr L}$ is extended from smooth $\Z$-algebras as in Construction~\ref{cst:dQoppa}.
	Denote by $\Tate^\sym_\scr L \in \QCoh(X)_{/\Tate_\scr L}$ the sheaf representing the linear part of $\Qoppa^\sym_\scr L$, which fits in the fiber sequence
	\[
	\map((\ph)^{\otimes 2},\scr L)_{\h\C_2}\to \map(\Sym^2(\ph),\scr L) \to \map(\ph,\Tate^\sym_\scr L).
	\]
	 Since $(\ph)_{\h\C_2}$ preserves finite limits, the section $U\mapsto\Tate^\sym_{\scr L_U}$ satisfies Nisnevich descent and hence is right Kan extended from affine derived schemes.
	 If $A\in\CAlg^\an$ and $L\in\Pic^\dagger(A)$, then $\Tate^\sym_L$ is the cofiber of a map $L_{\h\C_2}\to L$. It follows that $\Tate^\sym_{\ph\otimes_AL}\colon \CAlg^\an_A\to\Sp$ preserves sifted colimits and hence is left Kan extended from smooth $A$-algebras.
\end{proof}

\begin{remark}
	Under the identifications $\c\Qoppa^{\geq\deg\scr L}_{\scr L}=\Qoppa^\sym_{\scr L}$ and $\c\Qoppa^{\geq\deg\scr L+2}_{\scr L}=\Qoppa^\qu_{\scr L}$ of Proposition~\ref{prop:Sym^n}, the morphism $\c\Qoppa^{\geq\deg\scr L+2}_{\scr L}\to \c\Qoppa^{\geq\deg\scr L}_{\scr L}$ is induced by the norm map $\Sym^2\to\Gamma^2$.
\end{remark}

\begin{proposition}[Descent properties of classical Poincaré structures]
	\label{prop:descent-animated}
	Let $X$ be a qcqs derived algebraic space, $\scr L\in \Pic^\dagger(X)$, $\epsilon\in\{\pm 1\}$, and $m\in \Z\cup\{\pm\infty\}$.
	\begin{enumerate}
		\item If $f\colon Y\to X$ is an étale morphism of qcqs derived algebraic spaces, the canonical map
	\[
	f^*(\c\Tate^{\geq m}_{\epsilon\scr L}) \to \c\Tate^{\geq m}_{\epsilon f^*(\scr L)}
	\]
	in $\QCoh(Y)$ is an isomorphism.
	\item If $j\colon U\into X$ is a quasi-compact open immersion, the Poincaré functor
	\[
	j^*\colon (\Perf(X), \c\Qoppa^{\geq m}_{\epsilon\scr L}) \to (\Perf(U),\c\Qoppa^{\geq m}_{\epsilon\scr L})
	\]
	is a Poincaré–Karoubi projection.
	\item The functor
	\[
	(\dSpc_X^\qcqs)^\op \to \Cat^\p,\quad U\mapsto (\Perf(U),\c\Qoppa^{\geq m}_{\epsilon\scr L}),
	\]
	is an étale hypersheaf, and even an fpqc hypersheaf if $m\neq -\infty$.
	\item The functor
	\[
	(\dSpc^\qcqs)^\op\to\CAlg(\Cat_2),\quad U\mapsto \Mod_{(\Perf(U),\Qoppa^\c)}(\Cat^\p_\idem),
	\]
	is an étale sheaf.
	\end{enumerate}
\end{proposition}

\begin{proof}
	We first prove (i) in the affine case, i.e., that if $A\to B$ is étale, then the map
	\[
	B\otimes_A \c\Tate^{\geq m}_{\epsilon L}\to \c\Tate^{\geq m}_{B\otimes_A\epsilon L}
	\]
	is an isomorphism.
	The functor $\CAlg^\an\to \An$ sending $A$ to the anima of étale maps $A\to B$ is left Kan extended from smooth $\Z$-algebras, as is $\Pic^\dagger$, by \cite[Proposition A.0.1]{EHKSY3}.
	Since $\c\Tate^{\geq m}_{\ph\otimes_A\epsilon L}$ preserves sifted colimits by Remark~\ref{rmk:dTate}(iii), we can assume that $A$ is a smooth $\Z$-algebra. Then the claim follows from Lemma~\ref{lem:etale-Tate}. Since the section $U\mapsto \c\Tate^{\geq m}_{\epsilon\scr L_U}$ of $\QCoh$ is right Kan extended from affines by definition, we deduce (i) in general.
	Since $j^*$ is a duality-preserving Karoubi projection, (i) implies (ii).
	
	If $m\in \Z$, we have an isomorphism
	\[(\Perf(U),\c\Qoppa^{\geq m}_{\epsilon\scr L})\simeq (\Perf(U),\Qoppa^\c_{\scr L[s]})\] 
	for some $s\in\Z$ (see Definition~\ref{def:derivedGW}). Since $\Qoppa^\c_{\scr L}=\map(\Sym^2(\ph),\scr L)$ by Proposition~\ref{prop:Sym^n}, it is clear that the functor $U\mapsto (\Perf(U),\Qoppa^\c_{\scr L})$ is an fpqc hypersheaf. Furthermore, since $\c\Tate^{\geq m}_{\epsilon\scr L}$ is $m$-connective, we have $\c\Qoppa^{\geq \infty}_{\epsilon\scr L}=\lim_{m\to\infty}\c\Qoppa^{\geq m}_{\epsilon\scr L}$, so we deduce fpqc hyperdescent for $m=\infty$ as well. 
	Étale hyperdescent for $m=-\infty$ now follows from the fiber sequence
	\[
	\Qoppa^\q_{\epsilon\scr L}=\c\Qoppa^{\geq \infty}_{\epsilon\scr L} \to \c\Qoppa^{\geq -\infty}_{\epsilon\scr L} \to \map(\ph,\c\Tate_{\epsilon\scr L})
	\]
	and the fact that $U\mapsto \c\Tate_{\epsilon\scr L_U}$ satisfies étale hyperdescent by (i).
	This proves (iii). 
	
	To prove (iv), it suffices as in Proposition~\ref{prop:mod-descent} to show that, for any étale map $Y\to X$ in $\dSpc^\qcqs$, the induced map
	\[
	(\Perf(X),\Qoppa^\c)\otimes_{(\Perf(X),\Qoppa^{\t})} (\Perf(Y),\Qoppa^{\t}) \to (\Perf(Y),\Qoppa^\c)
	\]
	is an isomorphism. But this follows from (i).
\end{proof}

\begin{corollary}[Nisnevich descent for classical $\scr F$-theory]
	\label{cor:animated-nis}
	Let $S$ be a qcqs derived algebraic space, $\scr L\in \Pic^\dagger(S)$, $\epsilon\in\{\pm 1\}$, and $m\in \Z\cup\{\pm\infty\}$. Then the functor
	\[
	(\dSpc_S^{\qcqs})^\op \to \Cat^\p,\quad X \mapsto (\Perf(X),\c\Qoppa^{\geq m}_{\epsilon\scr L}),
	\]
	sends Nisnevich squares to Poincaré–Karoubi squares. Hence, for any category $\scr S$ and Karoubi-localizing invariant $\scr F\colon \Mod_{(\Perf(S),\Qoppa^\c)}(\Cat^\p)\to \scr S$, the presheaf $\c\scr F^{\geq m}_{\epsilon\scr L}$ on $\dSpc_S^{\qcqs}$ is a Nisnevich sheaf.
\end{corollary}

\begin{proof}
	This follows immediately from Proposition~\ref{prop:descent-animated}(ii,iii).
\end{proof}

\begin{proposition}[Comparison between genuine and classical Poincaré structures]
	\label{prop:genuine-vs-derived}
	Let $X$ be a qcqs derived algebraic space, $\scr L\in \Pic^\dagger(X)$, $m\in\Z\cup\{\pm\infty\}$ and $\epsilon\in\{\pm 1\}$.
	The canonical map
	\[
	\c\Qoppa^{\geq m}_{\epsilon\scr L}\to \Qoppa^{\geq m}_{\epsilon\scr L}\colon \Perf(X)^\op\to\Sp
	\]
	is an isomorphism in the following cases:
	\begin{enumerate}
		\item $X$ is classical and $2$-torsionfree.
		\item $X$ is classical and either $m=\pm\infty$ or $\epsilon=(-1)^{m+\deg\scr L}$.
		\item $X$ is bounded and $m=-\infty$.
		\item $X$ is arbitrary and $m=\infty$.
	\end{enumerate}
\end{proposition}

\begin{proof}
	It suffices to consider the map on linear parts $\c\Tate^{\geq m}_{\epsilon\scr L}\to \Tate^{\geq m}_{\epsilon\scr L}$.
	The claims are Nisnevich-local on $X$ by Propositions \ref{prop:etale-descent} and~\ref{prop:descent-animated}(iii), so we can assume $X=\Spec(A)$. Cases (i) and (ii) follow from Remark~\ref{rmk:dTate}(i) by comparing filtrations. Case (iii) follows from Remark~\ref{rmk:dTate}(ii), since $(\epsilon(\ph))^{\t\C_2}=(\ph)\otimes_\Z(\epsilon \Z)^{\t\C_2}$ on bounded above $\Z$-modules. Case (iv) is clear since $\Tate^{\geq\infty}_{\epsilon\scr L}=\c\Tate^{\geq\infty}_{\epsilon\scr L}=0$ by definition.
\end{proof}

\begin{remark}
	The boundedness assumption in Proposition~\ref{prop:genuine-vs-derived}(iii) cannot be removed. Indeed, by Proposition~\ref{prop:descent-animated}(ii) and Remark~\ref{rmk:bounded}, there exists an animated ring $A$ for which the canonical map $\Qoppa^\varsigma_A\to\Qoppa^\s_A$ is not an isomorphism.
\end{remark}

\section{The projective bundle formula and smooth blowup excision}
\label{sec:PBF}

We recall the isotropic decomposition theorem \cite[Theorem 3.2.10]{hermitianII}. Given a Poincaré category $(\scr C,\Qoppa)$, a full stable subcategory $\scr L\subset\scr C$ is called \emph{isotropic} if $\Qoppa$ vanishes on $\scr L$ and the composite
\[
\scr L^\op\into\scr C^\op \xrightarrow{\D_\Qoppa}\scr C \onto \scr C/\scr L^\bot
\]
is an isomorphism. The \emph{homology} of $\scr L$ is then the Poincaré category
\[
\Hlgy(\scr L)=(\scr L^\bot\cap \mathrm{D}_\Qoppa(\scr L^\bot),\Qoppa),
\]
and the isotropic decomposition theorem states that, for any grouplike additive invariant $\scr F$ on $\Cat^\p$, the Poincaré functors $\Hyp(\scr L)\to (\scr C,\Qoppa)$ and $\Hlgy(\scr L)\into (\scr C,\Qoppa)$ induce an isomorphism
\[
\scr F(\Hyp(\scr L)) \times\scr F(\Hlgy(\scr L))\simto \scr F(\scr C,\Qoppa).
\]
The subcategory $\scr L$ is called a \emph{Lagrangian} if $\Hlgy(\scr L)=0$, in which case $\scr F(\scr C,\Qoppa)=\scr F(\Hyp(\scr L))$.

\begin{proposition}[Decomposing Poincaré categories of projective bundles]
	\label{prop:beilinson}
	Let $X\colon\CAlg^\an\to\An$ be a functor, $\scr E$ a finite locally free sheaf of rank $r+1$ on $X$, and $p\colon \P_X(\scr E)\to X$ the associated projective bundle.
	Let $\scr L$ be an invertible sheaf with involution on $X$, let $n\in \Z$, and let $\Qoppa\colon \Perf(\P_X(\scr E))^\op\to\Sp$ be a Poincaré structure with symmetric bilinear part $\B_{p^*(\scr L)(n)}$. Consider the $\Perf(X)$-submodules
	\begin{align*}
	\scr C&=\langle\scr O(\lceil\tfrac{n-r}2\rceil), \dotsc, \scr O(\lfloor \tfrac{n+r}2\rfloor)\rangle\subset \Perf(\P_X(\scr E)),\\
	\scr C_+&=\langle\scr O(\lceil\tfrac{n+1}2\rceil), \dotsc, \scr O(\lfloor \tfrac{n+r}2\rfloor)\rangle\subset\scr C.
	\end{align*}
	\begin{enumerate}
		\item The hermitian category $(\scr C,\Qoppa)$ is Poincaré. Moreover:
		\begin{enumerate}
			\item[(a)] If $n+r$ is even, then $\scr C=\Perf(\P_X(\scr E))$.
			\item[(b)] If $n+r$ is odd, there is a split Poincaré–Verdier sequence
		\[
		(\scr C,\Qoppa) \into (\Perf(\P_X(\scr E)),\Qoppa) \xrightarrow{p_*(\ph\otimes\scr O(-\frac{n+r+1}2))} (\Perf(X),\bar\Qoppa),
		\]
		where $\bar\Qoppa=\Qoppa\circ(p^*(\ph)\otimes\scr O(\tfrac{n+r+1}2))$.
		\end{enumerate}
		\item $\scr C_+$ is an isotropic subcategory of $(\scr C,\Qoppa^\s_{p^*(\scr L)(n)})$. If $\scr C_+$ is isotropic for $\Qoppa$, i.e., if $\Lambda_\Qoppa(\scr C_+)=0$, then:
		\begin{enumerate}
			\item[(a)] If $n$ is odd, then $\Hlgy(\scr C_+)=0$, i.e., $\scr C_+$ is a Lagrangian in $(\scr C,\Qoppa)$.
			\item[(b)] If $n$ is even, the homology inclusion $\Hlgy(\scr C_+)\into (\scr C,\Qoppa)$ can be identified with
			\[
			p^*(\ph)\otimes\scr O(\tfrac n2)\colon (\Perf(X),\Qoppa') \into (\scr C,\Qoppa),
			\]
			where $\Qoppa'=\Qoppa\circ(p^*(\ph)\otimes\scr O(\tfrac n2))$.
		\end{enumerate}
	\end{enumerate}
\end{proposition}

\begin{proof}
This is essentially contained in \cite[\sect 6.1]{CHN}.
We first recall some basic facts about projective bundles. 
For each $i\in \Z$, the functor $p^*(\ph)\otimes \scr O(i)\colon \Perf(X)\to \Perf(\P_X(\scr E))$ is fully faithful with right adjoint $p_*(\ph\otimes \scr O(-i))$ (in particular, $p_*$ preserves perfect quasi-coherent sheaves).
Moreover, $\Perf(\P_X(\scr E))$ is generated as a $\Perf(X)$-module by any $r+1$ successive Serre twists $\scr O(i)$, and we have $\map(\scr O(i),\scr O(j))=0$ whenever $i\in[j+1,j+r]$ \cite[Theorem B]{KhanBlow}.

(i) It is clear that $\scr C=\Perf(\P_X(\scr E))$ when $n+r$ is even.
The subcategory $\scr C$ is preserved by the duality $\mathrm{D}_\Qoppa=\Hom(\ph,p^*(\scr L)(n))$, since it sends $\scr O(i)$ to $p^*(\scr L)(n-i)$, so that $(\scr C,\Qoppa)$ is Poincaré.
 If $n+r$ is odd, there is a split Verdier sequence
	\[
	\scr C\into \Perf(\P_X(\scr E)) \xrightarrow{p_*(\ph\otimes\scr O(-\frac{n+r+1}2))} \Perf(X).
	\]
	Noting that $\bar\Qoppa$ is the left Kan extension of $\Qoppa$ along $p_*(\ph\otimes\scr O(-\frac{n+r+1}2))$, we obtain the desired split Poincaré–Verdier sequence.
	
(ii) For any $\scr M\in\scr C_+$ and $\scr N\in\scr C$, we have
 	\[
	\map(\scr M\otimes \scr N,p^*(\scr L)(n))=0
	\]
	if and only if $\scr M\otimes\scr N\in\langle \scr O(n+1),\dotsc, \scr O(n+r)\rangle$. Thus,
	\[
	\scr C_+^\bot =\langle\scr O(\lfloor\tfrac{n+1}2\rfloor), \dotsc, \scr O(\lfloor \tfrac{n+r}2\rfloor)\rangle.
	\]
	It follows that $\scr C_+\subset \scr C_+^\bot$ and that the composite
 \[
 \scr C_+^\op \into \scr C^\op \xrightarrow{\D_\Qoppa} \scr C\onto \scr C/\scr C_+^\bot
 \]
 is an isomorphism, so that $\scr C_+$ is an isotropic subcategory of $(\scr C,\Qoppa^\s_{p^*(\scr L)(n)})$.
 
	Assuming that $\scr C_+$ is isotropic in $(\scr C,\Qoppa)$, we have $\Hlgy(\scr C_+)=(\scr C_+^\bot\cap \mathrm{D}_\Qoppa(\scr C_+^\bot),\Qoppa)$. 
	Thus, if $n$ is odd, the homology vanishes, whereas if $n$ is even, the homology is $(\langle \scr O(\frac n2)\rangle,\Qoppa)$. Finally, the latter is isomorphic to $(\Perf(X),\Qoppa')$ via $p^*(\ph)\otimes\scr O(\tfrac n2)$.
\end{proof}

We will see that Proposition~\ref{prop:beilinson} applies to the genuine Poincaré structures $\Qoppa^{\geq m}$ as well as the classical versions $\c\Qoppa^{\geq m}$.
The following theorem generalizes \cite[Theorem 6.1.5]{CHN}, which is the case $m=-\infty$ and $X$ classical, as well as \cite[Proposition 6.2.2]{CHN}, which is the case $\scr E=\scr O^2$, $n=0$, and $X$ classical.

\begin{theorem}[Genuine projective bundle formula]
	\label{thm:PBF}
	Let $X$ be a qcqs derived algebraic space, $\scr E$ a locally free sheaf of rank $r+1$ on $X$, $\scr L$ an invertible sheaf with involution on $X$, $n\in\Z$, and $m\in\Z\cup\{\pm\infty\}$.
	Then Proposition~\ref{prop:beilinson} applies with
	\[
		\Qoppa=\Qoppa^{\geq m}_{p^*(\scr L)(n)},\quad
		\Qoppa'=\Qoppa^{\geq m}_{\scr L},\quad
		\bar\Qoppa=\Qoppa^{\geq m-r}_{\scr L\otimes\det(\scr E)^\vee[-r]}.
	\]
	Hence, the following holds for any grouplike additive invariant $\scr F$ on $\Mod_{(\Perf(X),\Qoppa^\g)}(\Cat^\p)$:
	\begin{enumerate}
		\item If both $n$ and $r$ are odd, the functors $p^*(\ph)\otimes\scr O(i)$ for $i\in[\frac{n+1}2,\frac{n+r}2]$ induce an isomorphism
		\[
		\scr F^\mathrm{hyp}(X)^{\oplus(r+1)/2}\simto \scr F^{\geq m}_{p^*(\scr L)(n)}(\P_X(\scr E)).
		\]
		\item If both $n$ and $r$ are even, the functors $p^*(\ph)\otimes\scr O(i)$ for $i\in[\frac{n}2,\frac{n+r}2]$ induce an isomorphism
		\[
		\scr F^{\geq m}_{\scr L}(X)\oplus\scr F^\mathrm{hyp}(X)^{\oplus r/2}\simto \scr F^{\geq m}_{p^*(\scr L)(n)}(\P_X(\scr E)).
		\]
		\item If $n$ is odd and $r$ is even, there is a fiber sequence
		\[
		\scr F^\mathrm{hyp}(X)^{\oplus r/2} \to \scr F^{\geq m}_{p^*(\scr L)(n)}(\P_X(\scr E)) \to \scr F^{\geq m-r}_{\scr L\otimes\det(\scr E)^\vee[-r]}(X),
		\]
		where the first map is induced by the functors $p^*(\ph)\otimes\scr O(i)$ for $i\in[\frac{n+1}2,\frac{n+r-1}2]$ and the second map is induced by the functor $p_*(\ph\otimes\scr O(-\tfrac{n+r+1}2))$.
		\item If $n$ is even and $r$ is odd, there is a fiber sequence
		\[
		\scr F^{\geq m}_{\scr L}(X)\oplus\scr F^\mathrm{hyp}(X)^{\oplus (r-1)/2} \to \scr F^{\geq m}_{p^*(\scr L)(n)}(\P_X(\scr E)) \to \scr F^{\geq m-r}_{\scr L\otimes\det(\scr E)^\vee[-r]}(X),
		\]
		where the first map is induced by the functors $p^*(\ph)\otimes\scr O(i)$ for $i\in[\frac{n}2,\frac{n+r-1}2]$ and the second map is induced by the functor $p_*(\ph\otimes\scr O(-\tfrac{n+r+1}2))$.
	\end{enumerate}
\end{theorem}

\begin{proof}
By Proposition~\ref{prop:beilinson} and the isotropic decomposition theorem, it suffices to prove the following three claims (where $\scr C_+\subset\scr C\subset\Perf(\P_X(\scr E))$ are defined as in Proposition~\ref{prop:beilinson}):
\begin{enumerate}
	\item[(a)] The linear part of $\Qoppa^{\geq m}_{p^*(\scr L)(n)}$ vanishes on $\scr C_+$.
	\item[(b)] If $n+r$ is odd, then
	\[
	\Qoppa^{\geq m-r}_{\scr L\otimes\det(\scr E)^\vee[-r]} = \Qoppa^{\geq m}_{p^*(\scr L)(n)} \circ \left(p^*(\ph)\otimes\scr O(\tfrac{n+r+1}2)\right).
	\]
	\item[(c)] If $n$ is even, then
	\[
	\Qoppa^{\geq m}_{\scr L} = \Qoppa^{\geq m}_{p^*(\scr L)(n)} \circ \left(p^*(\ph)\otimes\scr O(\tfrac n2)\right).
	\]
\end{enumerate}

For (a), we must show that
\[
\map(\scr M,\tau_{\geq m}\Tate_{p^*(\scr L)(n)})=0
\]
for any $\scr M\in\scr C_+$. This holds for $m=-\infty$ since $\scr C_+$ is isotropic for $\Qoppa^\s_{p^*(\scr L)(n)}$.
Hence, it suffices to show that the t-structure on $\QCoh(\P_X(\scr E))$ preserves the kernel of the functors $\map(\scr M,\ph)$ for $\scr M\in\scr C_+$. This kernel is a $\QCoh(X)$-submodule generated by suitable $\scr O(i)$'s, so the claim follows since $p^*$ and Serre twisting are t-exact.

We now prove (b). Since $p_*\scr O(-r-1)=\det(\scr E)^\vee[-r]$, we have $\C_2$-equivariant isomorphisms
\begin{align*}
&\map(p^*(\scr M)\otimes\scr O(\tfrac{n+r+1}2) \otimes p^*(\scr M)\otimes\scr O(\tfrac{n+r+1}2),p^*(\scr L)(n))\\
=\:&\map(p^*(\scr M\otimes\scr M), p^*\scr L\otimes \scr O(-r-1))\\
=\:&\map(\scr M\otimes\scr M,\scr L\otimes p_*\scr O(-r-1))\\
=\:&\map(\scr M\otimes\scr M,\scr L\otimes \det(\scr E)^\vee[-r]).
\end{align*}
(In the first isomorphism, we used that the $\C_2$-action on any square in $\Pic$ of a derived algebraic space is trivial, and in the second isomorphism we used the projection formula.)
Applying the $\C_2$-Tate construction, we get
\[
\map(p^*(\scr M)\otimes\scr O(\tfrac{n+r+1}2),\Tate_{p^*(\scr L)(n)}) = \map(\scr M,\Tate_{\scr L\otimes \det(\scr E)^\vee[-r]}),
\]
that is, 
\[
\Tate_{\scr L\otimes \det(\scr E)^\vee[-r]}=p_*(\Tate_{p^*(\scr L)(n)}\otimes \scr O(-\tfrac{n+r+1}2)).
\] 
We have to show
\[
\map(p^*(\scr M)\otimes\scr O(\tfrac{n+r+1}2),\tau_{\geq m}\Tate_{p^*(\scr L)(n)}) = \map(\scr M,\tau_{\geq m-r}\Tate_{\scr L\otimes \det(\scr E)^\vee[-r]}),
\]
or equivalently
\[
\tau_{\geq m-r}p_*(\Tate_{p^*(\scr L)(n)}\otimes \scr O(-\tfrac{n+r+1}2)) = p_*(\tau_{\geq m}\Tate_{p^*(\scr L)(n)}\otimes \scr O(-\tfrac{n+r+1}2)).
\]
We have $\map(\scr O(i),\Tate_{p^*(\scr L)(n)})=0$ for $2i\in [n+1,n+r]$, i.e., for $i\in[\lceil\tfrac{n+1}2\rceil,\tfrac{n+r-1}2]$, so that 
\[
\Tate_{p^*(\scr L)(n)}\in\langle \scr O(\tfrac{n-r-1}2),\dotsc,\scr O(\lfloor\tfrac n2\rfloor)\rangle
\]
and hence
\[
\Tate_{p^*(\scr L)(n)}\otimes\scr O(-\tfrac{n+r+1}2)\in\langle \scr O(-r-1),\dotsc,\scr O(\lfloor\tfrac {-r-1}2\rfloor)\rangle.
\]
Thus, it suffices to check that $p_*$ is t-exact with a shift of $-r$ on this subcategory, which holds as $p_*\scr O(-r-1)=\det(\scr E)^\vee[-r]$ and $p_*\scr O(-i)=0$ for $i\in [1,r]$.
 
The proof of (c) is similar. Since $p_*\scr O=\scr O$, we have $\C_2$-equivariant isomorphisms
\begin{align*}
	&\map(p^*(\scr M)\otimes\scr O(\tfrac n2)\otimes p^*(\scr M)\otimes\scr O(\tfrac n2),p^*(\scr L)(n))\\
	=\:&\map(p^*(\scr M\otimes\scr M),p^*(\scr L))\\
	=\:&\map(\scr M\otimes\scr M,\scr L\otimes p_*\scr O)\\
	=\:&\map(\scr M\otimes\scr M,\scr L).
\end{align*}
Applying the $\C_2$-Tate construction, we find
\[
\Tate_\scr L = p_*(\Tate_{p^*(\scr L)(n)}\otimes\scr O(-\tfrac n2)).
\]
We thus have to show 
\[
\tau_{\geq m}p_*(\Tate_{p^*(\scr L)(n)}\otimes\scr O(-\tfrac n2)) = p_*(\tau_{\geq m}\Tate_{p^*(\scr L)(n)}\otimes\scr O(-\tfrac n2)).
\]
We have $\map(\scr O(i),\Tate_{p^*(\scr L)(n)})=0$ for $2i\in [n+1,n+r]$, i.e., for $i\in[\tfrac {n+2}2,\lfloor\tfrac{n+r}2\rfloor]$, so that
\[
\Tate_{p^*(\scr L)(n)} \in \langle\scr O(\lfloor \tfrac{n-r}2\rfloor),\dotsc,\scr O(\tfrac n2)\rangle
\]
and hence
\[
\Tate_{p^*(\scr L)(n)}\otimes\scr O(-\tfrac n2) \in \langle\scr O(\lfloor \tfrac{-r}2\rfloor),\dotsc,\scr O(0)\rangle.
\]
Thus, it suffices to check that $p_*$ is t-exact on this subcategory, which holds as $p_*\scr O=\scr O$ and $p_*\scr O(-i)=0$ for $i\in [1,r]$.
\end{proof}

\begin{corollary}\label{cor:ebu0}
	Let $X$ be a qcqs derived algebraic space, $\scr E$ a finite locally free sheaf on $X$, $\scr L$ an invertible sheaf with involution on $X$, and $m\in\Z\cup\{\pm\infty\}$.
	Let $\scr F\colon\Mod_{(\Perf(X),\Qoppa^\g)}(\Cat^\p)\to\scr S$ be an additive invariant valued in a stable category $\scr S$. Then $\scr F^{\geq m}_{\scr L}$ sends the blowup square
	\[
	\begin{tikzcd}
		\P_X(\scr E) \ar{r} \ar{d} & \Bl_X(\P_X(\scr E\oplus\scr O)) \ar{d} \\
		X \ar{r} & \P_X(\scr E\oplus\scr O)
	\end{tikzcd}
	\]
	to a cartesian square.
\end{corollary}

\begin{proof}
	Since $\scr S$ is stable, it suffices to check that $\scr F^{\geq m}_{\scr L}$ sends the square
	\[
	\begin{tikzcd}
		\P_{\P(\scr E)}(\scr O(1)\oplus\scr O) \ar{r}{q} \ar{d}[swap]{b} & \P_X(\scr E) \ar{d} \\
		\P_X(\scr E\oplus\scr O) \ar{r} & X
	\end{tikzcd}
	\]
	to a cartesian square. By Theorem~\ref{thm:PBF}(iv), the cofiber of $q^*$ is
	\[
	\scr F^{\geq m-1}_{\scr L(-1)[-1]}(\P_X(\scr E)),
	\]
	and we have to show that the Poincaré functor $\Phi=q_*(b^*(\ph)\otimes\scr O_q(-1))$ induces an isomorphism
	\[
	\scr F^{\geq m}_{\scr L}(\P_X(\scr E\oplus\scr O))/\scr F^{\geq m}_{\scr L}(X) \simto \scr F^{\geq m-1}_{\scr L(-1)[-1]}(\P_X(\scr E)).
	\]
	If $r$ is even, this follows from Theorem~\ref{thm:PBF}(i,ii), and if $r$ is odd, it follows from Theorem~\ref{thm:PBF}(iii,iv).
	
	In more details, note that $b^*$ sends $\scr O(i)$ to $\scr O_q(i)$, and hence
	\[
	\Phi(\scr O(i))=\begin{cases}
		0, & \text{if $i=0$,} \\
		\Sym^{i-1}(\scr O(1)\oplus\scr O)=\scr O\oplus\dotsb\oplus\scr O(i-1), & \text{if $i\geq 1$.}
	\end{cases}
	\]
	Consider as in Proposition~\ref{prop:beilinson} the $\Perf(X)$-submodules 
	\begin{align*}
		\scr C&=\langle\scr O(\lceil -\tfrac{r}2\rceil),\dotsc,\scr O(\lfloor \tfrac r2\rfloor)\rangle \subset \Perf(\P_X(\scr E\oplus\scr O)), \\
		\scr C_+&=\langle\scr O(1),\dotsc,\scr O(\lfloor \tfrac r2\rfloor)\rangle \subset\scr C, \\
		\scr D&=\langle \scr O(\lceil -\tfrac{r}2\rceil),\dotsc,\scr O(\lfloor \tfrac {r-2}2\rfloor)\rangle \subset\Perf(\P_X(\scr E)), \\
		\scr D_+&=\langle \scr O(0),\dotsc,\scr O(\lfloor \tfrac {r-2}2\rfloor)\rangle\subset\scr D.
	\end{align*}
	Then $\scr C_+$ and $\scr D_+$ are Lagrangians in $(\scr C/\Perf(X),\Qoppa^{\geq m}_{\scr L})$ and $(\scr D,\Qoppa^{\geq m-1}_{\scr L(-1)[-1]})$, respectively. Moreover, by the above computation of $\Phi$ (and the fact that $\Phi$ is $\Perf(X)$-linear and duality-preserving), $\Phi$ sends $\scr C$ to $\scr D$ and $\scr C_+$ to $\scr D_+$. Thus, $\Phi$ induces a morphism of (split) Poincaré–Verdier sequences
	\[
	\begin{tikzcd}
		 \scr C/\Perf(X) \ar{d} \ar[hook]{r} & \Perf(\P_X(\scr E\oplus\scr O))/\Perf(X) \ar{d}{\Phi} \ar{r} & \Perf(\P_X(\scr E\oplus\scr O))/\scr C \ar{d}[sloped]{\sim}  \\
	  \scr D \ar[hook]{r} & \Perf(\P_X(\scr E)) \ar{r} & \Perf(\P_X(\scr E))/\scr D\rlap.
	\end{tikzcd}
	\]
	Note that the right vertical map is an isomorphism: both quotients are zero when $r$ is even; when $r$ is odd, they are copies of $(\Perf(X),\Qoppa^{\geq m-r}_{\scr L\otimes\det(\scr E)^\vee[-r]})$ generated by $\scr O(\tfrac{r+1}2)$ and $\scr O(\tfrac{r-1}2)$, respectively, and we have
	\[
	\Phi(\scr O(\tfrac{r+1}2))\equiv \scr O(\tfrac{r-1}2)\text{ modulo }\scr D.
	\]
	 To conclude, it suffices to show that $\scr F$ sends $\Phi\colon\scr C/\Perf(X)\to\scr D$ to an isomorphism, or equivalently that $\scr F^\mathrm{hyp}$ sends $\Phi\colon \scr C_+\to \scr D_+$ to an isomorphism. This follows from the above formula for $\Phi$, as 
	 \[
	 \Phi(\scr O(i))\equiv\scr O(i-1)\text{ modulo }\langle\scr O(0),\dotsc,\scr O(i-2)\rangle.\qedhere
	 \]
\end{proof}

\begin{corollary}[Smooth blowup excision for genuine $\scr F$-theory]
	\label{cor:ebu}
	Let $\scr L$ be an invertible sheaf with involution on a qcqs derived algebraic space $S$ and let $m\in\Z\cup\{\pm\infty\}$. 
	If $m\neq\infty$, assume that $S$ is bounded.
	Let $\scr F\colon\Mod_{(\Perf(S),\Qoppa^\g)}(\Cat^\p)\to\scr S$ be a Karoubi-localizing invariant valued in a stable category $\scr S$.
	Then the presheaf 
	\[\scr F^{\geq m}_{\scr L}\colon (\Sm_S^\fp)^\op\to \scr E\]
	satisfies smooth blowup excision.
\end{corollary}

\begin{proof}
By Corollary~\ref{cor:nis-descent}, $\scr F^{\geq m}_{\scr L}$ is a Nisnevich sheaf. The statement then follows from Corollary~\ref{cor:ebu0} by virtue of \cite[Proposition 2.2]{AHI}.
\end{proof}

\begin{remark}[Projective bundle formula and smooth blowup excision for Tate $\scr F$-theory]
	\label{rmk:Tate-PBF}
	Let $X$ be a qcqs derived algebraic space, let $\sqrt\scr L\in \Pic^\dagger(S)$, and let $\scr L=(\sqrt\scr L)^{\otimes 2}\in\Pic^\dagger(S)^{\B\C_2}$.
	If $n$ is even, we can consider the Tate Poincaré structure $\Qoppa^\t_{p^*(\scr L)(n)}$ on $\Perf(\P_X(\scr E))$ as in Remark~\ref{rmk:Tate-descent}. 
	An analysis as in the proof of Theorem~\ref{thm:PBF} (but simpler) shows that we can apply Proposition~\ref{prop:beilinson} with 
	\[
	\Qoppa=\Qoppa^\t_{p^*(\scr L)(n)},\quad \Qoppa'=\Qoppa^\t_{\scr L},\quad \bar\Qoppa=\Qoppa^\q_{\scr L\otimes\det(\scr E)^\vee[-r]}.
	\]
	Hence, we get a projective bundle formula for $\scr F^{\t}_{p^*(\scr L)(n)}(\P_X(\scr E))$ similar to cases (ii) and (iv) of Theorem~\ref{thm:PBF}.
	This implies as in Corollary~\ref{cor:ebu0} that, for any stable additive invariant $\scr F$, $\scr F^\t_{\scr L}$ sends the blowup square of the zero section of $\P_X(\scr E\oplus\scr O)$ to a cartesian square.
	If $\scr F$ is moreover Karoubi-localizing, it follows as in Corollary~\ref{cor:ebu} that $\scr F^{\t}_{\scr L}$ satisfies smooth blowup excision on $\Sm_X^\fp$.
	Note that the above projective bundle formula shows that the $\P^1$-loop space of Tate $\scr F$-theory $\scr F^{\t}_{\scr L}$ is quadratic $\scr F$-theory $\scr F^\q_{\scr L[-1]}$.
	We do not know if $\scr F^{\t}_{\scr L}$ admits a $\P^1$-delooping.
\end{remark}

We now turn our attention to the classical Poincaré structures $\c\Qoppa^{\geq m}$.

\begin{theorem}[Classical projective bundle formula]
	\label{thm:derived-PBF}
	Let $X$ be a qcqs derived algebraic space, $\scr E$ a locally free sheaf of rank $r+1$ on $X$, $\scr L\in\Pic^\dagger(X)$, $\epsilon\in\{\pm 1\}$, $n\in\Z$, and $m\in\Z\cup\{\pm\infty\}$.
	Then Proposition~\ref{prop:beilinson} applies with
	\[
		\Qoppa=\c\Qoppa^{\geq m}_{\epsilon p^*(\scr L)(n)},\quad
		\Qoppa'=\c\Qoppa^{\geq m}_{\epsilon \scr L},\quad
		\bar\Qoppa=\c\Qoppa^{\geq m-r}_{\epsilon \scr L\otimes\det(\scr E)^\vee[-r]}.
	\]
	Hence, the following holds for any grouplike additive invariant $\scr F$ on $\Mod_{(\Perf(X),\Qoppa^\c)}(\Cat^\p)$:
	\begin{enumerate}
		\item If both $n$ and $r$ are odd, the functors $p^*(\ph)\otimes\scr O(i)$ for $i\in[\frac{n+1}2,\frac{n+r}2]$ induce an isomorphism
		\[
		\scr F^\mathrm{hyp}(X)^{\oplus(r+1)/2}\simto \c\scr F^{\geq m}_{\epsilon p^*(\scr L)(n)}(\P_X(\scr E)).
		\]
		\item If both $n$ and $r$ are even, the functors $p^*(\ph)\otimes\scr O(i)$ for $i\in[\frac{n}2,\frac{n+r}2]$ induce an isomorphism
		\[
		\c\scr F^{\geq m}_{\epsilon \scr L}(X)\oplus\scr F^\mathrm{hyp}(X)^{\oplus r/2}\simto \c\scr F^{\geq m}_{\epsilon p^*(\scr L)(n)}(\P_X(\scr E)).
		\]
		\item If $n$ is odd and $r$ is even, there is a fiber sequence
		\[
		\scr F^\mathrm{hyp}(X)^{\oplus r/2} \to \c\scr F^{\geq m}_{\epsilon p^*(\scr L)(n)}(\P_X(\scr E)) \to \c\scr F^{\geq m-r}_{\epsilon \scr L\otimes\det(\scr E)^\vee[-r]}(X),
		\]
		where the first map is induced by the functors $p^*(\ph)\otimes\scr O(i)$ for $i\in[\frac{n+1}2,\frac{n+r-1}2]$ and the second map is induced by the functor $p_*(\ph\otimes\scr O(-\tfrac{n+r+1}2))$.
		\item If $n$ is even and $r$ is odd, there is a fiber sequence
		\[
		\c\scr F^{\geq m}_{\epsilon \scr L}(X)\oplus\scr F^\mathrm{hyp}(X)^{\oplus (r-1)/2} \to \c\scr F^{\geq m}_{\epsilon p^*(\scr L)(n)}(\P_X(\scr E)) \to \c\scr F^{\geq m-r}_{\epsilon \scr L\otimes\det(\scr E)^\vee[-r]}(X),
		\]
		where the first map is induced by the functors $p^*(\ph)\otimes\scr O(i)$ for $i\in[\frac{n}2,\frac{n+r-1}2]$ and the second map is induced by the functor $p_*(\ph\otimes\scr O(-\tfrac{n+r+1}2))$.
	\end{enumerate}
\end{theorem}

\begin{proof}
	By Proposition~\ref{prop:beilinson} and the isotropic decomposition theorem, it suffices to prove the following three claims (where $\scr C_+\subset\scr C\subset\Perf(\P_X(\scr E))$ are defined as in Proposition~\ref{prop:beilinson}):
	\begin{enumerate}
		\item[(a)] $\c\Qoppa^{\geq m}_{\epsilon p^*(\scr L)(n)}$ vanishes on $\scr C_+$.
		\item[(b)] If $n+r$ is odd, then
		\[
		\c\Qoppa^{\geq m-r}_{\epsilon\scr L\otimes\det(\scr E)^\vee[-r]} = \c\Qoppa^{\geq m}_{\epsilon p^*(\scr L)(n)} \circ \left(p^*(\ph)\otimes\scr O(\tfrac{n+r+1}2)\right).
		\]
		\item[(c)] If $n$ is even, then
		\[
		\c\Qoppa^{\geq m}_{\epsilon\scr L} = \c\Qoppa^{\geq m}_{\epsilon p^*(\scr L)(n)} \circ \left(p^*(\ph)\otimes\scr O(\tfrac n2)\right).
		\]
	\end{enumerate}
	Since these statements are known for the bilinear parts, they reduce to the following statements about the linear parts:
	\begin{enumerate}
		\item[(a)]If $i\in[\lceil \frac{n+1}2\rceil, \lfloor \frac{n+r}2\rfloor]$, then
		\[p_*\bigl(\c\Tate^{\geq m}_{\epsilon p^*(\scr L)(n)} \otimes\scr O(i)\bigr)=0.\]
		\item[(b)] If $n+r$ is odd, then
		\[\c\Tate^{\geq m-r}_{\epsilon\scr L\otimes\det(\scr E)^\vee[-r]}=p_*\bigl(\c\Tate^{\geq m}_{\epsilon p^*(\scr L)(n)}\otimes \scr O(-\tfrac{n+r+1}2)\bigr).\]
		\item[(c)] If $n$ is even, then
		\[\c\Tate^{\geq m}_{\epsilon\scr L}=p_*\bigl(\c\Tate^{\geq m}_{\epsilon p^*(\scr L)(n)}\otimes \scr O(-\tfrac n2)\bigr).\]
	\end{enumerate}
	By Proposition~\ref{prop:descent-animated}(i), these claims are étale-local on $X$, so that we can assume $X=\Spec(A)$. As both $\Vect_{r+1}$ and $\Pic^\dagger$ are left Kan extended from smooth $\Z$-algebras \cite[Proposition A.0.4]{EHKSY3}, we may assume that $A$ is an algebra over some smooth $\Z$-algebra $R$, over which $\scr E$ and $\scr L$ are defined.
	All the $A$-modules above can then be viewed as functors $\CAlg^\an_R\to \Sp$ in the variable $A$. If $A$ is smooth over $R$, then the desired equations are known by Theorem~\ref{thm:PBF}. It thus suffices to observe that all these functors are left Kan extended from $\Poly_R$, since they preserve sifted colimits in $A$ by Remark~\ref{rmk:dTate}(iii) (to see this for the functors involving $p_*$, write $\P_R(\scr E)$ as a finite colimit of affine open subschemes and use Proposition~\ref{prop:descent-animated}(i)).
\end{proof}

\begin{corollary}\label{cor:animated-ebu0}
	Let $X$ be a qcqs derived algebraic space, $\scr E$ a finite locally free sheaf on $X$, $\scr L$ an invertible sheaf on $X$, $\epsilon\in\{\pm 1\}$, and $m\in\Z\cup\{\pm\infty\}$.
	Let $\scr F\colon\Mod_{(\Perf(X),\Qoppa^\c)}(\Cat^\p)\to\scr S$ be an additive invariant valued in a stable category $\scr S$. Then $\c\scr F^{\geq m}_{\epsilon\scr L}$ sends the blowup square
	\[
	\begin{tikzcd}
		\P_X(\scr E) \ar{r} \ar{d} & \Bl_X(\P_X(\scr E\oplus\scr O)) \ar{d} \\
		X \ar{r} & \P_X(\scr E\oplus\scr O)
	\end{tikzcd}
	\]
	to a cartesian square.
\end{corollary}

\begin{proof}
	Same as the proof of Corollary~\ref{cor:ebu0}, using Theorem~\ref{thm:derived-PBF}.
\end{proof}

\begin{corollary}[Smooth blowup excision for classical $\scr F$-theory]
	\label{cor:animated-ebu}
	Let $S$ be a qcqs derived algebraic space, $\scr L\in \Pic^\dagger(S)$, $\epsilon\in\{\pm 1\}$, and $m\in \Z\cup\{\pm\infty\}$.
	Let $\scr F\colon\Mod_{(\Perf(S),\Qoppa^\c)}(\Cat^\p)\to\scr S$ be a Karoubi-localizing invariant valued in a stable category $\scr S$.
	Then the presheaf 
	\[\c\scr F^{\geq m}_{\epsilon\scr L}\colon (\Sm_S^\fp)^\op\to \scr S\]
	satisfies smooth blowup excision.
\end{corollary}

\begin{proof}
	By Corollary~\ref{cor:animated-nis}, $\c\scr F^{\geq m}_{\epsilon\scr L}$ is a Nisnevich sheaf. The statement then follows from Corollary~\ref{cor:animated-ebu0} by virtue of \cite[Proposition 2.2]{AHI}.
\end{proof}

\section{Poincaré–Azumaya algebras}
\label{sec:Brauer}

The algebraic K-theory of schemes can naturally be twisted by Azumaya algebras.
We investigate the analogous theory of twists for GW-theory.
To that end, recall that the (extended) Picard and Brauer stacks are defined by
\begin{alignat*}{2}
	\Pic&=\Pic(\QCoh_{\geq 0}), &\qquad \Br&=\Pic(\QStk^\t), \\
	\Pic^\dagger&=\Pic(\QCoh),& \Br^\dagger&=\Pic(\QStk).
\end{alignat*}
Here, $\QStk$ and $\QStk^\t$ are the symmetric monoidal $2$-categories of (compactly generated) stable and t-structured quasi-coherent stacks, respectively \cite[\sect 11.5]{SAG}.
All four stacks are étale sheaves of $\E_\infty$-groups on spectral algebraic spaces, and they fit in the Picard–Brauer exact sequence
\[
\Pic\to \Pic^\dagger\xrightarrow{\deg} \Z \to \Br\to \Br^\dagger\xrightarrow{\deg} \B\Z.
\]
In particular, the Brauer stacks are the étale-local deloopings of the corresponding Picard stacks.
We will abusively refer to elements of $\Br^\dagger(X)$ as \emph{Azumaya algebras} over $X$.
The trivial Azumaya algebra over $X$ (that is, the unit of $\QStk(X)$) will be denoted by $\scr Q_X$.

\begin{construction}
	We can enlarge the $\E_\infty$-group $\Br$ of invertible t-structured quasi-coherent stacks $(\scr A,\scr A_{\geq 0})$ to a symmetric monoidal category $\Br_{\pm\infty}\subset\QStk^\t$ by allowing the degenerate t-structures $(\scr A,0)$ and $(\scr A,\scr A)$ as objects and right t-exact isomorphisms as morphisms. There is still a forgetful functor $\Br_{\pm\infty}\to\Br^\dagger$, whose kernel is now the poset $\Z_{\pm\infty}=(\Z\cup\{\pm\infty\},\geq)$ with the symmetric monoidal structure $+$ such that $(-\infty)+\infty=\infty$. We then have a categorical enhancement of the Picard–Brauer exact sequence
	\[
	\Pic^\dagger\xrightarrow{\deg} \Z_{\pm\infty} \to \Br_{\pm\infty}\to \Br^\dagger\xrightarrow{\deg} \B\Z_{\pm\infty},
	\]
	which exhibits $\Br_{\pm\infty}$ as the lax kernel of $\deg\colon \Br^\dagger\to \B\Z_{\pm\infty}$, or equivalently as the étale-local cokernel of $\deg\colon \Pic^\dagger\to \Z_{\pm\infty}$.
\end{construction}

\begin{definition}[The Poincaré–Brauer stack]
	 We define the \emph{Poincaré–Brauer stack} as the pullback
	\[\Br^\g\simeq \Br_{\pm\infty}\times_{\Br^\dagger}(\Br^\dagger)^{\h\C_2},\]
	where $\C_2$ acts on $\Br^\dagger$ by inversion; it is an étale sheaf of symmetric monoidal categories on spectral algebraic spaces.
	Objects of $\Br^\g(X)$ will be called \emph{Poincaré–Azumaya algebras} over $X$: they are triples $(\scr A,\scr A_{\geq 0},\phi)$, where $\scr A$ is an invertible stable quasi-coherent stack on $X$, $\scr A_{\geq 0}\subset\scr A$ is (locally) either $0$, $\scr A$, or an invertible prestable refinement of $\scr A$, and $\phi$ is a $\C_2$-equivariant isomorphism $(\scr A^\vee)^{\otimes 2}\simeq \scr Q_X$. A morphism $(\scr A,\scr A_{\geq 0},\phi)\to(\scr B,\scr B_{\geq 0},\psi)$ in $\Br^\g(X)$ is an isomorphism $(\scr A,\phi)\simeq(\scr B,\psi)$ in $(\Br^\dagger)^{\h\C_2}$ sending $\scr A_{\geq 0}$ to $\scr B_{\geq 0}$.
\end{definition}

\begin{remark}[The Poincaré–Brauer stack as a quotient]
	\label{rmk:Brp-quotient}
	Since $\Br^\g$ is the lax kernel of
	\[
	\Br^\dagger \to \B\Z_{\pm\infty}\times (\Br^\dagger)^{\B\C_2} ,\quad \scr A\mapsto(\deg(\scr A),\scr A^{\otimes 2}),
	\]
	it is also the étale-local cokernel of
	\[
	\Pic^\dagger \to \Z_{\pm\infty}\times (\Pic^\dagger)^{\B\C_2},\quad \scr L\mapsto (\deg(\scr L),\scr L^{\otimes 2}).
	\]
	The cokernel functor is explicitly given by
	\begin{equation}\label{eqn:Br-coker}
	\Z_{\pm\infty}\times (\Pic^\dagger)^{\B\C_2}\to \Br^\g,\quad (m,\scr L)\mapsto (\scr Q,\scr Q_{\geq m},\phi_{\scr L}),
	\end{equation}
	where $\phi_{\scr L}\colon (\scr Q^\vee)^{\otimes 2}\simto\scr Q$ is the symmetric bilinear functor $\B_{\scr L}$.
\end{remark}

\begin{construction}[Twists of genuine Poincaré categories]
	\label{cst:Br-to-Cat^p}
	Let $X$ be a qcqs spectral algebraic space. An element $(\scr A,\scr A_{\geq 0},\phi)\in \Br^\g(X)$ defines a Poincaré category $(\scr A(X)^\omega, \Qoppa^{\geq 0}_\phi)$, where the symmetric bilinear part of $\Qoppa^{\geq 0}_\phi$ is induced by $\phi$ and its linear part is induced by the t-structure (as in \cite[\sect 3.4]{CHN}).
	These assemble into a lax natural transformation
	\[
	\Br^\g\to \Mod_{(\Perf(\ph),\Qoppa^\g)}(\Cat^\p_\idem)\colon (\SpSpc^\qcqs)^\op\to \CAlg^\lax(\Cat),\quad (\scr A,\scr A_{\geq 0},\phi)\mapsto (\scr A(\ph)^\omega,\Qoppa^{\geq 0}_\phi),
	\]
	where $\CAlg^\lax(\Cat)$ is the $2$-category of symmetric monoidal categories and lax symmetric monoidal functors.
	Precomposing with~\eqref{eqn:Br-coker} gives the lax natural transformation 
	\[
	\Z_{\pm\infty}\times (\Pic^\dagger)^{\B\C_2}\to \Mod_{(\Perf(\ph),\Qoppa^\g)}(\Cat^\p_\idem),\quad (m,\scr L)\mapsto (\Perf(\ph),\Qoppa^{\geq m}_{\scr L}).
	\]
\end{construction}

\begin{remark}
	On bounded spaces, it follows from Lemma~\ref{lem:etale-Tate} and Nisnevich descent that the lax natural transformation $\Br^\g\to \Mod_{(\Perf(\ph),\Qoppa^\g)}(\Cat^\p_\idem)$ of Construction~\ref{cst:Br-to-Cat^p} is strict on étale morphisms.
	It is therefore uniquely determined by the $\Pic^\dagger$-invariant natural transformation $(m,\scr L)\mapsto (\Perf(\ph),\Qoppa^{\geq m}_{\scr L})$.
\end{remark}

\begin{variant}[The classical Poincaré–Brauer stack]
	\label{var:dBr^p}
	Let $\L\Br^\g\colon \dSpc^{\op}\to \CAlg(\Cat)$ be the étale-local left Kan extension of $\Br^\g$ from smooth $\Z$-schemes. To describe it more explicitly, recall that the functor $\Pic^\dagger\colon \CAlg^\an\to\An$ is left Kan extended from smooth $\Z$-algebras. Note further that $(\Pic^\dagger)^{\B\C_2}=\mu_2\times\Pic^\dagger$ on static rings. Since $\mu_2=\C_2$ on smooth $\Z$-algebras, the left Kan extension of $(\Pic^\dagger)^{\B\C_2}$ from smooth $\Z$-algebras is $\C_2\times\Pic^\dagger$.
	By Remark~\ref{rmk:Brp-quotient}, $\L\Br^\g$ is then the étale-local cokernel of
	\[
	\Pic^\dagger \to \Z_{\pm\infty}\times\C_2\times \Pic^\dagger,\quad \scr L\mapsto(\deg(\scr L),\mathrm{sgn}(\scr L),\scr L^{\otimes 2}),
	\]
	where $\mathrm{sgn}(\scr L)=(-1)^{\deg\scr L}$.
	This is equivalently the kernel of
	\[
	\Br_{\pm\infty} \to \B\C_2\times \Br^\dagger,\quad (\scr A,\scr A_{\geq 0})\mapsto (\mathrm{sgn}(\scr A),\scr A^{\otimes 2}).
	\]
	An element of $\L\Br^\g(X)$ is thus a triple $(\scr A,\scr A_{\geq 0},\phi)$, where $(\scr A,\scr A_{\geq 0})\in \Br_{\pm\infty}(X)$ and $\phi$ is an isomorphism $(\scr A^\vee)^{\otimes 2}\simeq\scr Q_X$ made $\C_2$-equivariant by a trivialization of the $\C_2$-torsor $\mathrm{sgn}(\scr A)$ (such a trivialization is a choice of prestable refinement of $\scr A$ up to even shifts). The cokernel functor is given by
	\begin{equation}\label{eqn:dBr-coker}
	\Z_{\pm\infty}\times\C_2\times \Pic^\dagger \to \L\Br^\g,\quad (m,\epsilon,\scr L)\mapsto (\scr Q,\scr Q_{\geq m},\phi_{\epsilon\scr L}).
	\end{equation}
	We denote by $\Br^\c\subset \L\Br^\g$ the symmetric monoidal full subcategory where the trivialization of $\mathrm{sgn}(\scr A)$ is induced by the t-structure $\scr A_{\geq 0}$, whenever the latter is a prestable refinement of $\scr A$.
	Equivalently, this is the quotient by $\Pic^\dagger$ of the full subcategory of $\Z_{\pm\infty}\times\C_2\times \Pic^\dagger$ consisting of triples $(m,\epsilon,\scr L)$ where either $m=\pm\infty$ or $\epsilon=(-1)^{m+\deg\scr L}$. The inclusion of this subcategory has a right adjoint sending any other triple $(m,\epsilon,\scr L)$ to $(m+1,\epsilon,\scr L)$, which is lax symmetric monoidal but strictly $\Pic^\dagger$-linear. It therefore induces, upon taking the quotient by $\Pic^\dagger$, a right Bousfield localization
	\[
	\begin{tikzcd}[column sep=2em]
		\Br^\c \ar[shift left=2pt,hook]{r} \ar[<-,shift right=2pt]{r} & \L\Br^\g\rlap.
	\end{tikzcd}
	\]
\end{variant}

\begin{variant}[Twists of classical Poincaré categories]
	\label{var:dBr-to-Cat^p}
	Consider the lax natural transformation
	\begin{equation*}\label{eqn:Br-to-Cat^p}
	\Br^\g\to \Mod_{(\Perf(\ph),\Qoppa^\g)}(\Cat^\p_\idem)\colon (\SpSpc^\qcqs)^\op\to \CAlg^\lax(\Cat),\quad (\scr A,\scr A_{\geq 0},\phi)\mapsto (\scr A(\ph)^\omega,\Qoppa^{\geq 0}_\phi),
	\end{equation*}
	from Construction~\ref{cst:Br-to-Cat^p}. We note that it is in fact strict on $2$-torsionfree reduced rings. Since the target is an étale sheaf by Proposition~\ref{prop:mod-descent}, one can check this after precomposing with~\eqref{eqn:Br-coker}; the claim is then that for any map of $2$-torsionfree reduced rings $A\to B$, any $L\in \Pic^\dagger(A)^{\B\C_2}$, and any $m\in \Z\cup\{\pm\infty\}$, the map 
	\[
	\tau_{\geq m}L^{\t\C_2}\otimes_{\tau_{\geq 0}A^{\t\C_2}}\tau_{\geq 0}B^{\t\C_2}\to \tau_{\geq m}(L\otimes_AB)^{\t\C_2}
	\]
	is an isomorphism, which is clear.
	Since the functor $X\mapsto \Mod_{(\Perf(X),\Qoppa^\c)}(\Cat^\p_\idem)$ is an étale sheaf of symmetric monoidal categories on qcqs derived algebraic spaces (Proposition~\ref{prop:descent-animated}(iv)), and since $\L\Br^\g$ is étale-locally left Kan extended from smooth $\Z$-algebras, where $\Qoppa^\g=\Qoppa^\c$, there is a unique natural transformation
	\begin{equation*}\label{eqn:dBr-to-Cat^p}
	\L\Br^\g\to \Mod_{(\Perf(\ph),\Qoppa^\c)}(\Cat^\p_\idem)\colon (\dSpc^\qcqs)^\op\to \CAlg^\lax(\Cat), \quad (\scr A,\scr A_{\geq 0},\phi)\mapsto (\scr A(\ph)^\omega,\c\Qoppa^{\geq 0}_\phi),
	\end{equation*}
	which agrees with the previous one on smooth $\Z$-algebras.
	By Remark~\ref{rmk:dTate}(i), this natural transformation factors uniquely through the (lax symmetric monoidal) right Bousfield localization $\L\Br^\g\to \Br^\c$, so that one does not lose any information by restricting it to $\Br^\c\subset\L\Br^\g$.
	Precomposing it with~\eqref{eqn:dBr-coker} gives the natural transformation
	\[
	\Z_{\pm\infty}\times\C_2\times \Pic^\dagger \to \Mod_{(\Perf(\ph),\Qoppa^\c)}(\Cat^\p_\idem),\quad (m,\epsilon,\scr L)\mapsto (\Perf(\ph),\c\Qoppa^{\geq m}_{\epsilon\scr L}).
	\] 
\end{variant}

\begin{summary}[Parameter spaces for genuine and classical Poincaré categories]
	The Poincaré–Brauer stack $\Br^\g$ and its classical variant $\Br^\c$ are natural parameter spaces for the Poincaré categories of spectral and derived algebraic spaces considered in Sections \ref{sec:genuine} and~\ref{sec:derived}:
	\[
	\begin{tikzcd}[column sep={4em}]
		\Br^\g \ar{dr}{\Qoppa^{\geq 0}} & {} \\
		\L\Br^\g \ar{r}[description]{\c\Qoppa^{\geq 0}} \ar{u} \ar[phantom]{ur}[near start,left,sloped]{\Rightarrow} & \Cat^\p\rlap. \\
		\Br^\c \ar[hook,shift left]{u} \ar[<-,shift right]{u} \ar[dashed]{ur} &
	\end{tikzcd}
	\]
	We spell out more precisely how this parametrization relates to Definitions \ref{def:genuineGW} and~\ref{def:derivedGW}.
	In the following two diagrams of symmetric monoidal categories, the third column is the étale-local cokernel of the first two. 
	We first consider the genuine case:
	\[
	\begin{tikzcd}
		\Pic^\dagger \ar{r}{(\deg,2)} & \Z_{\pm\infty}\times(\Pic^\dagger)^{\B\C_2} \ar{r} & \Br^\g \ar{r}{\Qoppa^{\geq 0}} & \Cat^\p \\
		\Pic \ar[hook]{u} \ar{r}{(0,2)} & \{0,\pm\infty\} \times (\Pic^\dagger)^{\B\C_2} \ar[hook]{u}[swap]{(\deg,\id)} \ar{r} & \mathrm{Core}(\Br^\g) \ar[hook]{u} & \\
		\Pic \ar[equal]{u} \ar{r}{2} & (\Pic^\dagger)^{\B\C_2} \ar[hook]{u}[swap]{0} \ar{r} & \Pic(\Br^\g)\rlap. \ar[hook]{u} &
	\end{tikzcd}
	\]
	Here, the composite functor $\{0,\pm\infty\} \times (\Pic^\dagger)^{\B\C_2}\to\Cat^\p$ sends $(0,\scr L)$, $(-\infty,\scr L)$, and $(\infty,\scr L)$ to the Poincaré structures $\Qoppa^\g_{\scr L}$, $\Qoppa^\s_{\scr L}$, and $\Qoppa^\q_{\scr L}$ of Definition~\ref{def:genuineGW}.
	The fact that $\{0,\pm\infty\} \times (\Pic^\dagger)^{\B\C_2}\to \Br^\g$ is surjective means that these capture all the Poincaré categories $(\Perf(X),\Qoppa^{\geq m}_{\scr L})$. The category structure on $\Br^\g$ further encodes the Poincaré functors
	\[
	(\Perf(X),\Qoppa^\q_{\scr L}) \to (\Perf(X),\Qoppa^\g_{-\scr L[2]}) \to (\Perf(X),\Qoppa^\g_{\scr L}) \to (\Perf(X),\Qoppa^\s_{\scr L}).
	\]
	The classical case is a bit more complicated:
	\[
	\begin{tikzcd}
		\Pic^\dagger \ar{r}{(\deg,\mathrm{sgn},2)} & \Z_{\pm\infty}\times\C_2\times\Pic^\dagger \ar{r} & \L\Br^\g \ar{r}{\c\Qoppa^{\geq 0}} & \Cat^\p \\
		\Pic^\dagger \ar[equal]{u} \ar{r} & m=\pm\infty\text{ or }\epsilon=(-1)^{m+\deg\scr L} \ar[hook,shift left]{u} \ar[<-,shift right]{u} \ar{r} & \Br^\c \ar[hook,shift left]{u} \ar[<-,shift right]{u} \ar[dashed]{ur} & \\
		\deg^{-1}(2\Z) \ar[hook]{u} \ar{r}{(\deg,2)} & \epsilon=1 \ar[hook]{u} \ar{r} & \Br^\c \ar[equal]{u} \\
		\Pic \ar[hook]{u} \ar{r}{(0,2)} & \{0,\pm\infty\} \times \Pic^\dagger \ar[hook]{u}[swap]{(\deg,\id)} \ar{r} & \mathrm{Core}(\Br^\c) \ar[hook]{u} & \\
		\Pic \ar[equal]{u} \ar{r}{2} & \Pic^\dagger \ar[hook]{u}[swap]{0} \ar{r} & \Pic(\Br^\c)\rlap. \ar[hook]{u} &
	\end{tikzcd}
	\]
	Here, the composite functor $\{0,\pm\infty\} \times \Pic^\dagger\to\Cat^\p$ sends $(0,\scr L)$, $(-\infty,\scr L)$, and $(\infty,\scr L)$ to the Poincaré structures $\Qoppa^\c_{\scr L}$, $\Qoppa^\varsigma_{\scr L}$, and $\Qoppa^\qoppa_{\scr L}$ of Definition~\ref{def:derivedGW}.
	The surjectivity of $\{0,\pm\infty\} \times \Pic^\dagger\to \Br^\c$ and the fact that $\L\Br^\g\to \Cat^\p$ factors through $\Br^\c$ means that these capture all the Poincaré categories $(\Perf(X),\c\Qoppa^{\geq m}_{\epsilon\scr L})$. The category structure on $\Br^\c$ further encodes the Poincaré functors
	\[
	(\Perf(X),\Qoppa^\qoppa_{\scr L}) \to (\Perf(X),\Qoppa^\c_{\scr L[4]}) \to (\Perf(X),\Qoppa^\c_{\scr L}) \to (\Perf(X),\Qoppa^\varsigma_{\scr L}).
	\]
\end{summary}

\section{Motivic spectra from Poincaré localizing invariants}
\label{sec:KO}

In this section, we construct motivic spectra representing Grothendieck–Witt theory and other localizing invariants of Poincaré categories of derived algebraic spaces.

We will see in particular that all variants of classical Grothendieck–Witt theory considered in Section~\ref{sec:derived} are encoded by a single motivic $\E_\infty$-ring spectrum $\KO$ with a Bott element $\beta^4\colon (\P^1)^{\otimes 4}\to \KO$, which is stable under arbitrary base change. Letting $\KO^\varsigma=\KO[\beta^{-1}]$ and $\KO^\qoppa=\fib(\KO\to \KO^\wedge_\beta)$ denote the left and right Bott periodizations of $\KO$, respectively, we have
\begin{align*}
	\Omega^{\infty}_{\P^1}\Sigma^\xi\KO &= \GW^\c_{\det^\dagger(\xi)}, \\
	\Omega^{\infty}_{\P^1}\Sigma^\xi\KO^\varsigma &= \GW^\varsigma_{\det^\dagger(\xi)},\\
	\Omega^{\infty}_{\P^1}\Sigma^\xi\KO^\qoppa &= \GW^\qoppa_{\det^\dagger(\xi)},
\end{align*}
for any K-theory element $\xi$ with shifted determinant $\det^\dagger(\xi)$.
In particular, on affine derived schemes, the motivic spectra $\KO$, $\Sigma_{\P^1}^2\KO$, and $\Sigma_{\P^1}^4\KO$ represent the group completions of the monoids of nondegenerate symmetric, alternating, and quadratic forms, respectively.

Furthermore, on bounded derived algebraic spaces, we will see that the genuine Grothendieck–Witt theory $\GW^\g_{\scr L}$ considered in Section~\ref{sec:genuine} is refined by a $(\Pic^\dagger)^{\B\C_2}$-graded $\KO$-$\E_\infty$-algebra $\KO^\g_?$, whose $\Pic^\dagger$-graded part agrees with $\Sigma^?\KO$ on classical algebraic spaces (but not in general).
Letting $\KO^\s_?=\KO^\g_?[\beta^{-1}]$ and $\KO^\q_?=\fib(\KO^\g_?\to (\KO^\g_?)^\wedge_\beta)$, we then have
\begin{align*}
	\Omega^{\infty-n}_{\P^1}\KO^\g_{\scr L} &= \GW^{\g}_{\scr L[n]}, \\
	\Omega^{\infty-n}_{\P^1}\KO^\s_{\scr L} &= \GW^\s_{\scr L[n]},\\
	\Omega^{\infty-n}_{\P^1}\KO^\q_{\scr L} &= \GW^\q_{\scr L[n]}.
\end{align*}
Unlike $\KO$, however, the motivic spectra $\KO^\g_?$ are not stable under arbitrary base change.

\begin{construction}[Motivic spectra representing genuine $\scr F$-theory]
\label{cst:MS}
Let $S$ be a qcqs derived algebraic space.
If $\dSpc^\g\to \dSpc$ denotes the cartesian fibration classified by $(\Br^\g)^\op$, Construction~\ref{cst:Br-to-Cat^p} provides an oplax symmetric monoidal functor
\[
\dSpc^{\g,\qcqs}_S\to \Mod_{(S,\Qoppa^\g)}(\Cat^\p)^\op,\quad (X,\scr A,\scr A_{\geq 0},\phi)\mapsto (\scr A(X)^\omega,\Qoppa^{\geq 0}_\phi),
\]
where $(S,\Qoppa^\g)$ is short for $(\Perf(S),\Qoppa^\g_{\scr O})$.
On bounded spaces, it sends Nisnevich squares to Poincaré–\allowbreak Karoubi squares by Corollary~\ref{cor:nis-descent} (note that this is an étale-local statement by Proposition~\ref{prop:mod-descent}), so that we get a colimit-preserving oplax symmetric monoidal functor
\[
\scr P_\Nis(\dSpc^{\g,\b}_S)=\scr P_\Nis(\dSpc^{\g,\b,\qcqs}_S) \to \Fun^\loc(\Mod_{(S,\Qoppa^\g)}(\Cat^\p),\An)
\]
from the category of Nisnevich sheaves (where we allow large values). By Corollary~\ref{cor:ebu}, its stabilization factors through the category $\scr P_{\Nis,\sbu}$ of Nisnevich sheaves satisfying smooth blowup excision, and we get
\begin{equation}\label{eqn:nc-realization}
\scr P_{\Nis,\sbu}(\dSpc^{\g,\b}_S,\Sp) \to \Fun^\loc(\Mod_{(S,\Qoppa^\g)}(\Cat^\p),\Sp).
\end{equation}
Let $p\colon \P^1_X\to X$ be the structure map and $\infty\colon X\to \P^1_X$ its section at infinity.
By Theorem~\ref{thm:PBF}(iv), for any $(X,\scr A,\scr A_{\geq 0},\phi)\in \dSpc^{\g,\b,\qcqs}_S$, the sequence
\[
(\scr A(X)^\omega,\Qoppa^{\geq 0}_\phi) \xrightarrow{p^*} ((p^*\scr A)(\P^1_X)^\omega,\Qoppa^{\geq 0}_\phi) \xrightarrow{p_*(\ph\otimes\scr O(-1))} (\scr A(X)^\omega,\Omega\Qoppa^{\geq 0}_\phi)
\]
is a split Poincaré–Verdier sequence, as can be checked étale-locally. Since the cofiber of $p^*$ is stably identified with the fiber of $\infty^*$, we deduce that the functor~\eqref{eqn:nc-realization} sends 
\[(\P^1,\infty)\otimes \Sigma^\infty_+ j (X,\scr A,\scr A_{\geq 0},\phi)\quad\text{to}\quad\Sigma^\infty j(\scr A(X)^\omega,\Omega\Qoppa^{\geq 0}_\phi),\]
where $j$ is the Yoneda embedding.
In particular, the oplax $\Fin^\simeq$-linear structure on the functor~\eqref{eqn:nc-realization} defined by $(\P^1,\infty)$ is strict, and its target is a $(\Fin^\simeq)^\gp$-module. By the universal property of $\P^1$-spectra, there is a unique colimit-preserving oplax symmetric monoidal functor
\[
\Sp_{\P^1}(\scr P_{\Nis,\sbu}(\dSpc^{\g,\b}_S,\Sp)) \to \Fun^\loc(\Mod_{(S,\Qoppa^\g)}(\Cat^\p),\Sp)
\]
extending~\eqref{eqn:nc-realization}.
Its right adjoint is a lax symmetric monoidal functor
\[
\Fun^\loc(\Mod_{(S,\Qoppa^\g)}(\Cat^\p),\Sp) \to \Sp_{\P^1}(\scr P_{\Nis,\sbu}(\dSpc^{\g,\b}_S,\Sp)),
\]
which preserves smallness, limits, and colimits (since these are computed objectwise on $\dSpc^{\g,\b,\qcqs}_S$).
On the other hand, restriction to the categories $\Sm_X$ gives a (fully faithful) symmetric monoidal functor
\[
\Sp_{\P^1}(\scr P_{\Nis,\sbu}(\dSpc^{\g,\b}_S,\Sp)) \to \Fun_S^\b(\Br^\g,\MS),
\]
where the right-hand side is the category of sections of the cocartesian fibration over $(\dSpc_S^\b)^\op$ classified by $X\mapsto \Fun(\Br^\g(X),\MS_X)$. 
Hence, we get a lax symmetric monoidal functor
\[
\scr R^\g\colon \Fun^\loc(\Mod_{(S,\Qoppa^\g)}(\Cat^\p),\Sp) \to\Fun_S^\b(\Br^\g,\MS),
\]
which amounts to a collection of lax symmetric monoidal functors
\[
\scr R^\g_X\colon \Fun^\loc(\Mod_{(S,\Qoppa^\g)}(\Cat^\p),\Sp) \to \Fun(\Br^\g(X),\MS_X),
\]
laxly natural in $X\in \dSpc_S^\b$. Note that $\scr R^\g$ and $\scr R_X^\g$ preserve limits and colimits.

Given $\scr L\in \Pic^\dagger(X)^{\B\C_2}$, we will write 
\[
\scr R^\g_X(\scr F)_{\scr L},\,\scr R^\s_X(\scr F)_{\scr L},\,\scr R^\q_X(\scr F)_{\scr L}\in \MS_X
\]
for the values of $\scr R^\g_X(\scr F)_?$ on the Poincaré–Azumaya algebras $(\scr Q,\scr Q_{\geq 0},\phi_{\scr L})$, $(\scr Q,\scr Q,\phi_{\scr L})$, $(\scr Q,0,\phi_{\scr L})$, and we omit the subscript $\scr L$ when $\scr L=\scr O$.
\end{construction}

\begin{remark}
	The boundedness restriction in Construction~\ref{cst:MS} is not necessary if we only consider the quadratic Poincaré structures $\Qoppa^\q$. Thus, 
	we get a nonunital lax symmetric monoidal functor
	\[
	\scr R^\q\colon \Fun^\loc(\Mod_{(S,\Qoppa^\g)}(\Cat^\p),\Sp) \to\Fun_S((\Br^\dagger)^{\h\C_2},\MS).
	\]
\end{remark}

\begin{remark}
	If $S$ is a classical qcqs scheme and $(\scr A,\phi)\in \Br^\dagger(S)^{\h\C_2}$, the motivic spectrum $\scr R_S^\g(\scr F)_{(\scr A,\scr A,\phi)}$ in $\MS_S$ coincides with the $\P^1$-spectrum $\scr R_S^\s(\scr F;(\scr A(S)^\omega,\phi))$ of \cite[Construction 7.4.11]{CHN}.
\end{remark}

\begin{variant}[Motivic spectra representing classical $\scr F$-theory]
	\label{var:MS}
	Let $S$ be a qcqs derived algebraic space.
	By Variant~\ref{var:dBr-to-Cat^p}, there is an oplax symmetric monoidal functor
	\[
	\dSpc^{\c,\qcqs}_S\to \Mod_{(S,\Qoppa^\c)}(\Cat^{\p})^\op,\quad (X,\scr A,\scr A_{\geq 0},\phi)\mapsto (\scr A(X)^\omega,\c\Qoppa^{\geq 0}_\phi),
	\]
	where $\dSpc^\c\to\dSpc$ is now the cartesian fibration classified by $(\Br^\c)^\op$.
	If we repeat Construction~\ref{cst:MS} using Proposition~\ref{prop:descent-animated}(iv), Corollary~\ref{cor:animated-nis}, Corollary~\ref{cor:animated-ebu}, and Theorem~\ref{thm:derived-PBF},
	we obtain a lax symmetric monoidal functor
	\[
	\scr R^\c\colon \Fun^\loc(\Mod_{(S,\Qoppa^\c)}(\Cat^\p),\Sp) \to\Fun_S(\Br^\c,\MS),
	\]
	where the right-hand side is the category of sections of the cocartesian fibration over $\dSpc_S^\op$ classified by $X\mapsto \Fun(\Br^\c(X),\MS_X)$.
	This amounts to a collection of lax symmetric monoidal functors
	\[
	\scr R^\c_X\colon \Fun^\loc(\Mod_{(S,\Qoppa^\c)}(\Cat^\p),\Sp) \to \Fun(\Br^\c(X),\MS_X),
	\]
	laxly natural in $X\in \dSpc_S$. Note that $\scr R^\c$ and $\scr R^\c_X$ preserve limits and colimits.
	
	Given $\scr L\in \Pic^\dagger(X)$, we will write 
	\[
	\scr R^\c(\scr F)_{\scr L},\, \scr R^\varsigma(\scr F)_{\scr L},\, \scr R^\qoppa(\scr F)_{\scr L} \in \MS_X
	\]
	for the values of $\scr R^\c(\scr F)_?$ on the classical Poincaré–Azumaya algebras $(\scr Q,\scr Q_{\geq 0},\phi_{\scr L})$, $(\scr Q,\scr Q,\phi_{\scr L})$, $(\scr Q,0,\phi_{\scr L})$, and we omit the subscript $\scr L$ when $\scr L=\scr O$.
\end{variant}

\begin{remark}\label{rmk:c-to-g}
	Since there is an $\E_\infty$-map $\Qoppa^\c\to \Qoppa^\g$, the functor $\scr R^\g$ (precomposed with $\Br^\c\to\Br^\g$) is an $\E_\infty$-algebra over $\scr R^\c$ (restricted to bounded derived algebraic spaces). On classical $2$-torsionfree algebraic spaces, there is no difference between $\scr R^\g$ and $\scr R^\c$. On classical algebraic spaces with $2$-torsion, $\scr R^\g$ and $\scr R^\c$ agree on triples $(m,\epsilon,\scr L)$ with $m=\pm\infty$ or $\epsilon=(-1)^{m+\deg\scr L}$, but $\scr R^\g$ also contains the ``skew'' variants with $\epsilon=-(-1)^{m+\deg\scr L}$ as well as $\C_2$-actions on $\scr L$ given by arbitrary square roots of $1$. Furthermore, the maps $\scr R^\qoppa\to\scr R^\q$ and $\scr R^\varsigma\to\scr R^\s$ are isomorphisms whenever both sides are defined (see Proposition~\ref{prop:genuine-vs-derived}).
\end{remark}

Let us unpack the functor $\scr R^\c$ from Variant~\ref{var:MS}.
For any Karoubi-localizing invariant 
\[\scr F\colon \Mod_{(S,\Qoppa^\c)}(\Cat^\p)\to\Sp\]
and any $X\in\dSpc_S$, we get a $\Br^\c(X)$-graded motivic spectrum $\scr R^\c_X(\scr F)_?$ in $\MS_X$.
Since this assignment is lax symmetric monoidal, it preserves $\E_\infty$-algebras, modules, etc. 
For any morphism $f\colon Y\to X$ in $\dSpc_S$, we get a morphism
\[
f^*\scr R^\c_X(\scr F)_{?} \to \scr R^\c_Y(\scr F)_{f^*(?)}
\]
in $\Fun(\Br^\c(Y),\MS_Y)$, which is an isomorphism when $f$ is smooth.

By construction, the components of the $\P^1$-spectrum $\scr R^\c_X(\scr F)_?$ are as follows: for any $n\in \Z$, we have
\[
\Omega^{\infty-n}_{\P^1}\scr R^\c_X(\scr F)_{(\scr A,\scr A_{\geq 0},\phi)}=\scr F(\scr A(\ph)^\omega,\Sigma^n\c\Qoppa_{\phi}^{\geq 0})=\scr F(\scr A(\ph)^\omega,\c\Qoppa_{\phi[n]}^{\geq n})
\]
as presheaves of spectra on qcqs smooth $X$-spaces. In particular, the weight $n$ cohomology theory represented by the motivic spectrum $\scr R^\c_X(\scr F)_{\scr L}\in \MS_X$ is classical $\scr L[n]$-valued $\scr F$-theory:
\[
\Omega^{\infty-n}_{\P^1}\scr R^\c_X(\scr F)_{\scr L}=\scr F^\c_{\scr L[n]}.
\]
Similarly, if $X$ is bounded, the weight $n$ cohomology theory represented by the motivic spectrum $\scr R^\g_X(\scr F)_{\scr L}\in \MS_X$ is genuine $\scr L[n]$-valued $\scr F$-theory:
\[
\Omega^{\infty-n}_{\P^1}\scr R^\g_X(\scr F)_{\scr L}=\scr F^\g_{\scr L[n]}.
\]

\begin{definition}[$\KO$ and $\KW$]
	\leavevmode
\begin{enumerate}
	\item We define the $\E_\infty$-algebras $\KO_?$ and $\KW_?$ in $\Fun_\Z(\Br^\c,\MS)$ by
	\[
		\KO_? = \scr R^\c(\GW)_?\quad\text{and}\quad	\KW_? = \scr R^\c(\L)_?.
	\]
	\item We define the $\E_\infty$-algebras $\KO^\g_?$ and $\KW^\g_?$ in $\Fun_\Z^\b(\Br^\g,\MS)$ by
	\[
		\KO^\g_? = \scr R^\g(\GW)_?\quad\text{and}\quad	\KW^\g_? = \scr R^\g(\L)_?.
	\]
\end{enumerate}
\end{definition}

Thus, the motivic spectra $\KO_?$ and $\KW_?$ represent $\GW^\c$ and $\L^\c$ (more precisely, their right Kan extensions from qcqs derived algebraic spaces), and their twists by classical Poincaré–Azumaya algebras. 
Since $\GW$ is the unit in Karoubi-localizing invariants, the functor $\scr R^\c$ lands in $\KO_?$-modules:
\[
\scr R^\c\colon \Fun^\loc(\Mod_{(S,\Qoppa^\c)}(\Cat^\p),\Sp) \to \Mod_{\KO_?}(\Fun_S(\Br^\c,\MS)).
\]
In particular, each motivic spectrum $\scr R^\c_X(\scr F)_{(\scr A,\scr A_{\geq 0},\phi)}$ is a module over $\KO$ in $\MS_X$.
Similarly, the motivic spectra $\KO^\g_?$ and $\KW^\g_?$ represent $\GW^\g$ and $\L^\g$, and each motivic spectrum $\scr R^\g_X(\scr F)_{(\scr A,\scr A_{\geq 0},\phi)}$ is a module over $\KO^\g$ in $\MS_X$, which is in turn an $\E_\infty$-$\KO$-algebra (by Remark~\ref{rmk:c-to-g}).

\begin{remark}
	Over classical schemes, the $\A^1$-localizations of $\KO^\s_{\scr L}$ and $\KW^\s_{\scr L}$ (which coincide with $\KO^\varsigma_{\scr L}$ and $\KW^\varsigma_{\scr L}$ if $\scr L$ has trivial involution) recover the $\A^1$-invariant motivic $\E_\infty$-ring spectra defined by Calmès–Harpaz–Nardin in \cite[Definition 8.1.1]{CHN}.
	We will see in Corollary~\ref{cor:A1-localization} that this is also true of $\KO^{*}_{\scr L}$ and $\KW^{*}_{\scr L}$ for any $*\in\{\g,\s,\q\}$, as $\L_{\A^1}$ identifies all these variants.
\end{remark}

\begin{remark}
	Evaluating $\KO_?$ on $(\scr Q, 0,\phi_{\scr O})\to (\scr Q, \scr Q_{\geq 0},\phi_{\scr O})\to (\scr Q, \scr Q,\phi_{\scr O})$ yields the sequence
	\[
	\KO^\qoppa\to\KO\to\KO^\varsigma
	\]
	in $\MS_X$. Since $\KO_?\colon \Br^\c(X)\to \MS_X$ is a lax symmetric monoidal functor, $\KO\to \KO^\varsigma$ is an $\E_\infty$-map and both $\KO^\qoppa\to \KO$ and $\KO^\qoppa\to\KO^\varsigma$ are $\E_\infty$-ideals. Similar remarks apply to $\KW_?$, $\KO^\g_?$, and $\KW^\g_?$.
\end{remark}

\begin{remark}[Graded enhancements]
	\label{rmk:KSp}
	Let $\Gamma$ be a symmetric monoidal category with a lax symmetric monoidal functor $\gamma\colon\Gamma\to \Z_{\pm\infty}\times\C_2$.
	Restricting along $\gamma$ defines a lax symmetric monoidal lift of $\scr R^\c$ to $\Fun_S(\Br^\c,\MS)^{\Gamma}$ and of $\scr R^\g$ to $\Fun_S^\b(\Br^\g,\MS)^{\Gamma}$.
	For example, using the functor $\Z\to \Z_{\pm\infty}\times\C_2$ sending $n$ to $(n,(-1)^n)$, we obtain a $\Z$-graded $\E_\infty$-$\KO_?$-algebra structure on $(\KO_{?[2n]})_{n\in \Z}$.
	
	Consider the lax symmetric monoidal functor $\C_2\to \Z_{\pm\infty}\times\C_2$ sending $1$ to $(0,1)$ and $-1$ to $(1,-1)$.
	Writing $\KSp_?=\Sigma^2_{\P^1}\KO_?$, we obtain a $\C_2$-graded $\E_\infty$-$\KO_?$-algebra structure on $(\KO_?,\KSp_?)$, hence an $\E_\infty$-$\KO_?$-algebra structure on $\KO_?\oplus\KSp_?$. Note that $\KSp$ represents the Grothendieck–Witt theory of alternating forms, also known as symplectic K-theory. The $\C_2$-graded algebra structure on $(\KO,\KSp)$ reflects the fact that the tensor product of two alternating forms is a quadratic form, which induces a symmetric form.
\end{remark}

\begin{remark}[Comparison with localizing invariants of stable categories]
	A similar construction to Construction~\ref{cst:MS} yields a lax symmetric monoidal functor
	\[
	\scr R\colon \Fun^\loc(\Mod_{S}(\Cat^\st),\Sp) \to\Fun_S(\Br^\dagger,\MS),
	\]
	which is equivariant for the actions of $\C_2$ on $\Cat^\st$ and $\Br^\dagger$ and sends the unit $\K$ to $\KGL$ (and its twists by Azumaya algebras).
	For a classical Azumaya algebra $\scr A$, the motivic spectrum $\KGL_{\scr A}$ was also defined in \cite[\sect 4]{ElNaYa}.
	There is then a diagram of lax symmetric monoidal functors
	\[
	\begin{tikzcd}
		\Fun^\mathrm{bord}(\Mod_{(S,\Qoppa^\c)}(\Cat^\p),\Sp) \ar{r}{\scr R^\c} \ar[<-]{d}[swap]{(\ph)^\mathrm{bord}} & \Mod_{\KW}(\Fun_S(\Br^\c,\MS)) \ar{d}{\mathrm{forget}} \\
		\Fun^\loc(\Mod_{(S,\Qoppa^\c)}(\Cat^\p),\Sp) \ar{r}{\scr R^\c} \ar{d}[swap]{(\ph)^\mathrm{hyp}} \ar[phantom,start anchor=center,end anchor=center]{ur}[description]{\textstyle\Uparrow} \ar[phantom,start anchor=center,end anchor=center]{dr}[description]{\textstyle\Downarrow} & \Mod_{\KO}(\Fun_S(\Br^\c,\MS)) \ar[<-]{d}{\mathrm{forget}} \\
		\Fun^\loc(\Mod_{S}(\Cat^\st),\Sp)^{\h\C_2} \ar{r}{\scr R} & \Mod_\KGL(\Fun_S(\Br^\dagger,\MS)^{\h\C_2})\rlap,
	\end{tikzcd}
	\]
	as well as a related diagram with $\scr R^\g$ instead of $\scr R^\c$.
\end{remark}

\section{The Hopf and Bott fracture squares}
\label{sec:fracture}

In this section, we recast the following two cartesian squares of $\E_\infty$-algebras in $\Fun^\loc(\Cat^\p,\Sp)$ as fracture squares of the motivic spectrum $\KO$:
\[
\begin{tikzcd}
	\GW \ar{r} \ar{d} & \K^{\h\C_2} \ar{d} &[4em] \GW \ar{r} \ar{d} & \GW/\GW^\q \ar{d} \\
	\L \ar{r} & \K^{\t\C_2}\rlap, & \GW^\s \ar{r} & \L^\mathrm{n}\rlap.
\end{tikzcd}
\]
The first square is the so-called \emph{fundamental cartesian square}, which expresses GW-theory as a gluing of K-theory with L-theory, and we will see that it is the fracture square of $\KO$ with respect to the Hopf element $\eta$ of the motivic sphere spectrum \cite[\sect 3]{motsphere}. The second square is the definition of \emph{normal L-theory} $\L^\mathrm{n}=\L^\s/\L^\q$, and we will see that it is the fracture square of $\KO$ with respect to the Bott element $\beta^4$ of $\KO$, which is a lift to $\KO$ of the fourth power of the usual Bott element of $\KGL$.

Let $R$ be an $\E_\infty$-algebra in $\MS_S$, $L$ an invertible $R$-module, and $\alpha\colon L\to R$ an $R$-linear map. An $R$-module $M$ is called \emph{$\alpha$-periodic} if $\alpha$ acts invertibly on $M$. The inclusion of $\alpha$-periodic $R$-modules into $\Mod_R(\MS_S)$ preserves limits and colimits and hence admits adjoints on both sides. The left adjoint $M\mapsto M[\alpha^{-1}]$ is called \emph{left $\alpha$-periodization} and is computed as the colimit of the sequence
\[
\dotsb \to L\otimes_R M \xrightarrow{\alpha} M \xrightarrow{\alpha} L^{-1}\otimes_R M \to \dotsb
\]
(the claim is local on $S$, so we can assume $\MS_S$ compactly generated\footnote{$\MS_S$ is in fact compactly generated even if $S$ is not qcqs, but this fails in similar cases, e.g., for $\QCoh(S)$.} and apply \cite[Lemma 12.1]{norms}). The right adjoint $M\mapsto \fib(M\to M^\wedge_\alpha)$ is called \emph{right $\alpha$-periodization} and is computed as the the limit of the above sequence (by \cite[Lemma 12.13]{norms}). 
The fracture square of $M$ with respect to $\alpha$ is then the following cartesian square in $\Mod_R(\MS_S)$, which is lax symmetric monoidal in $M$:
\[
\begin{tikzcd}
	M \ar{r} \ar{d} & M^\wedge_\alpha \ar{d} \\
	M[\alpha^{-1}] \ar{r} & M^\wedge_\alpha[\alpha^{-1}]\rlap.
\end{tikzcd}
\]

\begin{remark}[The fundamental cartesian square]
	\label{rmk:fund-square}
	Since $\scr R^\c$ preserves limits and colimits, we get from \cite[Corollary 3.6.7]{hermitianII} a cartesian square
	\[
	\begin{tikzcd}
		\scr R^\c(\scr F)_? \ar{r} \ar{d} & \scr R(\scr F^\mathrm{hyp})^{\h\C_2}_? \ar{d} \\
		\scr R^\c(\scr F^\mathrm{bord})_? \ar{r} & \scr R(\scr F^\mathrm{hyp})^{\t\C_2}_?
	\end{tikzcd}
	\]
	in $\Fun_S(\Br^\c,\MS)$, which is natural and symmetric monoidal in $\scr F$.
	For $\scr F=\GW$, this is a cartesian square of $\E_\infty$-algebras
	\[
	\begin{tikzcd}
		\KO_{?} \ar{r} \ar{d} & \KGL_{?}^{\h\C_2} \ar{d} \\
		\KW_{?} \ar{r} & \KGL_{?}^{\t\C_2},
	\end{tikzcd}
	\]
	where $\KGL_{(\scr A,\phi)}$ is the motivic spectrum representing $\scr A$-twisted K-theory, with $\C_2$-action induced by $\phi$.
\end{remark}

\begin{remark}[The Bott--Genauer fiber sequence]
	\label{rmk:bott-genauer}
	The Bott--Genauer sequence \cite[\sect 4.3]{hermitianII} induces a fiber sequence
	\[
	\Omega_{\P^1}\scr R^\c(\scr F)_? \to \scr R(\scr F^\mathrm{hyp})_? \xrightarrow{\mathrm{hyp}} \scr R^\c(\scr F)_?
	\]
	in $\Fun_S(\Br^\c,\MS)$, which is natural and symmetric monoidal in $\scr F$.
	For $\scr F=\GW$, this is the \emph{Wood fiber sequence}
	\[
	\Omega_{\P^1}\KO_? \to \KGL_? \xrightarrow{\mathrm{hyp}} \KO_?.
	\]
\end{remark}

\begin{definition}\label{def:hopf}
	The \emph{Hopf element} $\eta\colon \Sigma^{-1}\P^1\to\1$ in $\MS_\Z$ is the connecting map of the fiber sequence
	\[
	\P^1\to \P^2\to \P^2/\P^1=(\P^1)^{\otimes 2}
	\]
	(see \cite[Proposition 3.6]{motsphere}).
\end{definition}

\begin{lemma}\label{lem:wood}
	The connecting map of the Wood fiber sequence
	\[
	\Omega \KO\to \Omega_{\P^1}\KO
	\]
	is multiplication by the Hopf element $\eta\colon \Sigma_{\P^1}\1\to \Sigma\1$.
\end{lemma}

\begin{proof}
	Let $\infty\colon \P^1\to \P^2$ be the pointed map $[x:y]\mapsto [x:y:0]$, where $\P^1$ and $\P^2$ have base points $[1:0]$ and $[1:0:0]$, let $0\colon *\to \P^2$ be the point $[0:0:1]$, and let $p\colon \P^2\to *$ be the structure map.
	There is a commutative diagram of Poincaré categories
	\[
	\begin{tikzcd}
		(\Perf(X),\Qoppa^\c_{\scr O}) \ar{r}{0_*} \ar[equal]{d} & (\Perf(\P^2_X),\Qoppa^\c_{\scr O[2]}) \ar{d}{q\to q\infty_*\infty^*} \ar{r}{\infty^*} & (\Perf(\P^1_X),\Qoppa^\c_{\scr O[2]}) \ar{d}{q\infty_*} \\
		(\Perf(X),\Qoppa^\c_{\scr O}) \ar{r}{\id\to 0} & (\mathrm{Ar}(\Perf(X)),{\cofib}\circ {\Qoppa^\c_{\scr O}}) \ar{r}{\mathrm{target}} & (\Perf(X),\Qoppa^\c_{\scr O[1]})\rlap,
	\end{tikzcd}
	\]
	where $0_*$ is a Poincaré functor via Grothendieck duality (see Definition~\ref{def:Thom}) and $q=p_*(\ph\otimes\scr O(-1))$.
	The lower row is the split Poincaré–Verdier sequence inducing the Bott–Genauer fiber sequence, and the vertical functors induce the isomorphisms
	\[
	\Omega_{\P^2}\GW^\c_{\scr O[2]} \simeq \K\quad \text{and}\quad \Omega_{\P^1}\GW^\c_{\scr O[2]} \simeq \GW^\c_{\scr O[1]}
	\]
	from the projective bundle formula (Theorem~\ref{thm:derived-PBF}).\footnote{This already implies the statement of the lemma up to a unit in $\GW^\c(\Z)$, which is all that is used in the sequel.}
	In particular, the composite isomorphism
	\begin{equation}\label{eqn:P^1-Thom-0}
	\GW^\c\simeq \fib\bigl(\K\xrightarrow{\hyp}\GW^\c_{\scr O[1]}\bigr) \overset{\mathrm{pbf}}{\simeq} \Omega_{\P^2/\P^1}\GW^\c_{\scr O[2]}
	\end{equation}
	is induced by the Poincaré functor
	\[
	0_*\colon (\Perf(X),\Qoppa^\c_{\scr O}) \to (\Perf(\A^2_X\;\mathrm{on}\; 0_X), \Qoppa^\c_{\scr O[2]})
	\]
	(cf.\ the even rank case of Theorem~\ref{thm:Thom}).
	It remains to show that the isomorphism~\eqref{eqn:P^1-Thom-0} coincides with the isomorphism induced by the $\P^1$-spectrum structure of $\KO$, which is the following composition:
	\begin{equation}\label{eqn:P^1-Thom}
	\GW^\c \overset{\mathrm{pbf}}{\simeq} \Omega_{\P^1}\GW^\c_{\scr O[1]} \overset{\mathrm{pbf}}{\simeq} \Omega_{\P^1}^2\GW^\c_{\scr O[2]} \simto \Omega_{B/\partial B} \GW^\c_{\scr O[2]}\simfrom \Omega_{\P^2/\P^1}\GW^\c_{\scr O[2]}.
	\end{equation}
	Here, $B\to \P^1\times\P^1$ is the blowup at $([1:0],[1:0])$ and $B\to \P^2$ is the blowup at $[1:0:0]\sqcup [0:1:0]$.
	The retract diagram of Poincaré categories
	\[
	(\Perf(X), \Qoppa^\c_{\scr L}) \xrightarrow{0_*} (\Perf(\P^1_X), \Qoppa^\c_{\scr L[1]}) \xrightarrow{q\infty_*} (\Perf(X),\Qoppa^\c_{\scr L})
	\]
	shows that the isomorphism $\GW^\c_{\scr L}\simeq \Omega_{\P^1}\GW^\c_{\scr L[1]}$ from the projective bundle formula is similarly induced by the Poincaré functor
	\[
	0_*\colon (\Perf(X),\Qoppa^\c_{\scr L}) \to (\Perf(\A^1_X\;\mathrm{on}\; 0_X), \Qoppa^\c_{\scr L[1]})
	\]
	(cf.\ the odd rank case of Theorem~\ref{thm:Thom}). Applying this twice shows that~\eqref{eqn:P^1-Thom} is induced by the same Poincaré functor as~\eqref{eqn:P^1-Thom-0}.
\end{proof}

\begin{proposition}[The Hopf fracture square]
	\label{prop:eta-localization}
	There are canonical isomorphisms of lax symmetric monoidal functors 
\[
	\scr R^\c(\ph)[\eta^{-1}]=\scr R^\c((\ph)^\mathrm{bord}),\quad \scr R^\c(\ph)^\wedge_\eta=\scr R((\ph)^\mathrm{hyp})^{\h\C_2},\quad \scr R^\c(\ph)^\wedge_\eta[\eta^{-1}]=\scr R((\ph)^\mathrm{hyp})^{\t\C_2}.
\]
In particular, we have the following isomorphisms in $\CAlg(\Fun_\Z(\Br^\c,\MS))$:
\[
\KO_?[\eta^{-1}]=\KW_?,\quad (\KO_?)^\wedge_\eta=\KGL_?^{\h\C_2},\quad (\KO_?)^\wedge_\eta[\eta^{-1}]=\KGL_?^{\t\C_2}.
\]
\end{proposition}

\begin{proof}
	Given Remark~\ref{rmk:fund-square}, it suffices to show the following:
	\begin{enumerate}
		\item $\KW$ is $\eta$-periodic.
		\item $\scr R((\ph)^\mathrm{hyp})_{\h\C_2}$ is $\eta$-nilpotent.
		\item $\scr R((\ph)^\mathrm{hyp})^{\h\C_2}$ is $\eta$-complete.
	\end{enumerate}
	By Lemma~\ref{lem:wood}, the connecting map of the Bott–Genauer fiber sequence of Remark~\ref{rmk:bott-genauer} is multiplication by $\eta$.
	The first assertion follows from the Bott–Genauer sequence for $\L$, as $\L^\hyp=0$. Since $\eta$-nilpotent spectra are closed under colimits and $\eta$-complete spectra under limits, the last two assertions follow from the fact that $\eta$ acts trivially on $\KGL$-modules, since they are orientable.
\end{proof}

\begin{remark}
	Proposition~\ref{prop:eta-localization} recovers known results after $\A^1$-localization over classical $\Z[\tfrac 12]$-schemes. The identification of the $\eta$-completion of $\L_{\A^1}\KO$ in this generality is due to Heard \cite[Corollary 3.11]{Heard} (and was previously obtained by Röndigs–Spitzweck–Østvær over fields of finite virtual $2$-cohomological dimension \cite[Theorem 1.2]{RSO-Thomason}). The identification of the left $\eta$-periodization of $\L_{\A^1}\KO$ is formulated as a conjecture of Morel in \cite[Section 6]{Hornbostel}, which to our knowledge was first proved by Ananyevskiy in \cite[Theorem 6.5]{AnanyevskiyMSL}.
\end{remark}

\begin{definition}
	The \emph{Bott element} $\beta^4\colon (\P^1)^{\otimes 4}\to \KO$ in $\MS_\Z$ is the image of $1$ by the canonical map
	 \[\GW^\c(\Z)=\GW^{\geq 0}(\Z)\to \GW^{\geq -2}(\Z)=\GW^\c_{\scr O[-4]}(\Z)=(\Omega_{\P^1}^4\GW^\c)(\Z). \]
	 We call this element $\beta^4$ because it is a lift to $\KO$ of the fourth power of the usual Bott element $\beta\colon \P^1\to \KGL$. However, $\beta^4$ is not a square in $\KO$.
\end{definition}

\begin{proposition}[The Bott fracture square]
	\label{prop:bott}
 Let $S$ be a derived algebraic space and let $\scr L\in \Pic^\dagger(S)$. Then:
 \begin{enumerate}
 	\item $\KO^\varsigma_{\scr L}$ and $\KW^\varsigma_{\scr L}$ are the left $\beta$-periodizations of $\KO_{\scr L}$ and $\KW_{\scr L}$ in $\MS_S$.
 	\item $\KO^\qoppa_{\scr L}$ and $\KW^\qoppa_{\scr L}$ are the right $\beta$-periodizations of $\KO_{\scr L}$ and $\KW_{\scr L}$ in $\MS_S$.
 \end{enumerate}
\end{proposition}

\begin{proof}
	By the fundamental cartesian square and the Bott periodicity of $\KGL$, it suffices to consider $\KW$.
	Let $X\in \Sm_S$ be qcqs. Applying $\Omega^{\infty-n}_{\P^1}(\ph)(X)$ to the sequence of multiplication by $\beta^4$ on $\KW_{\scr L}$ gives the sequence 
	\[
	\dotsb\to \L(\Perf(X),\c\Qoppa^{\geq n+2}_{\scr L[n]})\to \L(\Perf(X),\c\Qoppa^{\geq n}_{\scr L[n]})\to \L(\Perf(X),\c\Qoppa^{\geq n-2}_{\scr L[n]})\to\dotsb.
	\]
	
	(i) Since $\L\colon \Cat^\p\to \Sp$ preserves filtered colimits and $\colim_{m\to-\infty} \c\Qoppa^{\geq m}=\Qoppa^\varsigma$, the colimit of the above sequence is 
	\[
	\L(\Perf(X),\Qoppa^\varsigma_{\scr L[n]})=\Omega^{\infty-n}_{\P^1}(\KW^\varsigma_{\scr L})(X).
	\]
	
	(ii) We need to show that the limit of the above sequence is $\L(\Perf(X),\Qoppa^\qoppa_{\scr L[n]})$.
	We may assume $X$ affine and $\deg(\scr L)=d$ constant, and it suffices to prove the claim for a cofinal set of integers $n$. If $d$ and $n$ have the same parity, then
	\[
	\L(\Perf(X),\c\Qoppa^{\geq m+n}_{\scr L[n]}) = \L(\Perf(X),\c\Qoppa^{\geq m+ (n-d)/2}_{(-1)^{(n+d)/2}\scr L[-d]}).
	\]
	Since $\c\Tate^{\geq m}$ is $m$-connective, it follows from \cite[Corollary 1.2.30]{hermitianIII} that
	\[
	\L(\Perf(X),\Qoppa^\q_{\pm\scr L[-d]}) \to \L(\Perf(X),\c\Qoppa^{\geq m}_{\pm\scr L[-d]})
	\]
	is $(2m-2)$-connective, which implies the desired limit statement. 
\end{proof}

\begin{remark}
	Proposition~\ref{prop:bott} admits the following generalization for a Karoubi-localizing invariant $\scr F$ and a classical Poincaré–Azumaya algebra $(\scr A,\scr A_{\geq 0},\phi)\in\Br^\c(S)$ with $(\scr A,\scr A_{\geq 0})\in\Br(S)$:
	\begin{enumerate}
		\item If $\scr F$ is finitary, $\scr R^\c_S(\scr F)_{(\scr A,\scr A,\phi)}$ is the left $\beta$-periodization of $\scr R^\c_S(\scr F)_{(\scr A,\scr A_{\geq 0},\phi)}$.
		\item If $\scr F=\GW$ or $\scr F=\L$, $\scr R^\c_S(\scr F)_{(\scr A,0,\phi)}$ is the right $\beta$-periodization of $\scr R^\c_S(\scr F)_{(\scr A,\scr A_{\geq 0},\phi)}$.
	\end{enumerate}
	To justify (ii), we use the relative L-theory formula from \cite{HNS}: the cofiber of
	\begin{equation}\label{eqn:relative-L}
	\L(\scr A(X)^\omega,\Qoppa^\q_\phi)\to\L(\scr A(X)^\omega,\c\Qoppa^{\geq m}_\phi)
	\end{equation}
	is the equalizer of two maps
	\[
	\cofib\bigl(\Sigma^{-1}\B_\phi(\D_\phi\lambda,\D_\phi\lambda)\to \B_\phi(\D_\phi\lambda,\D_\phi\lambda)_{\h\C_2}\bigr) \rightrightarrows \B_\phi(\D_\phi\lambda,\D_\phi\lambda),
	\]
	where $\lambda$ is the linear part of $\c\Qoppa^{\geq m}_\phi$ viewed as an ind-object of $\scr A(X)^\omega$. If $X$ is affine, then $\B_\phi(\D_\phi\lambda,\D_\phi\lambda)$ is $(2m-\deg\phi)$-connective, so that~\eqref{eqn:relative-L} is $(2m-\deg\phi-2)$-connective.
\end{remark}

\begin{remark}[Normal L-theory]
	\label{rmk:normal}
	Let $S$ be a derived algebraic space and let $\scr L\in \Pic^\dagger(S)$.
	By Proposition \ref{prop:bott}(ii), we have $\KO^\qoppa_{\scr L}=\fib(\KO_{\scr L}\to (\KO_{\scr L})^\wedge_\beta)$ in $\MS_S$.
	Hence, the $\beta$-completion $(\KO_{\scr L})^\wedge_\beta$ represents classical normal L-theory:
	\[
	\Omega^{\infty-n}_{\P^1} (\KO_{\scr L})^\wedge_\beta = \cofib(\L^\qoppa_{\scr L[n]}\to \L^\c_{\scr L[n]}).
	\]
	It follows that $(\KO_{\scr L})^\wedge_\beta[\beta^{-1}]$ represents $\L^\nu_{\scr L}=\cofib(\L^\qoppa_{\scr L}\to\L^\varsigma_{\scr L})$.
\end{remark}

\begin{variant}[Genuine fracture squares]
	The results of this section have obvious analogues with $\scr R^\g$ instead of $\scr R^\c$. In particular, the following results hold for any bounded derived algebraic space $S$ and any $\scr L\in \Pic^\dagger(S)^{\B\C_2}$:
	\begin{enumerate}
		\item for each $*\in\{\g,\s,\q\}$, $\KO^*_{\scr L}[\eta^{-1}]=\KW^*_{\scr L}$, $(\KO^*_{\scr L})^\wedge_\eta=\KGL^{\h\C_2}_{\scr L}$, and $(\KO^*_{\scr L})^\wedge_\eta[\eta^{-1}]=\KGL^{\t\C_2}_{\scr L}$. 
		\item $\KO^\g_{\scr L}[\beta^{-1}]=\KO^\s_{\scr L}$, $\fib(\KO^\g_{\scr L}\to(\KO^\g_{\scr L})^\wedge_\beta)=\KO^\q_{\scr L}$, and similarly for $\KW$. The motivic spectrum $(\KO^\g_{\scr L})^\wedge_\beta[\beta^{-1}]$ represents normal L-theory $\L^\mathrm{n}_{\scr L}=\cofib(\L^\q_{\scr L}\to\L^\s_{\scr L})$.\footnote{Recall that, according to our conventions, L-theory always means Karoubi-localizing L-theory rather than additive L-theory. However, normal L-theory is the same in both settings.}
	\end{enumerate}
\end{variant}

\section{Stability under base change}
\label{sec:BC}

If $\Qoppa$ is a Poincaré structure on $\Perf(X)$, we denote by $\Vect(X,\Qoppa)$ the anima of finite locally free Poincaré objects in $(\Perf(X),\Qoppa)$.
By the group completion theorem of Hebestreit and Steimle \cite[Corollary 8.1.2]{HebestreitSteimle}, if $X$ is an affine spectral scheme, $\Qoppa$ is connective on $\Vect(X)$, and $\D_\Qoppa$ preserves $\Vect(X)$, then the $\E_\infty$-group $\Omega^\infty\GW(\Perf(X),\Qoppa)$ is the group completion of the $\E_\infty$-monoid $\Vect(X,\Qoppa)$:
\[
\Omega^\infty\GW(\Perf(X),\Qoppa)=\Vect(X,\Qoppa)^\gp.
\]
This applies to $\Qoppa^{\geq m}_{L}$ for any $L\in \Pic(X)^{\B\C_2}$ and $m\in\N\cup\{\infty\}$. If $X$ is an affine derived scheme, it also applies to $\c\Qoppa^{\geq m}_{\epsilon L}$ for any $L\in\Pic(X)$, $\epsilon\in\{\pm 1\}$, and $m\in\N\cup\{\infty\}$.

\begin{example}
	Let $R\to S$ be a $1$-connective map (resp.\ a henselian surjection) of animated rings and let $L\in\Pic(R)$. Then the induced maps $\Vect(R)\to \Vect(S)$ and $\Map_R(M,N)\to \Map_{S}(M\otimes_RS,N\otimes_RS)$ for any $M,N\in \Vect(R)$ are $1$-connective (resp.\ $0$-connective). 
	Applying the group completion theorem to the Poincaré structures $\Qoppa^\sym_L$, $\Qoppa^\alt_L$, and $\Qoppa^\qu_L$ of Construction~\ref{cst:Sym^n}, we deduce that the map
	\[
	\Omega^\infty\GW^\c_{L[n]}(R)\to \Omega^\infty\GW^\c_{L[n]}(S)
	\]
	is $1$-connective (resp.\ $0$-connective) for $n=0,2,4$. This applies in particular to the map $R\to\pi_0R$.
\end{example}

\begin{lemma}\label{lem:LKE}
	Let $\scr C\subset\scr D$ be a full subcategory and let $F\to G\leftarrow H$ be a diagram of presheaves on $\scr D$.
	If $H$ is left Kan extended from $\scr C$ and if, for every $U\in\scr C_{/G}$, the pullback $F\times_GU$ is left Kan extended from $\scr C$, then $F\times_GH$ is left Kan extended from $\scr C$.
\end{lemma}

\begin{proof}
	Write $H=\colim_i U_i$ with $U_i\in\scr C$. Then $F\times_GH= \colim_i (F\times_G U_i)$, and each $F\times_G U_i$ is a colimit of objects of $\scr C$.
\end{proof}

\begin{proposition}\label{prop:Vect-LKE}
	Let $R$ be an animated ring, $L\in\Pic(R)$, $\epsilon\in\{\pm 1\}$, and $m\in \N\cup\{\infty\}$. Then the functor $\Vect(\ph,\c\Qoppa^{\geq m}_{\epsilon L})\colon \CAlg_R^\an\to \An$ is left Kan extended from smooth $R$-algebras.
\end{proposition}

\begin{proof}
	Given $A\in \CAlg_R^\an$ and $V\in \Vect(A)$, let $F^{m}_{\epsilon L,V}\colon \CAlg_A^\an\to \An$ be the functor defined by
	\[
	F^{m}_{\epsilon L,V}(B) = \fib_{V\otimes_AB}(\Vect(B,\c\Qoppa^{\geq m}_{\epsilon L}) \to \Vect(B)).
	\]
	The left Kan extension of $F^{m}_{\epsilon L,V}$ along the forgetful functor $\CAlg_A^\an\to\CAlg_R^\an$ is the fiber product
	\[
	\Spec(A)\times_{\Vect}\Vect(\ph,\c\Qoppa^{\geq m}_{\epsilon L})\colon \CAlg_R^\an\to \An.
	\]
	Thus, if $A$ is a smooth $R$-algebra and if $F^{m}_{\epsilon L,V}$ is left Kan extended from smooth $A$-algebras, then this fiber product is left Kan extended from smooth $R$-algebras.
	By Lemma~\ref{lem:LKE} and the fact that $\Vect$ is left Kan extended from smooth $\Z$-algebra, it will suffice to show that $F^{m}_{\epsilon L,V}$ is left Kan extended from smooth $A$-algebras.
	
	There is an isomorphism
	\[
	F^{m}_{\epsilon L,V} = \Omega^\infty\c\Qoppa^{\geq m}_{\epsilon L}(V\otimes_A(\ph)) \times_{X} U,
	\]
	where $X$ is the affine space over $A$ of bilinear maps $V^{\otimes 2}\to L_A$ and $U\subset X$ is the nondegeneracy locus.
	By Lemma~\ref{lem:LKE} again, it suffices to show that $\Omega^\infty\c\Qoppa^{\geq m}_{\epsilon L}(V\otimes_A(\ph))$ is left Kan extended from smooth $A$-algebras.
	By definition,
	\[
	\Omega^\infty\c\Qoppa^{\geq m}_{\epsilon L}(V\otimes_A B) = \Omega^\infty\cofib\bigl(\map_A(V,\Sigma^{-1}\c\Tate^{\geq m}_{\epsilon L\otimes_R B})\to \map_A(V^{\otimes 2}, \epsilon L\otimes_R B)_{\h\C_2}\bigr).
	\]
	The right-hand side preserves sifted colimits in $B$ by Remark~\ref{rmk:dTate}(iii) and the fact that $\Omega^\infty$ preserves sifted colimits of connective spectra, so that it is even left Kan extended from polynomial $A$-algebras.
\end{proof}

\begin{corollary}\label{cor:GW-LKE}
	Let $R$ be an animated ring, $L\in\Pic(R)$, $\epsilon\in\{\pm 1\}$, and $m\in \N\cup\{\infty\}$. Then the functor $\Omega^\infty\c\GW^{\geq m}_{\epsilon L}\colon \CAlg_R^\an\to \An$ is left Kan extended from smooth $R$-algebras.
\end{corollary}

\begin{proof}
	As recalled above, $\Omega^\infty\c\GW^{\geq m}_{\epsilon L}$ is the group completion of $\Vect(\ph,\c\Qoppa^{\geq m}_{\epsilon L})$, so the result follows from Proposition~\ref{prop:Vect-LKE}.
\end{proof}

\begin{corollary}\label{cor:GW-LKE2}
	Let $R$ be a static ring, $L\in\Pic(R)$, $\epsilon\in\{\pm 1\}$, and $m\in \N\cup\{\infty\}$. If $m\neq\infty$, assume $\epsilon=(-1)^m$. Then the functor $\Omega^\infty\GW^{\geq m}_{\epsilon L}\colon \CAlg_R^\heart\to \An$ is left Kan extended from smooth $R$-algebras.
\end{corollary}

\begin{proof}
	The assumption and Proposition~\ref{prop:genuine-vs-derived}(ii) imply that $\GW^{\geq m}_{\epsilon L}=\c\GW^{\geq m}_{\epsilon L}$ on static rings, so the result follows from Corollary~\ref{cor:GW-LKE}.
\end{proof}

\begin{remark}\label{rmk:failure-GW}
	If $R$ is $2$-torsionfree but $2$ is not invertible in $R$, Corollary~\ref{cor:GW-LKE2} fails for $\epsilon=-(-1)^m$, since then $\Omega^\infty\GW^{\geq m+1}_{\epsilon L}\to \Omega^\infty\GW^{\geq m}_{\epsilon L}$ is an isomorphism on smooth $R$-algebras but not on all $R$-algebras.
\end{remark}

\begin{remark}[Moduli stacks of finite locally free Poincaré objects]
	Let $S$ be a derived algebraic space and let $\scr L\in \Pic(S)$. Define the algebraic stacks $\Vect^\sym_\scr L$, $\Vect^\alt_\scr L$, and $\Vect^\qu_\scr L$ over $S$ as the base change along $\scr L\colon S\to\Pic$ of the classical algebraic stacks classifying triples $(\scr E,\scr L,\phi)$, where $\scr E$ is a finite locally free sheaf, $\scr L$ is a locally free sheaf of rank $1$, and $\phi$ is a nondegenerate $\scr L$-valued symmetric, alternating, or quadratic form on $\scr E$, respectively. We then have
	\[
		\Vect(\ph,\Qoppa^{\sym}_{\scr L})=\Vect^\sym_{\scr L},\quad \Vect(\ph,\Qoppa^{\alt}_{\scr L})=\Vect^\alt_{\scr L},\quad \Vect(\ph,\Qoppa^{\qu}_{\scr L})=\Vect^\qu_{\scr L},
	\]
	as presheaves on $\dSpc_S$. 
	Indeed, the left-hand sides are clearly étale sheaves, and on $\CAlg_R^\an$ they are left Kan extended from $\CAlg_R^\sm$ by Proposition~\ref{prop:Vect-LKE}.
	The same holds for the right-hand sides by Lemma~\ref{lem:stacks} below and \cite[Proposition A.0.4]{EHKSY3}.
	Since $\Pic$ is also left Kan extended from $\CAlg_\Z^\sm$, we are reduced to the case of classical schemes, where these isomorphisms hold by definition.
\end{remark}

\begin{lemma}\label{lem:stacks}
	Let $S$ be a derived algebraic space and let $\scr L\in \Pic(S)$.
	\begin{enumerate}
		\item $\Vect^\sym_\scr L$ is a smooth algebraic stack over $S$ with affine diagonal.
		\item $\Vect^\alt_\scr L$ is a smooth algebraic stack over $S$ with smooth and affine diagonal.
		\item $\Vect^\qu_\scr L$ is a smooth algebraic stack over $S$ with smooth and affine diagonal.
	\end{enumerate}
\end{lemma}

\begin{proof}
	We can assume $\scr L=\scr O$ and hence $S=\Spec(\Z)$.
	
	(i) We have
	\[
	\Vect^\sym_n\simeq \GL_n^\sym/\GL_n,
	\]
	where $\GL_n^\sym\subset \GL_n$ is the scheme of symmetric matrices and $\GL_n$ acts on $\GL_n^\sym$ by $(S,A)\mapsto S^t\! AS$. The diagonal is affine as both $\GL_n^\sym$ and $\GL_n$ are affine.
	
	(ii) A nondegenerate alternating form over a local ring is isomorphic to a hyperbolic form. Hence,
	\[
	\Vect^\alt_{2n}\simeq \B\Sp_{2n}\quad\text{and}\quad \Vect^\alt_{2n+1}=\emptyset,
	\]
	which are indeed smooth algebraic stacks with smooth and affine diagonal.
	
	(iii) Let $q_d$ be the quadratic form in $d$ variables given by
	\[
	q_{2n}=\sum_{i=1}^nx_{2i-1}x_{2i}\quad\text{and}\quad q_{2n+1}=x_0^2+\sum_{i=1}^nx_{2i-1}x_{2i}.
	\]
	Then $q_{2n}$ is nondegenerate over any ring, and $q_{2n+1}$ is nondegenerate over $R$ if and only if $2\in R^\times$.
	Let $\mathrm O_d\subset \GL_d$ be the automorphism group scheme of $(\Z^d,q_d)$.
	Then $\mathrm O_{2n}$ is smooth over $\Z$ and $\mathrm O_{2n+1}$ becomes smooth over $\Z[\tfrac 12]$ \cite[Theorem C.1.5]{Conrad}.
	Moreover, every nondegenerate quadratic form of rank $d$ is étale-locally isomorphic to $q_d$ \cite[Lemma C.2.1]{Conrad}.
	Hence, we have 
	\[
	\Vect^\qu_{2n}\simeq \B\rm O_{2n}\quad\text{and}\quad \Vect^\qu_{2n+1}\simeq \Spec(\Z[\tfrac 12])\times \B\rm O_{2n+1},
	\] 
	both of which are smooth algebraic stacks with smooth and affine diagonal.
\end{proof}

\begin{remark}
	\leavevmode
	\begin{enumerate}
		\item The diagonal of $\Vect^\sym$ is not flat. For example, the automorphism group scheme of the hyperbolic symmetric form is not flat over $\Z_{(2)}$, since it is the $3$-dimensional group $\SL_2$ over $\F_2$ and the $1$-dimensional group $\G_\m\rtimes \C_2$ over $\Z[\tfrac 12]$.
		\item As observed in \cite[Proposition 5.6]{HJNY}, the presheaf $\Vect^\mathrm{ev}$ classifying nondegenerate even forms on classical affine schemes is not an algebraic stack, as it does not satisfy closed gluing.
	\end{enumerate}
\end{remark}

\begin{theorem}[Stability under base change]
	\label{thm:BC}
	The motivic spectra $\KO_{\scr L}$, $\KO^\varsigma_{\scr L}$, and $\KO^\qoppa_{\scr L}$ are stable under base change.
	More precisely, for any morphism of derived algebraic spaces $f\colon Y\to X$ and any $\scr L\in \Pic^\dagger(X)$, the canonical map
	\[
	f^*(\KO_{\scr L}) \to \KO_{f^*(\scr L)}
	\]
	is an isomorphism in $\MS_Y$, and similarly for $\KO^\varsigma_{\scr L}$ and $\KO^\qoppa_{\scr L}$.
\end{theorem}

\begin{proof}
	The assertion is Nisnevich-local on $X$, so we can assume $X=\Spec(\Z)$. Since $f^*$ commutes with $\Sigma_{\P^1}$, we can also assume $\scr L=\scr O$.
	By the Bass delooping theorem \cite[Corollary 4.13(i)]{AHI}, it further suffices to prove the analogous statement for $\Omega^\infty$ of these motivic spectra in $\MS^\mathrm{un}$.
	Since
	\[
	\Omega^{\infty-4n}_{\P^1}\Omega^\infty\KO = \Omega^\infty\c\GW^{\geq 2n}\quad\text{and}\quad \Omega^{\infty-4n}_{\P^1}\Omega^\infty\KO^\qoppa = \Omega^\infty\c\GW^{\geq\infty},
	\]
	the result for $\KO$ and $\KO^\qoppa$ follows from Corollary~\ref{cor:GW-LKE}.
	The result for $\KO^\varsigma$ follows since $\KO^\varsigma=\KO[\beta^{-1}]$, by Proposition~\ref{prop:bott}(i).
\end{proof}

\begin{corollary}
	The motivic spectra $\KW_{\scr L}$, $\KW^\qoppa_{\scr L}$, and $\KW^\varsigma_{\scr L}$ are stable under base change.
\end{corollary}

\begin{proof}
	This follows from Theorem~\ref{thm:BC} and Proposition~\ref{prop:eta-localization}.
\end{proof}

\begin{corollary}
	The motivic spectrum $(\KO_{\scr L})^\wedge_\beta$ is stable under base change.
\end{corollary}

\begin{proof}
	This follows from Theorem~\ref{thm:BC} and Remark~\ref{rmk:normal}.
\end{proof}

\begin{remark}
	By Theorem~\ref{thm:BC} and Proposition~\ref{prop:genuine-vs-derived}, the motivic spectrum $\KO^\g_{\epsilon\scr L}$ is stable under base change between classical algebraic spaces if $\scr L\in \Pic^\dagger$ and $\epsilon=(-1)^{\deg\scr L}$. If however $\epsilon=-(-1)^{\deg\scr L}$, this only holds on $2$-torsionfree algebraic spaces (or over $\F_2$). 
	For example, the motivic spectrum $\KO^\g_{-\scr O}$, whose underlying sheaf of spectra is genuine skew-symmetric Grothendieck–Witt theory, is not stable under base change from $\Z$ to $\F_2$, since it is isomorphic to $\Sigma^2_{\P^1}\KO$ over $\Z$ and to $\KO$ over $\F_2$.
\end{remark}

\begin{remark}
	After $\A^1$-localization and over classical schemes, Theorem~\ref{thm:BC} for $\KO^\varsigma$ recovers \cite[Proposition 8.2.2]{CHN}.
\end{remark}

\section{Thom isomorphisms and the metalinear orientation}
\label{sec:Thom}

Consider the $\E_\infty$-map
\[
\br\colon \K\xrightarrow{(\rk,\det^\dagger)} \Z \times \Pic^\dagger \to \Br^\c.
\]
The goal of this section is to prove the Thom isomorphism
\[
\Sigma^{\xi} \scr R^\c_X(\ph)_{?} \simeq \scr R^\c_X(\ph)_{?\otimes \br(\xi)}
\]
for any $X\in\dSpc$ and $\xi\in \K(X)$, and similarly for $\scr R^\g_X$ when $X$ is bounded. If $\xi$ comes from a finite locally free sheaf $\scr E\in\Vect(X)$, this essentially follows from the projective bundle formula. We then show that the Thom isomorphism is symmetric monoidal in $\scr E$, so that it extends to K-theory; this step is rather technical due to the intricate definition of the symmetric monoidal structure on $\scr E\mapsto\Sigma^{\scr E}$.

\begin{definition}[Thom class]
	\label{def:Thom}
	Let $X$ be a derived algebraic space and let $\scr E\in\Vect(X)$ be a finite locally free sheaf with shifted determinant $\omega\in \Pic^\dagger(X)$. 
	Let $p\colon \V(\scr E)\to X$ be the associated vector bundle with zero section $e\colon X\to \V(\scr E)$.
	By Grothendieck duality, the perfect sheaf $e_*(\scr O_X)$ with the symmetric form
	\[
	\Sym^2(e_*(\scr O_X))\to e_*(\Sym^2(\scr O_X)) = e_*(\scr O_X) = e_*e^!p^!(\scr O_X) \to p^!(\scr O_X)=p^*(\omega)
	\]
	is a Poincaré object of $(\Perf(\V(\scr E)\;\mathrm{on}\; X),\Qoppa^\sym_{\omega})$. It defines an element
	\[
	\th(\scr E)\in \GW(\Perf(\V(\scr E)\;\mathrm{on}\; X),\Qoppa^\sym_{\omega}),
	\]
	called the \emph{Thom class} of $\scr E$.
\end{definition}

\begin{theorem}[Thom isomorphism for additive invariants]
	\label{thm:Thom}
	Let $X$ be a qcqs derived algebraic space and let $\scr E\in\Vect(X)$ be a finite locally free sheaf with shifted determinant $\omega\in \Pic^\dagger(X)$. Let $\scr C$ be a stable category and suppose one of the following holds:
	\begin{enumerate}
		\item $\scr L\in \Pic^\dagger(X)$, $*\in\{\c,\varsigma,\qoppa\}$, and $\scr F\colon \Mod_{(X,\Qoppa^{\c})}(\Cat^\p)\to\scr C$ is an additive invariant;
		\item $\scr L\in \Pic^\dagger(X)^{\B\C_2}$, $*\in\{\g,\s,\q\}$, and $\scr F\colon \Mod_{(X,\Qoppa^{\g})}(\Cat^\p)\to\scr C$ is an additive invariant.
	\end{enumerate}
	Then multiplication with the Thom class $\th(\scr E)\in \GW^\c_{\omega}(\V(\scr E) \;\mathrm{on}\; X)$ induces a fiber sequence
	\[
	\scr F^*_{\scr L}(X) \to \scr F^*_{\scr L\otimes\omega}(\P(\scr E\oplus\scr O)) \to \scr F^*_{\scr L\otimes\omega}(\P(\scr E)).
	\]
	If the rank of $\scr E$ is odd, it is canonically split.
\end{theorem}

\begin{proof}
	We can assume that the rank $r$ of $\scr E$ is constant and we apply the projective bundle formula, more precisely the cases (ii) and (iv) of Theorems \ref{thm:PBF} and~\ref{thm:derived-PBF}.
	Suppose first that $r$ is odd, and consider the diagram
	\[
	\begin{tikzcd}
		& \scr F^*_{\scr L\otimes\omega}(X)\oplus \scr F^\mathrm{hyp}(X)^{\oplus (r-1)/2} \ar{d} \ar[equal]{r} & \scr F^*_{\scr L\otimes\omega}(X)\oplus \scr F^\mathrm{hyp}(X)^{\oplus (r-1)/2} \ar{d}[sloped]{\sim} \\
		\scr F^*_{\scr L}(X) \ar{r}{e_*} \ar[dashed]{dr}[swap]{\id} & \scr F^*_{\scr L\otimes\omega}(\P(\scr E\oplus\scr O)) \ar{d}{p_*(\ph\otimes\scr O(-\frac{r+1}{2}))} \ar{r} & \scr F^*_{\scr L\otimes\omega}(\P(\scr E)) \\
		& \scr F^*_{\scr L}(X)\rlap, & 
	\end{tikzcd}
	\]
	where the middle column is a fiber sequence and the right vertical map is an isomorphism. 
	The lower left triangle commutes since it is the image of a commuting triangle in $\Cat^\p$:
	\[
	\begin{tikzcd}
		& (\Perf(\P(\scr E\oplus\scr O),\Qoppa^*_{\scr L\otimes \omega}) \ar{d}{(-\frac{r+1}2)}[swap,sloped]{\sim} \\
		(\Perf(X),\Qoppa^*_{\scr L}) \ar{ur}{e_*} \ar{r}{e_*} \ar{dr}[swap]{\id} & (\Perf(\P(\scr E\oplus\scr O),\Qoppa^*_{\scr L\otimes \omega(-r-1)}) \ar{d}{p_*} \\
		& (\Perf(X),\Qoppa^*_{\scr L})\rlap.
	\end{tikzcd}
	\]
	This shows that the given nullsequence is a split fiber sequence.
	
	Suppose now $r$ even. If $r=0$ there is nothing to prove, so assume $r\geq 2$. We consider the diagram
	\[
	\begin{tikzcd}
		& \scr F^*_{\scr L\otimes\omega}(X)\oplus \scr F^\mathrm{hyp}(X)^{\oplus (r-2)/2} \ar{d} \ar[equal]{r} & \scr F^*_{\scr L\otimes\omega}(X)\oplus \scr F^\mathrm{hyp}(X)^{\oplus (r-2)/2} \ar{d} \\
		\scr F^*_{\scr L}(X) \ar{r}{e_*} \ar[dashed]{d}[swap]{\id} & \scr F^*_{\scr L\otimes\omega}(\P(\scr E\oplus\scr O)) \ar[shift right]{d} \ar[<-,shift left]{d}{p^*(\ph)\otimes\scr O(\frac r2)} \ar{r} & \scr F^*_{\scr L\otimes\omega}(\P(\scr E)) \ar{d}{p_*(\ph\otimes\scr O(-\frac{r}{2}))} \\
		\scr F^*_{\scr L}(X) \ar{r}{\mathrm{forget}} & \scr F^\mathrm{hyp}(X) \ar{r}{\mathrm{hyp}} & \scr F^*_{\scr L[1]}(X)\rlap,
	\end{tikzcd}
	\]
	where the middle column is a split fiber sequence and the last column is a fiber sequence. 
	We claim that the identity of $\scr F^*_{\scr L}(X)$ participates in a morphism of nullsequences between the last two rows. Since the third row is a fiber sequence, this will complete the proof. The last two rows can in fact be promoted to a morphism of nullsequences in $\Cat^\p$:
	\[
	\begin{tikzcd}
		(\Perf(X),\Qoppa^*_\scr L) \ar{r}{e_*} \ar[equal]{d} & (\Perf(\P(\scr E\oplus\scr O)),\Qoppa^*_{\scr L\otimes\omega}) \ar{d}{q\to qi_*i^*} \ar{r}{i^*} & (\Perf(\P(\scr E)),\Qoppa^*_{\scr L\otimes\omega}) \ar{d}{qi_*} \\
		(\Perf(X),\Qoppa^*_\scr L) \ar{r}{\id\to 0} & (\mathrm{Ar}(\Perf(X)),{\cofib}\circ {\Qoppa^{*}_{\scr L}}) \ar{r}{\mathrm{target}} & (\Perf(X),\Qoppa^*_{\scr L[1]})\rlap,
	\end{tikzcd}
	\]
	Here, $i\colon \P(\scr E)\to \P(\scr E\oplus\scr O)$ is the canonical closed immersion, $q$ is the functor $p_*(\ph\otimes\scr O(-\tfrac r2))$, and the lower row is the metabolic sequence. Using $p^!(\scr O)=\omega(-r-1)$ and the fiber sequence $\scr F(-1)\to \scr F\to i_*i^*(\scr F)$, one can easily check that the middle vertical map is a Poincaré functor making both squares commute.
\end{proof}

\begin{corollary}\label{cor:Thom}
	Let $X$ be a derived algebraic space and let $\scr E\in\Vect(X)$ be a finite locally free sheaf with shifted determinant $\omega\in \Pic^\dagger(X)$. 
	\begin{enumerate}
		\item For every $\scr L\in \Pic^\dagger(X)$, the Thom class $\th(\scr E)$ induces an isomorphism
	\[
	\Sigma^{\scr E}\scr R_X^\c(\ph)_{\scr L} \simeq \scr R_X^\c(\ph)_{\scr L\otimes\omega}
	\]
	in $\MS_X$, and similarly for $\scr R_X^\qoppa$ and $\scr R_X^\varsigma$.
	\item Suppose $X$ bounded. For every $\scr L\in \Pic^\dagger(X)^{\B\C_2}$, the Thom class $\th(\scr E)$ induces an isomorphism
	\[
	\Sigma^{\scr E}\scr R_X^\g(\ph)_{\scr L} \simeq \scr R_X^\g(\ph)_{\scr L\otimes\omega}
	\]
	in $\MS_X$, and similarly for $\scr R_X^\q$ and $\scr R_X^\s$.
	\end{enumerate}
\end{corollary}

\begin{proof}
We have to show that, for every qcqs $U\in \Sm_X$ and every $n\in \Z$, the map 
\[
\th(\scr E)\colon \scr F^*_{\scr L[n]}(U)\to \fib\bigl(\scr F^*_{\scr L[n]\otimes\omega}(\P(\scr E\oplus\scr O))\to \scr F^*_{\scr L[n]\otimes\omega}(\P(\scr E))\bigr)
\]
is an isomorphism. This is a special case of Theorem~\ref{thm:Thom}.
\end{proof}

\begin{remark}[Twisted Thom isomorphism]
	\label{rmk:twisted-Thom}
	The split Poincaré–Verdier sequences and isotropic subcategories appearing in the proof of Theorem~\ref{thm:Thom} can be twisted by classical Poincaré–Azumaya algebras using étale descent (Proposition~\ref{prop:descent-animated}(iv)). This yields the following stronger form of the Thom isomorphism: 
	for any $\alpha\in\Br^\c(X)$, multiplication with the Thom class $\th(\scr E)$ induces a fiber sequence
	\[
	\scr F^\c_{\alpha}(X) \to \scr F^\c_{\alpha\otimes \br(\scr E)}(\P(\scr E\oplus\scr O)) \to \scr F^\c_{\alpha\otimes \br(\scr E)}(\P(\scr E)),
	\]
	where $\scr F^\c_{\smash[b]{(\scr A,\scr A_{\geq 0},\phi)}}(X)=\scr F(\scr A(X)^\omega,\c\Qoppa_\phi^{\geq 0})$.
	If $X$ is bounded, the same holds for $\scr F^\g_\alpha$ for any $\alpha\in\Br^\g(X)$, using Proposition~\ref{prop:mod-descent}.
	We deduce as in Corollary~\ref{cor:Thom} an isomorphism $\Sigma^{\scr E}\scr R_X^\c(\ph)_{?} \simeq \scr R_X^\c(\ph)_{?\otimes\br(\scr E)}$ in $\MS_X$, and similarly for $\scr R^\g_X$ when $X$ is bounded.
\end{remark}

\begin{remark}
	On classical schemes and after applying $\A^1$-localization, Corollary~\ref{cor:Thom}(ii) for $\KO^\s_{\scr L}$ recovers the Thom isomorphism from \cite[Proposition 8.3.1]{CHN}.
\end{remark}

\begin{remark}[Gorenstein transfers]
	\label{cst:transfers}
The Thom class of Definition~\ref{def:Thom} is an instance of a more general covariant functoriality of classical Poincaré categories.
	Given a derived algebraic space $X$ and a quasi-coherent sheaf $\scr L$ on $X$, consider the quadratic functor
	\[
	\Qoppa^\sym_\scr L=\map(\Sym^2(\ph),\scr L)\colon \QCoh(X)^\op\to\Sp.
	\]
	Since $\Sym^2\colon \QCoh\to\QCoh$ is oplax symmetric monoidal, the assignment $(X,\scr L)\mapsto (\QCoh(X),\Qoppa^\sym_{\scr L})$ unfurls to a lax symmetric monoidal functor
	\[
	(\QCoh,\Qoppa^\sym)\colon \Span^{\QCoh,\op} \to \Cat^\h,
	\]
	where $\Span^\QCoh\to \Span(\dSpc)$ is the cartesian fibration classified by $\QCoh^\op$.
	Explicitly, a morphism in $\Span^\QCoh$ is a span
	\[
	\begin{tikzcd}[column sep=1.8em]
	   & Z \ar[swap]{ld}{p}\ar{rd}{f} & \\
	  (X,\scr L) &   & (Y,\scr M)
	\end{tikzcd}
	\]
	together with a map $\alpha\colon p_*f^*(\scr M)\to\scr L$. This induces the morphism
	\[
	(p_*f^*, \alpha_*)\colon (\QCoh(Y), \Qoppa^\sym_{\scr M}) \to (\QCoh(X),\Qoppa^\sym_{\scr L}),
	\]
	where $\alpha_*\colon \Qoppa^\sym_{\scr M} \to \Qoppa^\sym_{\scr L}\circ p_*f^*$ is the composite
	\[
	\begin{tikzcd}
		\map(\Sym^2(\ph),\scr M) \ar{r}{p_*f^*} & \map(p_*f^*\Sym^2(\ph),p_*f^*(\scr M)) \ar{d} \\
		& \map(\Sym^2(p_*f^*(\ph)),p_*f^*(\scr M)) \ar{r}{\alpha} & \map(\Sym^2(p_*f^*(\ph)),\scr L)\rlap.
	\end{tikzcd}
	\]
	The span $(p,f,\alpha)$ above is called \emph{Gorenstein} if:
	\begin{itemize}
		\item $p$ is proper, almost of finite presentation, and of finite Tor-amplitude,
		\item $\scr L$ and $\scr M$ are invertible,
		\item $\alpha$ induces an isomorphism $f^*(\scr M)\simeq p^!(\scr L)$.
	\end{itemize}
	We denote by $\Span^\mathrm{Gor}\subset \Span^\QCoh$ the subcategory of Gorenstein spans. The assumptions on $p$ imply that $p_*$ preserves perfect sheaves \cite[Theorem 6.1.3.2]{SAG}, and the assumptions on $\scr L$, $\scr M$, and $\alpha$ imply that the hermitian functor $(p_*f^*, \alpha_*)\colon (\Perf(Y), \Qoppa^\sym_{\scr M}) \to (\Perf(X),\Qoppa^\sym_{\scr L})$ is a Poincaré functor, so that we obtain a lax symmetric monoidal functor
	\[
	(\Perf,\Qoppa^\sym)\colon \Span^{\mathrm{Gor},\op} \to \Cat^\p.
	\]
Moreover, the Poincaré functor $\Sigma^{-2}\colon (\Perf(X),\Qoppa^\sym_{\scr L[4]})\to (\Perf(X),\Qoppa^\sym_{\scr L})$ induced by the norm $\Sym^2\to\Gamma^2$ is natural with respect to this functoriality. Passing to the limit and to the colimit, we obtain functors
\[
(\Perf,\Qoppa^\qoppa),\, (\Perf,\Qoppa^\varsigma)\colon \Span^{\mathrm{Gor},\op} \to \Cat^\p.
\]

A similar construction with $(\ph)^{\otimes 2}$ instead of $\Sym^2$ yields functors
	\[
	(\Perf,\Qoppa^\q),\, (\Perf,\Qoppa^\s)\colon \Span^{\C_2\mathrm{Gor},\op} \to \Cat^\p,
	\]
	where the objects of $\Span^{\C_2\mathrm{Gor}}$ are pairs $(X,\scr L)$ with $X\in \SpSpc$ and $\scr L\in \Pic^\dagger(X)^{\B\C_2}$ (this covariant functoriality of $(\Perf,\Qoppa^\s)$ is a special case of \cite[Construction 5.1.3]{CHN}).
	However, the functoriality of $(\Perf,\Qoppa^\g)$ is more subtle: it is a priori only a functor on the subcategory of $\Span^{\C_2\mathrm{Gor}}$ where $p$ is flat, so that $p_*\colon\QCoh(Z)\to \QCoh(X)$ decreases connectivity by at most $\deg(\omega_p)$.
\end{remark}

We now discuss the multiplicativity of the Thom isomorphism.
Note that the Thom class $\th(\scr E)$ is symmetric monoidal in $(X,\scr E)$, and in fact $(X,\scr E)\mapsto e_*(\scr O)$ is a symmetric monoidal cartesian section of the cartesian fibration classified by the lax symmetric monoidal functor $(X,\scr E)\mapsto \mathrm{Pn}(\Perf(\V(\scr E)\;\mathrm{on}\; X),\Qoppa^\sym_{\omega})$. 
The Thom isomorphism of Theorem~\ref{thm:Thom} factors as
\[
\scr F^\c_{\scr L}(X) \xrightarrow{\th(\scr E)} \scr F^\c_{\scr L\otimes\omega}(\V(\scr E)\;\mathrm{on}\; X)
=\scr F^\c_{\scr L\otimes\omega}(\P(\scr E\oplus\scr O)\;\mathrm{on}\; X) \to \fib\bigl(\scr F^\c_{\scr L\otimes\omega}(\P(\scr E\oplus\scr O)) \to \scr F^\c_{\scr L\otimes\omega}(\P(\scr E))\bigr),
\]
so it remains to consider the multiplicativity of the second map.
At least if $\scr F$ is a Karoubi-localizing invariant, so that $\scr F^\c_{\scr L}$ satisfies smooth blowup excision (Corollary~\ref{cor:animated-ebu}), then the target of the Thom isomorphism is also a lax symmetric monoidal functor of $(X,\scr E,\scr L,\scr F)$ by \cite[\sect 3]{AHI}. Construction~\ref{cst:fake-Thom} below then shows that the second map is a symmetric monoidal transformation.

\begin{definition}
	Let $\scr E$ be a finite locally free sheaf on $X\in\scr P(\Sm_S)$. The \emph{fake Thom space} of $\scr E$ is
	\[
	\fTh_X(\scr E)=\frac{\V(\scr E)}{\V(\scr E)-0} \in \scr P(\Sm_S)_*.
	\]
\end{definition}

The fake Thom space is related to the actual Thom space $\Th_X(\scr E)=\P(\scr E\oplus\scr O)/\P(\scr E)$ by a zigzag
\[
\Th_X(\scr E) \longrightarrow \frac{\P(\scr E\oplus\scr O)}{\P(\scr E\oplus\scr O)-0} \longleftarrow \fTh_X(\scr E),
\]
where the first map is an $\A^1$-equivalence and the second map is a Zariski-local equivalence. 
In particular, there is a canonical map
\[
\Th_X(\scr E) \to\fTh_X(\scr E)
\]
in $\scr P_\Zar(\Sm_S)_*$ and hence in $\MS_S$.

\begin{construction}[Symmetric monoidal comparison of Thom spaces and fake Thom spaces]
	\label{cst:fake-Thom}
	Fake Thom spaces trivially define a lax symmetric monoidal functor
	\[
	\fTh\colon \Vect^{\epi,\op} \to \scr P(\Sm)_*,
	\]
	which becomes strict after Zariski sheafification.
	On the other hand, actual Thom spaces form a symmetric monoidal functor
	\[
	\Th\colon \Vect^{\epi,\op} \to \scr P_{\ebu}(\Sm)_*,
	\]
	which was constructed in \cite[\sect 3]{AHI}. We will extend the latter construction to a symmetric monoidal transformation
	\[
	\Th \to \fTh\colon \Vect^{\epi,\op}\to \scr P_{\Zar,\ebu}(\Sm)_*.
	\]
	This is essentially trivial, but we have to contend with the complicated construction of the symmetric monoidal structure on $\Th$.
	Using the notation from \emph{loc.\ cit.}, we will construct a simplicial map
	\begin{equation}\label{eqn:Thom-monoidal}
	\mathrm N(\Vect^{\epi,\op,\otimes}) \to \Lambda((\mathrm{Cube}^{s}(\Sm^{0\to 1\from 2})^{\otimes})^\vee,\mathrm{all},\mathrm{vert}),
	\end{equation}
	whose restriction to $0$ is the map \cite[(3.11)]{AHI} encoding the symmetric monoidal structure of $\Th$, and whose restriction to $2$ encodes the (lax) symmetric monoidal structure of $\fTh$. 
	The map~\eqref{eqn:Thom-monoidal} will have the following features: $0\to 1$ induces isomorphisms on the final vertices of the cubes, $1\from 2$ consists of open immersions, and all the cubes over $1$ and $2$ themselves consist of open immersions. Since open immersions are monomorphisms, the map~\eqref{eqn:Thom-monoidal} will be uniquely determined once these open immersions are suitably specified, so that we will be able to define it with minimal knowledge of \cite[(3.11)]{AHI}.
	
	For a finite family of finite locally free sheaves $\scr E=(\scr E_i)_{i\in I}$, recall that the smooth scheme $\bb B(\scr E)$ para\-met\-rizes $I$-cubes of invertible sheaves $(\scr L_J)_{J\subset I}$ with a surjective map from the $I$-cube $(\bigoplus_{j\in J}\scr E_i\oplus \scr O)_{J\subset I}$, and there is a zigzag of maps
	\[
	\prod_{i\in I}\P(\scr E_i\oplus\scr O) \xleftarrow{b_\Pi} \bb B(\scr E) \xrightarrow{b_\P} \P\bigl({\textstyle \bigoplus_{i\in I}\scr E_i\oplus\scr O}\bigr),
	\]
	which are sequences of smooth blowups \cite[Proposition 3.6]{AHI}. Given a subset $K\subset I$, let $U_K\subset \bb B(\scr E)$ denote the preimage under $\pi_K\circ b_\Pi$ of the open subset $\prod_{i\in K}(\P(\scr E_i\oplus \scr O)-0)\subset \prod_{i\in K}\P(\scr E_i\oplus \scr O)$.
	
	The map \cite[(3.11)]{AHI} sends each $n$-simplex to some $n$-span of families of cubes that are relative strict normal crossings divisors on smooth schemes of the form $\prod_{j\in J}\bb B(\scr E|_{I_j})$, for some finite family of finite locally free sheaves $(\scr E_i)_{i\in I}$ and some map of finite sets $I\to J$. The smooth components of this divisor are indexed by pairs $(j,K)$ with $j\in J$ and $K\subset I_j$ nonempty, and the component indexed by $(j,K)$ is contained in the open subset $\pi_j^{-1}(U_K)$. The cube given by the divisor thus maps to the strongly cartesian cube of open subsets of $\prod_{j\in J}\bb B(\scr E|_{I_j})$ extending $(j,K)\mapsto \pi_j^{-1}(U_K)$, which defines~\eqref{eqn:Thom-monoidal} over $0\to 1$. We extend it over $1\from 2$ by simply intersecting the latter cube with the open subset
	\[
	\prod_{j\in J}\bb B(\scr E|_{I_j})\supset\prod_{j\in J}\V\bigl({\textstyle\bigoplus_{i\in I_j}\scr E_i}\bigr).
	\]
	
	Since the maps $b_\Pi$ and $b_\P$ above are mutually inverse isomorphisms over $\prod_{i\in I}\V(\scr E_i)=\V(\bigoplus_{i\in I}\scr E_i)$, the restriction of~\eqref{eqn:Thom-monoidal} to $2$ lands in
	\[
	\Lambda((\mathrm{Cube}^{s}(\Sm)^{\otimes})^\vee,\mathrm{cart},\mathrm{vert}) = \rm N(\mathrm{Cube}^{s}(\Sm)^{\otimes})
	\]
	and induces the canonical lax symmetric monoidal structure of $\fTh$ after composing with the symmetric monoidal functor $\mathrm{tcofib}\colon \mathrm{Cube}^{s}(\Sm)\to \scr P(\Sm)_*$.
	Moreover, the open immersions of cubes over $1\from 2$ become Zariski-local equivalences after applying $\mathrm{tcofib}$, which completes the desired construction.
\end{construction}

\begin{remark}
	Construction~\ref{cst:fake-Thom} shows that the $\E_\infty$-ring structure on $\MGL$ constructed in \cite[\sect 7]{AHI} induces the one on $\L_{\A^1}\MGL$ constructed in \cite[\sect 16]{norms}.
\end{remark}

\begin{construction}[Virtual Thom classes]
	\label{cst:virtual-Thom}
	Consider the lax symmetric monoidal functors
	\[
	\Sigma^{(\ph)}\KO,\; \KO_{\det^\dagger(\ph)}\colon \K(S) \to \MS_S.
	\]
By Construction~\ref{cst:fake-Thom}, the Thom isomorphism of Corollary~\ref{cor:Thom}(i) for $\GW$ identifies their restrictions to $\Vect(S)$, naturally in $S$.
Since $\Sigma^{(\ph)}\KO$ is strict symmetric monoidal
and since the lax $\Z$-linear structure of $\KO_{\det^\dagger(\ph)}$ is strict, we deduce that $\KO_{\det^\dagger(\ph)}$ is also strict symmetric monoidal. Hence, the Thom isomorphism extends uniquely from $\Vect$ to $\K$: for every $\xi\in\K(S)$, there is an isomorphism of $\KO$-modules
\[
\th(\xi)\colon \Sigma^{\xi}\KO\simeq \KO_{\det^\dagger(\xi)}
\]
in $\MS_S$, which is symmetric monoidal in $\xi\in\K(S)$ and natural in $S\in\dSpc$. In other words, we have the following triangle of presheaves of symmetric monoidal categories:
\[
\begin{tikzcd}[column sep={3em, between origins}]
	\K \ar{dr}[swap]{\Th} \ar{rr}{\br} & {} \ar[phantom]{d}[near start]{\overset{\th}\Longrightarrow} & \Br^\c \ar{dl}{\KO_?} \\
	& \MS\rlap. &
\end{tikzcd}
\]
\end{construction}

\begin{proposition}[Motivic Thom isomorphism]
	\label{prop:virtual-Thom}
	Let $X$ be a derived algebraic space and let $\xi\in\K(X)$ have shifted determinant $\omega\in \Pic^\dagger(X)$.
	\begin{enumerate}
	\item For every $\scr L\in \Pic^\dagger(X)$, the Thom class $\th(\xi)$ induces an isomorphism
	\[
	\Sigma^{\xi}\scr R_X^\c(\ph)_{\scr L} \simeq \scr R_X^\c(\ph)_{\scr L\otimes\omega}
	\]
	in $\MS_X$, and similarly for $\scr R_X^\qoppa$ and $\scr R_X^\varsigma$.
	\item Suppose $X$ bounded. For every $\scr L\in \Pic^\dagger(X)^{\B\C_2}$, the Thom class $\th(\xi)$ induces an isomorphism
	\[
	\Sigma^{\xi}\scr R_X^\g(\ph)_{\scr L} \simeq \scr R_X^\g(\ph)_{\scr L\otimes\omega}
	\]
	in $\MS_X$, and similarly for $\scr R_X^\q$ and $\scr R_X^\s$.
	\end{enumerate}
\end{proposition}

\begin{proof}
	Nisnevich-locally on $X$, there exists $n\geq 0$ such that $\xi+\scr O^n$ is in the image of $\Vect(X)$. The result thus follows from Corollary~\ref{cor:Thom}.
\end{proof}

\begin{remark}
More generally, the Thom class $\th(\xi)$ and the symmetric monoidal structure of $\scr R^\c_X$ induce an $\scr R_X^\c(\ph)_?$-linear map in $\Fun(\Br^\c(X),\MS_X)$
\[
\Sigma^\xi\scr R_X^\c(\ph)_? \xrightarrow{\th(\xi)} \KO_{\br(\xi)}\otimes\scr R_X^\c(\ph)_? \to\scr R_X^\c(\ph)_{?\otimes \br(\xi)},
\]
which is an isomorphism by Remark~\ref{rmk:twisted-Thom}. A similar statement holds for $\scr R^\g_X$ when $X$ is bounded.
\end{remark}

\begin{corollary}
	Let $X$ be a derived algebraic space and let $\scr L\in \Pic^\dagger(X)$. Then $\KO_{\scr L}$ is an invertible $\KO$-module in $\MS_X$. The same holds for $\KO^\varsigma$ and, if $X$ is bounded, for $\KO^\g$ and $\KO^\s$.
\end{corollary}

\begin{proof}
	This follows from Proposition~\ref{prop:virtual-Thom}, since there is some $\xi\in\K(X)$ with $\det^\dagger(\xi)=\scr L$.
\end{proof}

\begin{remark}
	If $\scr L\in\Pic^\dagger(X)^{\B\C_2}$ has a nontrivial $\C_2$-action, we do not know if $\KO^\g_{\scr L}$ is an invertible $\KO^\g$-module in $\MS_X$.
	We also do not know if $\KO_\alpha$ is an invertible $\KO$-module for all $\alpha\in \Br^\c(X)$.
\end{remark}

\begin{corollary}
	Let $X$ be an affine derived scheme, let $\xi\in\K(X)$ have rank $0$, and let $n\geq 0$. Then
	\[
	\Map_{\MS_X}(\1_X,\Sigma^{\xi}\Sigma_{\P^1}^{2n} \KO) = \Vect(X,\c\Qoppa^{\geq n}_{(-1)^n\det(\xi)})^\gp.
	\]
	For $n=0,1,2$, we have more explicitly
	\begin{align*}
		\Map(\1_X,\Sigma^\xi \KO) &= \Vect(X,\Qoppa^\sym_{\det(\xi)})^\gp,\\
		\Map(\1_X,\Sigma^{\xi}\Sigma_{\P^1}^{2} \KO) &= \Vect(X,\Qoppa^\alt_{\det(\xi)})^\gp,\\
		\Map(\1_X,\Sigma^{\xi}\Sigma_{\P^1}^{4} \KO) &= \Vect(X,\Qoppa^\qu_{\det(\xi)})^\gp. 
	\end{align*}
\end{corollary}

\begin{proof}
	By Construction~\ref{cst:virtual-Thom}, we have $\Sigma^{\xi}\Sigma_{\P^1}^{2n} \KO\simeq \KO_{\det(\xi)[2n]}$, which by definition represents classical $\det(\xi)[2n]$-valued GW-theory.
	The claim follows from the suspension isomorphism 
	\[
	\Sigma^{-n}\colon (\Perf(X),\Qoppa^\c_{\scr L[2n]})\simeq (\Perf(X),\c\Qoppa_{(-1)^n\scr L}^{\geq n})
	\]
	for $\scr L\in\Pic(X)$ and the group completion theorem.
\end{proof}

\begin{construction}[The metalinear $\E_\infty$-orientation of $\KO$]
	We explain how the virtual Thom classes of Construction~\ref{cst:virtual-Thom} induce an $\E_\infty$-map $\MML\to\KO$.
	Define the presheaf of $\E_\infty$-groups $\K^\ML$ on derived algebraic spaces by the pullback
	\[
	\begin{tikzcd}
		\K^\ML \ar{r} \ar{d} & \Pic^\dagger \ar{d}{2} \\
		\K \ar{r}{\det^\dagger} & \Pic^\dagger\rlap. & 
	\end{tikzcd}
	\]
	A square root of $\det^\dagger(\xi)$ induces an isomorphism $\br(\xi) \simeq (\scr Q,\scr Q_{\geq r},\phi_{(-1)^{r}\scr O})$ in $\Br^\c$, where $r=\rk(\xi)/2$, so that the restriction of $\br$ to $\K^\ML$ factors through $\Z_{\pm\infty}\times\C_2$. Consider the following diagram of presheaves of symmetric monoidal categories (where $\KSp=\Sigma^2_{\P^1}\KO$, see Remark~\ref{rmk:KSp}):
	\[
	\begin{tikzcd}[column sep={4.5em, between origins}]
		\K^\ML \ar{d} \ar{rr}{{\rk}/2} &  & \{(r,(-1)^r)\in \Z_{\pm\infty}\times\C_2\} \ar{d}{(\scr Q,\scr Q_{\geq r},\phi_{(-1)^r\scr O})} \ar[shift left=2pt]{rr}{\pi_2} & & \C_2 \ar[shift left=2pt,hook]{ll}{\iota} \ar[bend left=25]{ddlll}[description,pos=0.35]{(\KO,\KSp)} \ar{r} \ar[phantom]{dd}[pos=0.35]{\Longrightarrow} & *\rlap. \ar[bend left=30,to=MS.south east]{ddllll}[pos=0.3]{\KO\oplus\KSp} \\
		\K \ar{dr}[swap]{\Th} \ar{rr}{\br} & {} \ar[phantom]{d}[near start]{\overset{\th}\Longrightarrow} & \Br^\c \ar{dl}{\KO_?} & & {} \\
		& |[alias=MS]| \MS & & & {}
	\end{tikzcd}
	\]
	Here, $\iota$ is the lax symmetric monoidal functor that picks out $(0,1)$ and $(1,-1)$, which is right adjoint to the restriction of $\pi_2$ to nonnegative ranks.
	 The above diagram together with the unit of this adjunction defines a symmetric monoidal functor $\Th\colon \K^\ML_{\geq 0}\to \MS_{/\KO\oplus\KSp}$,
	whose restriction to $\rk^{-1}(4\N)$ factors through $\MS_{/\KO}$.
	Using the further diagram
	\[
	\begin{tikzcd}
		\{(r,(-1)^r)\in \Z_{\pm\infty}\times\C_2\} \ar{r}{-\infty} \ar{d} & \{-\infty\}\times\C_2 \ar{d} \ar{r} & * \ar{d} \\
		\Br^\c \ar{r}{-\infty} \ar{d}[swap]{\KO_?} \ar[phantom]{dr}{\Longrightarrow} & \Br^\c \ar{d}[swap]{\KO_?} \ar{r}{\mathrm{forget}} \ar[phantom]{dr}{\Longrightarrow} & \Br^\dagger \ar{d}{\KGL_?} \\
		\MS \ar[equal]{r} & \MS \ar[equal]{r} & \MS\rlap,
	\end{tikzcd}
	\]
	we obtain symmetric monoidal functors
		\[
		\begin{tikzcd}
			\rk^{-1}(4\N) \ar{r}{\Th} \ar{d} & \MS_{/\KO} \ar{d} \ar[phantom]{dr}{\longrightarrow} & \K^\ML_{\geq 0} \ar{r}{\Th} \ar{d} & \MS_{/\KO\oplus\KSp} \ar{d} \\
			\rk^{-1}(4\Z) \ar{r}{\Th} & \MS_{/\KO^\varsigma} & \K^\ML \ar{r}{\Th} \ar{d} & \MS_{/\KO^\varsigma\oplus\KSp^\varsigma} \ar{d} \\
			& & \K \ar{r}{\Th} & \MS_{/\KGL}\rlap.
		\end{tikzcd}
		\]
	Taking their $\Sm$-parametrized colimits yields the following commutative diagram of absolute motivic $\E_\infty$-ring spectra:
	\[
	\begin{tikzcd}
		\bigoplus_{n\geq 0}\Sigma_{\P^1}^{4n}\MML \ar{r} \ar{d} & \KO \ar{d} \ar[phantom]{drr}{\longrightarrow} & & \bigoplus_{n\geq 0}\Sigma_{\P^1}^{2n}\MML \ar{r} \ar{d} & \KO\oplus\KSp \ar{d} \\
		\bigoplus_{n\in \Z}\Sigma_{\P^1}^{4n}\MML \ar{r} & \KO^\varsigma & {} & \bigoplus_{n\in \Z}\Sigma_{\P^1}^{2n}\MML \ar{r} \ar{d} & \KO^\varsigma\oplus\KSp^\varsigma \ar{d} \\
		 & & & \llap{$\PMGL={}$} \bigoplus_{n\in \Z}\Sigma_{\P^1}^{n}\MGL \ar{r} & \KGL\rlap.
	\end{tikzcd}
	\]
\end{construction}

\begin{remark}
	In $\A^1$-homotopy theory, $\E_\infty$-orientations $\PMGL\to \L_{\A^1}\KGL$ and $\MML\to \L_{\A^1}\KO$ were constructed in \cite[\sect 6]{motive-hilb} and \cite[Remark 7.11]{HJNY} using the formalism of framed correspondences (the latter assuming $2\in\scr O^\times$, but one can now show $\Sigma^\infty_\mathrm{fr}\Vect^\sym[\beta^{-1}]=\L_{\A^1}\KO$ over any derived algebraic space using Theorem~\ref{thm:A1-localization}, so this assumption can be removed). The above constructions are more elementary, and it is not difficult to show that they recover the $\E_\infty$-maps from \emph{loc.\ cit.}, but this would be too much of a digression here. The incoherent metalinear orientation of $\KO$ over regular $\Z[\tfrac 12]$-schemes was previously constructed by Panin and Walter \cite[Theorem 5.1]{PaninWalterBO}.
\end{remark}

\section{\texorpdfstring{$\A^1$}{𝔸¹}-localization}
\label{sec:A1}

In this section, we show that all variants of Grothendieck–Witt theory $\GW^*_{\scr L}$ considered in Sections \ref{sec:genuine} and \ref{sec:derived} have the same $\A^1$-localization. On regular noetherian algebraic spaces of finite Krull dimension, this $\A^1$-localization is symmetric Grothendieck–Witt theory $\GW^\s_{\scr L}$, since the latter was shown to be $\A^1$-invariant by Calmès–\allowbreak Harpaz–\allowbreak Nardin \cite[\sect 6.3]{CHN}. This confirms an expectation from \emph{loc.\ cit}.

Recall that there are two versions of the affine line over the sphere spectrum: the toric affine line $\A^1=\Spec(\bb S[t])$ and the spectral affine line $\A^1_\mathrm{s}=\Spec(\bb S\{t\})$. These are related by a canonical map $\A^1\to \A^1_\mathrm{s}$, which is a morphism of bipointed objects \cite[\sect 3]{CisinskiKhan}. It follows that every $\A^1$-invariant presheaf on spectral algebraic spaces is $\A^1_\mathrm{s}$-invariant, or equivalently that every $\A^1_\mathrm{s}$-equivalence between such presheaves is an $\A^1$-equivalence \cite[Lemma 3.3.3]{CisinskiKhan}.

Let $\scr C$ be a stable category and let $\Qoppa'\to \Qoppa$ be a map of Poincaré structures on $\scr C$ inducing an isomorphism of bilinear parts. The relative L-theory formula of Harpaz–Nikolaus–Shah \cite{HNS} states that
\[
\cofib\bigl(\L(\scr C,\Qoppa')\to \L(\scr C,\Qoppa)\bigr)=\lim\bigl(\Qoppa(\D \lambda)\rightrightarrows \B(\D \lambda,\D \lambda)\bigr),
\]
where $\lambda\in\Ind(\scr C)$ represents the exact functor $\cofib(\Qoppa'\to\Qoppa)\colon \scr C^\op\to\Sp$ and $\D$, $\B$, and $\Qoppa$ are extended to ind-objects in the obvious way.
The claim about the $\A^1$-localization of Grothendieck–Witt theory is a straightforward consequence of this formula:\footnote{This was independently observed by Harpaz and Nikolaus.}

\begin{theorem}[$\A^1$-localization of additive invariants]
	\label{thm:A1-localization}
	\leavevmode
	\begin{enumerate}[beginpenalty=10000]
		\item Let $S$ be qcqs spectral algebraic space, let $\sqrt{\scr L}\in \Pic^\dagger(S)$, and let $\scr L=(\sqrt{\scr L})^{\otimes 2}\in \Pic^\dagger(S)^{\B\C_2}$. For any additive invariant $\scr F\colon \Mod_{(S,\Qoppa^\t)}(\Cat^\p)\to\scr C$ valued in a cocomplete stable category $\scr C$, the canonical map
		\[
		\scr F^\t_{\scr L}\to \scr F^\s_{\scr L}\colon (\SpSpc_S^\qcqs)^\op\to \scr C
		\]
		is an $\A^1_\mathrm{s}$-equivalence and hence an $\A^1$-equivalence.
		\item Let $S$ be qcqs spectral algebraic space, let $\scr L\in \Pic^\dagger(S)^{\B\C_2}$, and let $*\in\{\g,\s,\q\}$. For any additive invariant $\scr F\colon \Mod_{(S,\Qoppa^\g)}(\Cat^\p)\to\scr C$ valued in a cocomplete stable category $\scr C$, the canonical map
		\[
		\scr F^*_{\scr L}\to \scr F^\s_{\scr L}\colon (\SpSpc_S^\qcqs)^\op\to \scr C
		\]
		is an $\A^1_\mathrm{s}$-equivalence and hence an $\A^1$-equivalence.
		\item Let $S$ be qcqs derived algebraic space, let $\scr L\in \Pic^\dagger(S)$, and let $*\in\{\c,\varsigma,\qoppa\}$. For any additive invariant $\scr F\colon \Mod_{(S,\Qoppa^\c)}(\Cat^\p)\to\scr C$ valued in a cocomplete stable category $\scr C$, the canonical map
		\[
		\scr F^*_{\scr L}\to \scr F^\s_{\scr L}\colon (\dSpc_S^\qcqs)^\op\to \scr C
		\]
		is an $\A^1$-equivalence.
	\end{enumerate}
\end{theorem}

\begin{proof}
	Let $\scr F^{*\to\star}_{\scr L}$ denote the cofiber of $\scr F^*_{\scr L}\to \scr F^\star_{\scr L}$. Since $\Qoppa^*_{\scr L}$ is a module over $\Qoppa^\t$ in all cases, $\scr F^{*\to\s}_{\scr L}$ is a module over $\GW^{\q\to\t}=\L^{\q\to\t}$, so it suffices to show that the latter is $\A^1_\mathrm{s}$-contractible. The relative L-theory formula gives
	\begin{align*}
	\L^{\q\to\t}(X) &= \lim\bigl(\Qoppa^\t(\D(\scr O_X))\rightrightarrows \B(\D(\scr O_X),\D(\scr O_X))\bigr) \\
	& = \lim\bigl((\Gamma(X,\scr O_X)^{\h\C_2} \times_{\Gamma(X,\scr O_X)^{\t\C_2}} \Gamma(X,\scr O_X))\rightrightarrows \Gamma(X,\scr O_X)\bigr).
	\end{align*}
	Since $\A^1_\mathrm{s}$-localization preserves colimits and finite limits, and since $\Gamma(\ph,\scr O)^{\h\C_2}$ is a module over $\Gamma(\ph,\scr O)$, it remains to observe that the presheaf of ring spectra $\Gamma(\ph,\scr O)$ on spectral algebraic spaces is $\A^1_\mathrm{s}$-contractible.
	Indeed, since the zero and one sections of $\A^1_\mathrm{s}\to\Spec(\bb S)$ are identified by any $\A^1_\mathrm{s}$-invariant presheaf, the $\A^1_\mathrm{s}$-localization of $\Gamma(\ph,\scr O)$ is a ring in which $0=1$.
\end{proof}

\begin{corollary}
	\label{cor:A1-localization}
	Let $S$ be a derived algebraic space. 
	\begin{enumerate}[beginpenalty=10000]
		\item Suppose $S$ bounded. For any $\scr L\in\Pic^\dagger(S)^{\B\C_2}$, the maps
		\[
		\scr R^\q(\ph)_{\scr L} \to \scr R^\g(\ph)_{\scr L} \to \scr R^\s(\ph)_{\scr L}
		\]
		in $\MS_S$ are $\A^1$-equivalences.
		\item For any $\scr L\in\Pic^\dagger(S)$, the maps
		\[
		\scr R^\qoppa(\ph)_{\scr L} \to \scr R^\c(\ph)_{\scr L} \to \scr R^\varsigma(\ph)_{\scr L}
		\]
		in $\MS_S$ are $\A^1$-equivalences.
	\end{enumerate}
\end{corollary}

\begin{proof}
	This follows immediately from Theorem~\ref{thm:A1-localization}(ii,iii) applied to Karoubi-localizing invariants.
\end{proof}

\begin{corollary}\label{cor:A1-localization2}
	Let $S$ be a regular algebraic space locally of finite Krull dimension. For any $*\in\{\g,\s,\q\}$ and any $\scr L\in\Pic^\dagger(S)^{\B\C_2}$, we have the following isomorphisms in $\MS_S$:
	\[
	\L_{\A^1} \KO^*_{\scr L} = \KO^\s_{\scr L}\quad\text{and}\quad \L_{\A^1} \KW^*_{\scr L} = \KW^\s_{\scr L}.
	\]
\end{corollary}

\begin{proof}
	This follows from Corollary~\ref{cor:A1-localization} and the fact that $\GW^\s_{\scr L}$ and hence $\L^\s_{\scr L}$ are $\A^1$-invariant on $\Sm_U^\fp$ for any quasi-compact open $U\subset S$, by \cite[Theorem 6.3.1]{CHN} and Nisnevich descent.
\end{proof}

\bibliographystyle{alphamod}
\bibliography{references}
    
\end{document}